\documentclass[a4paper,10pt, oneside]{article}
\UseRawInputEncoding

\usepackage[font=small,labelfont=md,textfont=it]{caption}
\usepackage{floatrow,float}
\usepackage[colorlinks,linkcolor=green,citecolor=red, hyperindex,pagebackref,bookmarks]{hyperref}
\usepackage{booktabs, makecell,multirow} 

\usepackage{amsmath, amsthm,bm,amsmath,amsbsy,amsfonts,mathtools, amssymb}
\usepackage[capitalise, nosort]{cleveref}
\usepackage{syntonly,indentfirst}%,latexsym,mathrsfs}
\usepackage{array}
\usepackage{wasysym}

\usepackage{stmaryrd}
\usepackage{subfigure}
\usepackage{pgf,tikz}
\numberwithin{equation}{section}
\newcommand{\tabcaption}{\def\@captype{table}\caption}
\newtheorem{lem}{Lemma}[section]
\newtheorem{thm}{Theorem}[section]
\newtheorem{rem}{Remark}[section]
\newtheorem{exmp}{Example}[section]

\title{ An  interface/boundary-unfitted eXtended HDG method for  linear elasticity   problems%  involving interface or curved boundaries
  \thanks
  {
    This work was supported in part by  National Natural Science Foundation of
    China (11771312) and Guangdong Basic and Applied Basic Research Foundation (2020B1515310005).
  }
}

\author{
	Yihui Han\thanks{South China Research Center for Applied Mathematics and Interdisciplinary Studies, South China Normal University, Guangzhou 510630, China, Email: yhhan@m.scnu.edu.cn},
	Xiao-Ping Wang\thanks{Department of Mathematics, The Hong Kong University of Science and Technology, Clear Water Bay, Kowloon, Hong Kong, China, Email: mawang@ust.hk},
    Xiaoping Xie \thanks{Corresponding author. School of Mathematics, Sichuan University, Chengdu 610064, China, Email: xpxie@scu.edu.cn}
 }

\date{}

\begin{document}
	\maketitle
\begin{abstract}
	 An  interface/boundary-unfitted eXtended hybridizable discontinuous Galerkin (X-HDG) method of arbitrary order is proposed  for  linear elasticity interface problems  on unfitted meshes  with respect to the  interface and domain boundary. 
		  The    method  uses piecewise polynomials of degrees $k\ (\geq 1)$ and $k-1$ respectively for the  displacement and stress approximations in the interior of elements inside the subdomains  separated by the interface,   and  piecewise  polynomials of degree $k$  for the numerical traces of the displacement on the inter-element boundaries   inside the subdomains and on the interface/boundary of the domain.
	  Optimal error estimates in  $L^2$-norm for the stress and displacement are derived.  Finally, numerical experiments confirm the theoretical results and show that  the method also   applies to the case of crack-tip domain. 
	
	$\bf{Key\ Words}$: eXtended  HDG method,  linear elasticity, interface/boundary-unfitted,  error estimate, crack-tip domain 
	
	{\bf MSC(2010)}:  65N12, 65N30
\end{abstract}

\section{Introduction}
Let $\Omega \subset \mathbb{R}^d\  (d = 2, 3) $ be  a bounded domain  with piecewise smooth boundary  $\partial \Omega=\partial\Omega_{D} \cup \partial\Omega_{N}$, where meas$(\partial\Omega_{D})>0$ and $\partial\Omega_{D} \cap \partial\Omega_{N}=\emptyset$.  The domain $\Omega$  is divided into two subdomains, $\Omega_i \ (i = 1,2)$, by  a piecewise smooth interface $\Gamma$ (cf. Figure \ref{domain} for an example). Consider the linear  elasticity interface problem
\begin{subequations}\label{pb1}
	\begin{align} 
	\mathcal{A} \bm{\sigma}-\bm{\epsilon}(\bm{u}) &=\mathbf{0} \qquad \  \ \text { in } \Omega_1\cup\Omega_2, \\ 
	\nabla \cdot \bm{\sigma} &=\bm{f} \qquad \  \  \text { in } \Omega_1\cup\Omega_2, \\
	\bm{u} &=\bm{g}_{D} \qquad   \text { on } \partial\Omega_{D}, \\ 
	\bm{\sigma} \bm{n} &=\bm{g}_{N} \qquad  \text { on } \partial\Omega_{N}, \label{pb1-d} \\
	\llbracket \bm{u} \rrbracket= \bm{0}, \ \llbracket \bm{\sigma}\bm{n}\rrbracket &= \bm{g}_N^{\Gamma} \qquad   \text { on } \Gamma, \label{pb1-e}
	\end{align}
\end{subequations}
where $\bm{\sigma} : \Omega \rightarrow \mathbb{R}_{sym}^{d \times d}$ denotes the symmetric $d \times d$ stress tensor field, $\bm{u} : \Omega \rightarrow \mathbb{R}^{d}$
the displacement field, $\bm{\epsilon}(\bm{u})=\left(\nabla \bm{u}+(\nabla \bm{u})^{T}\right) / 2$ the strain tensor, and $\mathcal{A} \in
\mathbb{R}_{s y m}^{d \times d}$ the compliance tensor with
\begin{align}\label{def_A}
\mathcal{A} \bm{\sigma}=\frac{1}{2 \mu}\left(\bm{\sigma}-\frac{\lambda}{2 \mu+d \lambda} tr(\bm{\sigma}) \bm{I}\right).
\end{align}
Here $tr(\bm{\sigma})$ denotes the trace of $\bm{\sigma}$,
  $\bm{I}$  the $d \times d$ identity matrix, and  $\lambda$ and $ \mu$   the Lam\'e coefficients with $\lambda|_{\Omega_i} = \lambda_i >0$ and $\mu|_{\Omega_i}=\mu_i>0$.  $\bm{f}$  is the body force,   $\bm{g}_D$ and $\bm{g}_N$ are respectively the surface displacement on $\partial\Omega_D$
and the surface traction on $\partial\Omega_N$, and $\bm{n}$ in \eqref{pb1-d} and  \eqref{pb1-e}  denotes respectively the unit outer normal vector along $\partial\Omega_{N}$ and the unit  normal vector along $\Gamma$ pointing to $\Omega_2$.
The jump of a function $w$ across the interface $\Gamma$ is defined by $ \llbracket  w  \rrbracket = (w|_{\Omega_{1}})|_\Gamma-(w|_{\Omega_{2}})|_\Gamma$. % , and $\bm{n}$ denotes the unit normal vector along $\Gamma$ pointing to $\Omega_2$. 
Elasticity interface problems 
are usually used to describe complicated elasticity structure   characterized by discontinuous or even singular material properties, and have many   applications in  materials science and continuum mechanics \cite{gao2001continuum,gibiansky2000multiphase,Jou1997Microstructural,leo2000microstructural,persson2003effect,sigmund2001design,sutton1995interfaces}.%science and engineering \cite{}. 

\begin{figure}[htp]	
	\centering
	\includegraphics[height = 5.5 cm,width= 6.5 cm]{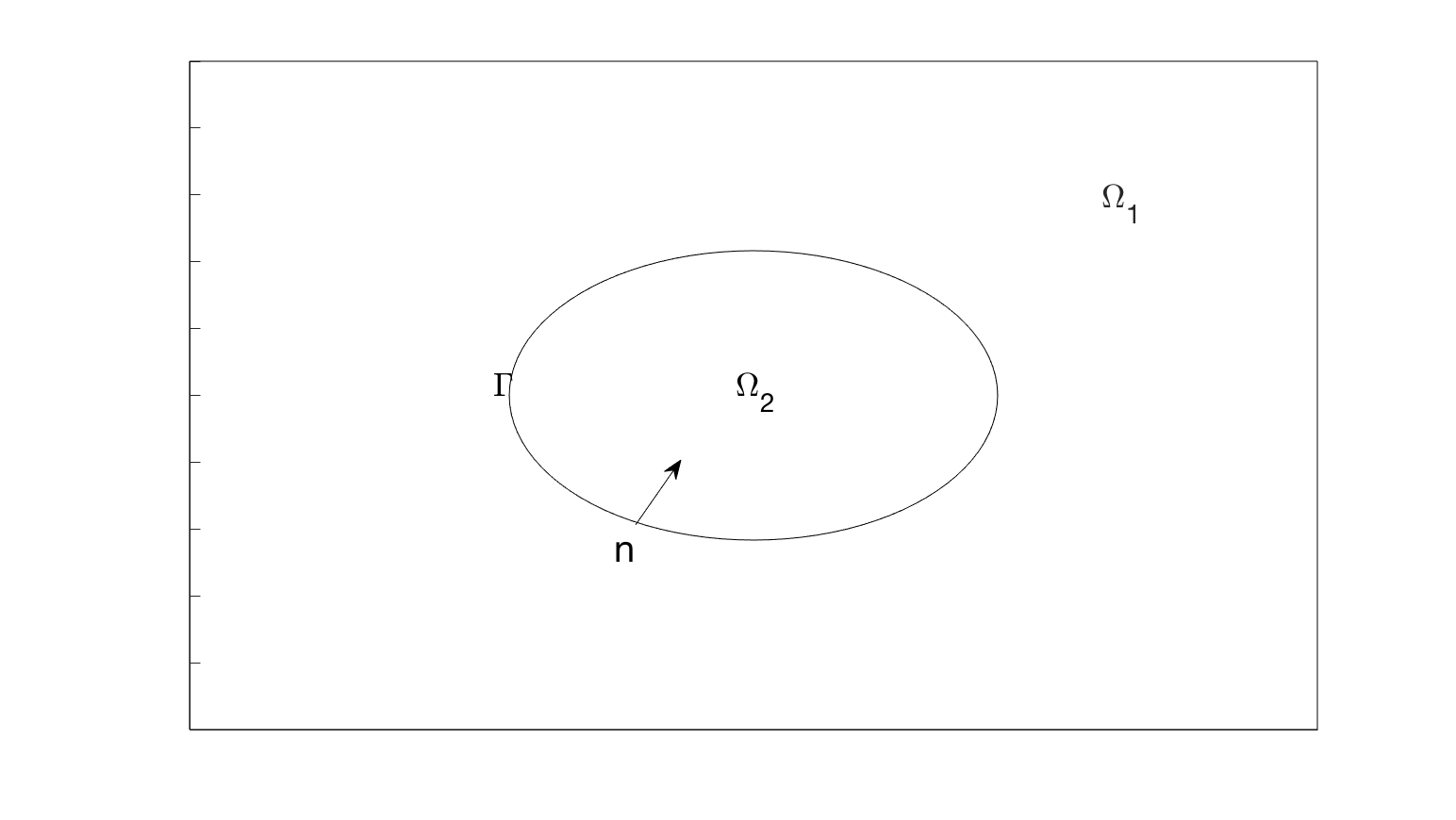} 
	\includegraphics[height = 5.5 cm,width= 6.5 cm]{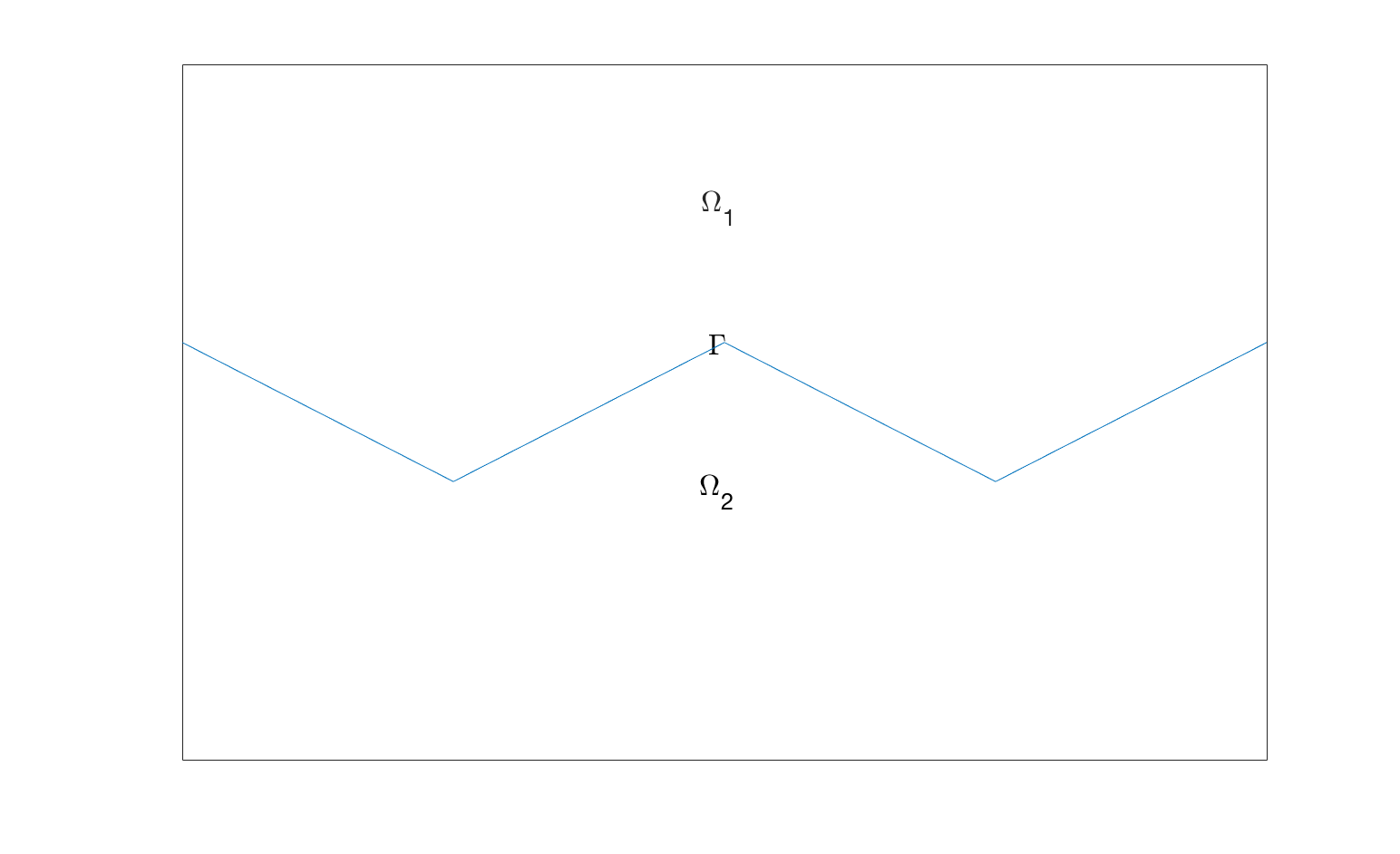} 
	\caption{The geometry of domain with circle interface or fold line interface}\label{domain}
\end{figure}

%involve mathematical models with very com- plicated structure that are characterized by discontinuous or even singular material properties. These problems are known as interface problems. Interface problems arise in various branches of science and engineering such as biological systems and material sciences [41]. Electrostatics play a key role in many biological processes. 
%
%arise in many physical and engineering fields 
%Elasticity problems of multiple phase elastic materials separated by phase interfaces often arise in materials science [1, 22, 49, 51]. They have wide applications in continuum mechanics, particularly for problems that include stresses and strains. 
%such as  Interface problems arise when dealing with physical problems composed of different materials or of the same material at different states. Elasticity interface problems have wide applications in continuum mechanics, particularly for problems that involve stresses and strains
%
% can describe some typical discontinuities phenomena in solid mechanics, which is caused by the sudden change of material properties, represented by bimaterial problemm.  %It is well-known that  the material becomes incompressible as $\lambda\rightarrow \infty$.  
 For  elliptic interface problems,  the global regularity of the solutions   is generally   very low, which may lead to reduced accuracy of finite element discretizations  \cite{babuvska1970finite,xu2013estimate}.  To tackle this situation  there are mainly two types of methods in the literature:  interface-fitted methods and interface-unfitted methods.  The  fitted methods use interface-fitted meshes  to dominate the approximation error caused by the low regularity of solutions  \cite{barrett1987fittedandunfitted,Brambles96fin,Cai2017Discontinuous,Cai2011Discontinuous,Chen1998Finite,Huang2002Some,Li2010Optimal,Plum2003Optimal}; see Figure \ref{unfitted mesh} for an example. However, it is usually expensive to generate interface-fitted meshes, especially when the interface is of complicated geometry or moving with time or iteration. 

The unfitted methods, based on  meshes independent of the interface,    
 employ certain types of  modification in the finite element discretization for approximating functions around the interface so as to avoid the  loss of numerical accuracy. One representative unfitted method is the  eXtended/generalized Finite Element Method \cite{babuska2011Stable,babuvska1994special,belytschko2009review,hansbo2002unfitted,Lehrenfeld2016Optimal,moes1999finite,strouboulis2000design,Thomas2010The,Wang2019},  where additional basis functions characterizing the solution singularity around the interface are adopted for  the corresponding approximation function space. 
For elasticity interface problems,  we refer to \cite{Hansbo2004elasticity} for an XFEM of  displacement-type, and to  \cite{Burman2009A,Hansbo2004elasticity} for   a mixed XFEM  based on a displacement-pressure formulation, both of which use   the cut linear polynomials  around the interface as additional basis functions to enrich  the standard linear element displacement spaces.  
We also refer to \cite{BodartXFEM2018,DuflotThe2008,Li2018A,ShenAn2010,ShenStability2010} for some applications of XFEMs in the  simulation of crack propagation  in  fracture mechanics and  \cite{burman2017cut,burman2018cut,guzman2018inf} for   curved domains.
 
%  \cite{Hansbo2004elasticity} dealed with the elasticity problem, however avoiding the incompressible case. In \cite{Burman2009A}, an eXtended finite element method was proposed  for elasticity problems with discontinuous modulus of elasticity, which can deal with the incompressible case and the interface condition is homogeneous. 
%
%Another important problem is in the region of fracture mechanics. These problems are caused by local geometric mutation in the region, this kind of problem is represented by crack problem. In particular, the simulation of crack propagation is more important and XFEM can avoid the remeshing. \cite{NicolasAXFEMcrack1999} is a representative article gives a XFEM approximation the problem with crack growth.  Some development of XFEM for fracture mechanics can be refered to \cite{DuflotThe2008,ShenAn2010,ShenStability2010,Li2018A,BodartXFEM2018}.

\begin{figure}[H]
	\begin{minipage}[t]{0.5\linewidth}
		\centering
		\includegraphics[height = 5 cm,width= 8.5 cm, clip, trim = 3.5cm 2cm 3cm 1cm]{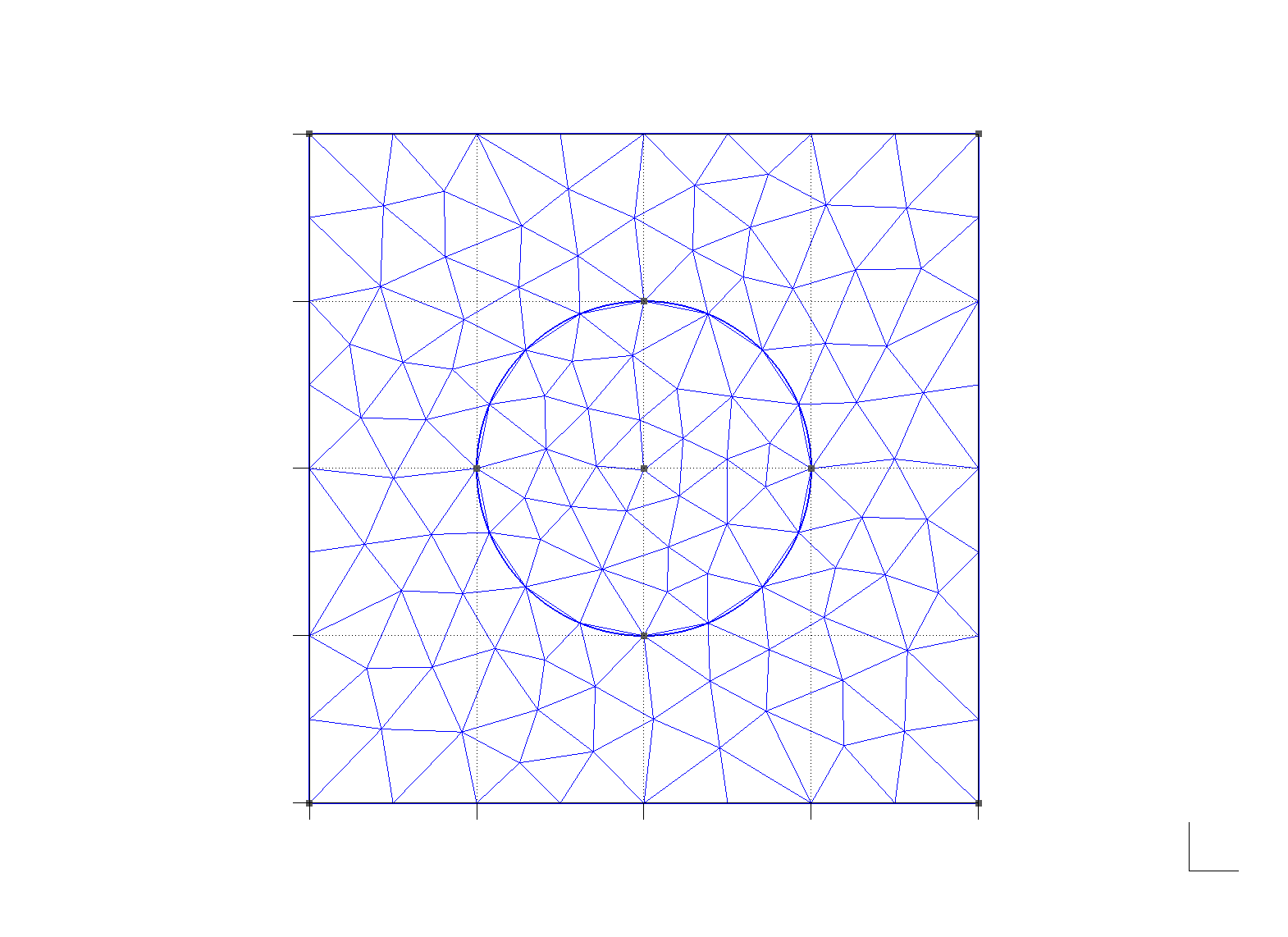}   
	%	\caption{Fitted mesh}
	%	\label{fitted mesh}
	\end{minipage}%
	\begin{minipage}[t]{0.4\linewidth}
		\centering
		\includegraphics[width=2.5in]{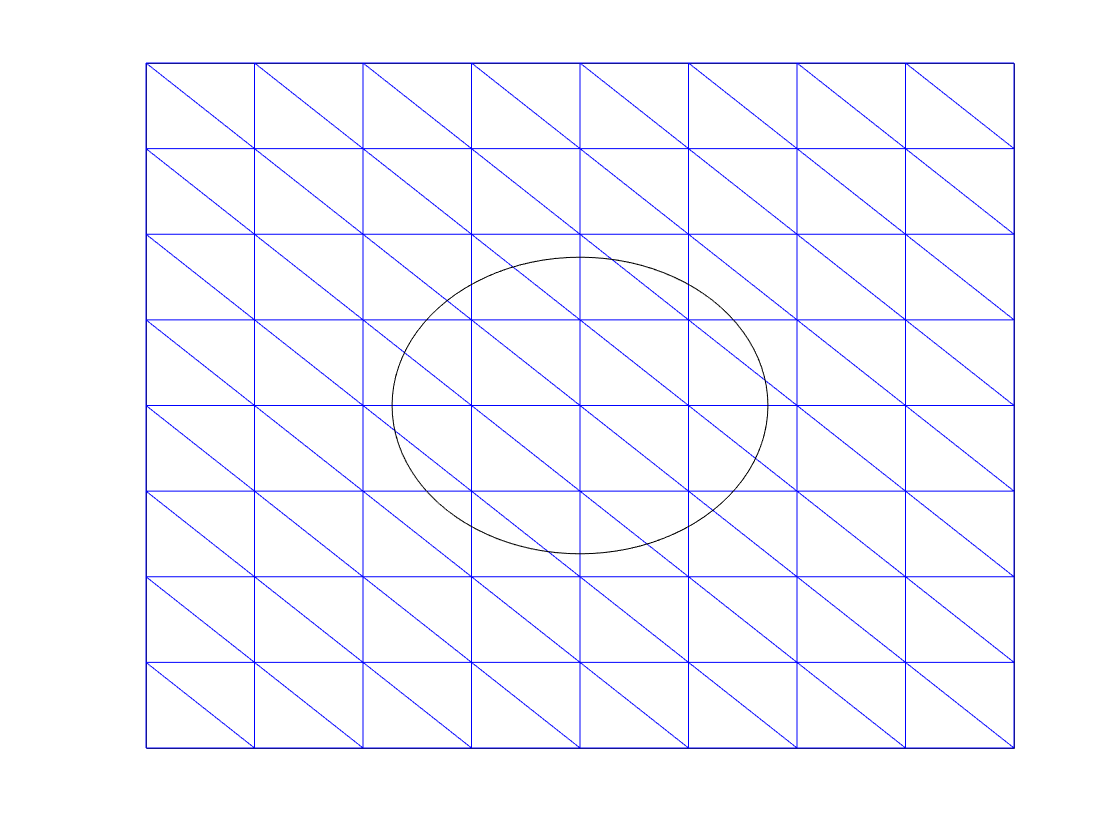}  
		\caption{Fitted mesh (left) and Unfitted mesh (right)}
		\label{unfitted mesh}
	\end{minipage}
\end{figure}

The immersed finite element method (IFEM) is another type of interface-unfitted methods,  where  special finite element basis  functions  are constructed to satisfy the interface jump conditions \cite{Adjerid2015An,GongYan2010Immersed,leveque1994immersed,lizhilin1998immersed,lizhilin2006immersed,lin2015partially, Qin2016A, zhang2004immersed}.  We refer to \cite{LinTao2013A,lintao2012linearifem} for linear/bilinear  immersed finite elements and a nonconforming immersed rectangular  element  for planar elasticity interface problems. %for planar elasticity interface problems
%; see, e.g. \cite{Adjerid2015An,leveque1994immersed,lizhilin1998immersed,lin2015partially, zhang2004immersed,lizhilin2006immersed,LinTao2013A,GongYan2010Immersed,Qin2016A} for some development of IFEM.  

The hybridizable discontinuous Galerkin (HDG) framework  \cite{Cockburn2009Unified}  provides a unifying strategy for hybridization of finite element methods. 
In this framework,   a trace variable  defined on the mesh skeleton is introduced, as a Lagrange multiplier, so as to relax the  continuity constraint of the approximation solution on the inter-element boundaries. Thus, the HDG method   allows for local elimination of unknowns defined in the interior of elements and   leads to a reduction of the number of degrees of freedom in the  final  discrete system.  % the  so as to reduce the size of   discrete systems.
%piecewise-independent approximation to the  solution.  
We refer to \cite{chen2015robust,chen2014robust,Cockburn2011HDGstokes,cockburn2010comparison,cockburn2014divergence, Cockburn2013HDGelasticityweaklystresses,Li-X2016analysis,Li-X2016SIAM,Li-X-Z2016analysis,Qiu2016An}
for some   developments   of   the  HDG method and \cite{cockburn2014priori,cockburn2012solving,cockburn2014solving,2016AQiu,solano2019high} for HDG method to deal with domain with curved boundary.  
In   \cite{Dong2016An} an unfitted HDG method  was developed for   two-dimensional Poisson interface problems by constructing a novel ansatz function in the vicinity of the interface. Based on the   XFEM  philosophy and a level set description of interface,  an equal order eXtended HDG (X-HDG) method   was proposed in  \cite{G2016eXtendedvoid} for diffusion problems with  voids  and  later applied to  heat bimaterial problems \cite{G2016eXtended1}.
%, but no theoretical analysis . {\color{blue} but lack of theoretical analysis for error estimation. }
%  for two types of elliptic interface problems,  voids problems and . 
 In  \cite{HanChenWangXie2019Extended}   two arbitrary order X-HDG methods with optimal convergence rates were presented and analyzed for diffusion interface problems in two and three dimensions. 

 This paper  aims to develop an interface/boundary-unfitted X-HDG method of arbitrary order for the linear elasticity interface problem \eqref{pb1}. The main features of our X-HDG method are as follows.
\begin{itemize}
	\item     The method   uses   piecewise polynomials of degrees $k\ (\geq 1)$ and $k-1$ respectively for the  displacement and stress approximations in the interior of elements inside the subdomains  separated by the interface,   and  piecewise  polynomials of degree $k$  for the numerical traces of the displacement on the inter-element boundaries   inside the subdomains and on the interface/boundary of the domain.  We note that  the unfitted methods in \cite{Burman2009A,Hansbo2004elasticity,LinTao2013A,lintao2012linearifem} are  low order ones, and that the methods in  \cite{chen-xie2016,Qiu2016An} for linear elasticity problems (without interface) use  piecewise polynomials of degrees $k+1$ ($k\geq1$),  $k$ and $k$ respectively for the  displacement, stress  approximations in the interior of elements and  the numerical traces of  displacement  on the inter-element boundaries. 
	
	\item The  method inherits the following  advantages of X-FEM and HDG:   does not require the used meshes   to fit the interface or boundary; allows for local elimination; and   does not require the  stabilization parameters to be ``sufficiently large".
 
 \item 	The derived   error estimates for the  displacement and stress  approximations are optimal. % and uniform with respect to the Lam\'e constant $\lambda$.%_i(i=1,2)$.
	
	\item  The   method applies to any piecewise $C^2$ smooth interface and  any crack-tip domain.
\end{itemize}

The rest of the  paper is organized  as follows. Section 2 introduces  the X-HDG scheme for the elasticity   problem on interface-unfitted meshes and boundary-unfitted meshes, respectively.   Section 3 is devoted to the a priori error estimation for the X-HDG method.  Section 4  provides several numerical examples   to verify the  theoretical results. Finally, Section 5 gives some concluding remarks.

\section{  X-HDG scheme}%{XFE and HDG scheme for problems with interface and problem with piecewise smooth boundary}
\subsection{Notation}
For any bounded polygonal/polyhedral domain $\Lambda\subset \mathbb{R}^s$ $(s= d, d-1)$ and nonnegative integer $m$, let $H^m(\Lambda)$ and $H_0^m(\Lambda)$ be the usual $m$-th order Sobolev spaces on $\Lambda$, with norm $\lVert \cdot\rVert_{m,\Lambda}$ and semi-norm   $\lvert\cdot\rvert_{m,\Lambda}$. In particular, $L^2(\Lambda):=H^0(\Lambda)$ is the space of square integrable functions, with the inner product   $(\cdot,\cdot)_{\Lambda}$.  When $\Lambda \subset  \mathbb{R}^{d-1}$, we use $\langle\cdot,\cdot\rangle_\Lambda$ to replace $(\cdot,\cdot)_\Lambda$.
We set %When $D = \Omega$, we use the notation
\begin{align*}
H^m(\Omega_{1}\cup\Omega_{2}):= \{v\in L^2(\Omega): \ v|_{\Omega_{1}} \in H^m(\Omega_{1}),\  \text{and}\  v|_{\Omega_{2}} \in H^m(\Omega_{2})\},  \\
\lVert\cdot\rVert_m := \lVert\cdot\rVert_{m,\Omega_1\cup\Omega_2} = \sum\limits_{i=1}^{2} \lVert\cdot\rVert_{m,\Omega_i}, \quad
\lvert\cdot\rvert_m :=\lvert\cdot\rvert_{m,\Omega_1\cup\Omega_2}= \sum\limits_{i=1}^{2} \lvert\cdot\rvert_{m,\Omega_i}. 
\end{align*} 
For integer $k\geqslant0$,    $P_k(\Lambda)$   denotes the set of all polynomials on $\Lambda$ with degree no more than $k$. We note that bold face fonts will be used for vector (or tensor) analogues of the Sobolev  spaces along with vector-valued (or tensor-valued) functions. In particular, for the  tensor case we set
$$\bm{L^2}( \Omega,S):=\{\bm{w}\in [L^2(\Omega)]^{d\times d}:\ \bm{w} \text{ is  symmetric} \},$$
$$\bm{H^{m}}(\Omega_1\cup \Omega_2,S):=\bm{L^2}( \Omega,S)\cap [H^m(\Omega_{1}\cup\Omega_{2})]^{d\times d}. $$

Let $\mathcal{T}_h=\cup\{K\}$ be a shape-regular triangulation of the domain $\Omega$ consisting of open triangles/tetrahedrons, which is unfitted with the interface. We define the set of all elements intersected by the interface $\Gamma$ as 
\begin{align*}
\mathcal{T}_h^{\Gamma} :=\{K\in \mathcal{T}_h: K\cap\Gamma \neq \emptyset\}.
\end{align*} 
For any $K\in \mathcal{T}_h^{\Gamma}$ which is called an interface element, let  $\Gamma_K := K\cap \Gamma$ be the part of $ \Gamma$ in $K$, $K_i = K\cap \Omega_{i}$ be the part of $K$ in $\Omega_{i}\ (i = 1,2)$, and   $\Gamma_{K,h}$ be the straight line/plane segment connecting the intersection between $\Gamma_K$ and $\partial K$. 
To ensure that $\Gamma$ is reasonably resolved by $\mathcal{T}_h$, we make the following standard assumptions on $\mathcal{T}_h$ and interface $\Gamma$: 

\noindent  {\bf(A1)}.  For $K\in \mathcal{T}_h^{\Gamma}$ and any edge/face $F\subset \partial K$ which intersects $\Gamma$,  $F_\Gamma:=\Gamma\cap F$ is simply connected with either   $F_\Gamma =F$ or $meas(F_\Gamma)=0$. 

\noindent 	{\bf(A2)}. For $K\in \mathcal{T}_h^{\Gamma}$, there is a smooth function $\psi$ which maps $\Gamma_{K,h}$ onto $\Gamma_K$.

\noindent 	{\bf(A3)}. For any two different points $\bm{x},\bm{y}\in \Gamma_{K}$, the unit normal vectors $\bm{n}(\bm{x})$ and $\bm{n}(\bm{y})$, pointing to $\Omega_{2}$, at $\bm{x}$ and $\bm{y}$ satisfy
\begin{align}\label{gamma}
\lvert \bm{n}(\bm{x})-\bm{n}(\bm{y})\rvert \leq \gamma h_K,
\end{align}
with $\gamma\geq 0$ (cf.\cite{Chen1998Finite,xu2013estimate}). Note that   $\gamma = 0$ when $\Gamma_{K}$ is a straight line/plane segment.

Let  $\mathcal{\varepsilon}_h$ be the set of all edges (faces) of all elements in $\mathcal{T}_h$ and   $\mathcal{\varepsilon}_h^{\Gamma}$ be  the partition of $\Gamma$ with respect to $\mathcal{T}_h$, i.e.
$$\mathcal{\varepsilon}_h^{\Gamma} := \{ F:\ F=\Gamma_{K}, \text{ or } F= \Gamma\cap \partial K \text{ if $\Gamma\cap \partial K$ is an edge/face of $K$},\forall   K\in \mathcal{T}_h\},$$
and  set 
$\mathcal{\varepsilon}_h^* :=\mathcal{\varepsilon}_h\setminus\mathcal{\varepsilon}_h^{\Gamma} $. For any $K\in \mathcal{T}_h$ and   $F\in \mathcal{\varepsilon}_h^*\cup \mathcal{\varepsilon}_h^{\Gamma},$    $h_K$  and $h_F$ denote respectively   the diameters of $K$ and   $F$, and $\bm{n}_K$ denotes the unit outward  normal vector along   $\partial K$. We denote by 
$h: = \max\limits_{K\in \mathcal{T}_h}h_K$  the mesh size of $\mathcal{T}_h$, and by   $\nabla_h$, $\nabla_h\cdot$ and $\bm{\epsilon}_h$ the piecewise-defined gradient, divergence and strain operators with respect to $\mathcal{T}_h$, respectively.  

Throughout the paper, we use $a\apprle b$ $(a\apprge b)$ to denote $a\leq Cb$ $(a\geq Cb)$, %and  $a\AC b$ to denote $a\apprle b \apprle a$, 
where   $C$ is a generic positive constant   independent of mesh parameters $h, h_K, h_e$,  % the coefficients $\mu_i, \lambda_i$ $(i =1,2)$ 
and the  location of the interface relative to the meshes, and may be different at its each occurrence.

\subsection{X-HDG scheme on interface-unfitted meshes}
%For  $i=1,2$, let $\chi_i$ be the characteristic function on $\Omega_{i}$, and for any  $K\in \mathcal{T}_h$, $F\in \mathcal{\varepsilon}_h^*\cup \mathcal{\varepsilon}_h^{\Gamma}$ and   integer $r\geq 0$,  
%let $Q_{r}^b: L^2(D)\rightarrow P_r(D)$ be the standard $L^2$ orthogonal projection operator with 
%$D=F\cap\bar{\Omega}_i$. And Let $Q_{r}: L^2(D)\rightarrow P_r(D)$ be the standard $L^2$ orthogonal projection operator with $D=K\cap\Omega_i$ for     $i=1,2$.  Vector or tensor analogues of $Q_{r}^b$ and $Q_{r}$ are denoted by $\bm{Q}_{r}^b$ and $\bm{Q}_{r}$, respectively.

For  $i=1,2$, let $\chi_i$ be the characteristic function on $\Omega_{i}$, and for  integer $r\geq 0$,  
let $\Pi_{r}^b: L^2(\Lambda)\rightarrow P_r(\Lambda)$ be the standard $L^2$ orthogonal projection operator for any  bounded domain
$\Lambda$.  Vector or tensor analogues of $\Pi_{r}^b$ are denoted by $\bm{\Pi}_{r}^b$ , respectively.
Set 
$$ \oplus\chi_iP_{r}(K) := \chi_1P_{r}(K)+\chi_2P_{r}(K),\quad  \bm{L^2}(\Omega,S) :=\{\bm{w}\in [L^2(\Omega)]^{d\times d}: \bm{w}^T = \bm{w} \}. $$

Let us  introduce  the following X-HDG finite element spaces:
\begin{align*}
\bm{W}_h =& \{\bm{w}\in \bm{L^2}(\Omega,S):  \forall K\in \mathcal{T}_h, \bm{w}|_K \in \bm{P}_{k-1}(K) \ \text{if} \ K\cap \Gamma = \varnothing; \bm{w}|_K \in \oplus\chi_i\bm{P}_{k-1}(K) \ \text{if}\  K\cap \Gamma \neq \varnothing\}, \\
\bm{V}_h =& \{\bm{v}\in \bm{L^2}(\Omega):  \forall K\in \mathcal{T}_h,\bm{v}|_K \in  \bm{P}_{k}(K) \ \text{if} \  K\cap \Gamma = \varnothing; \bm{v}|_K \in \oplus\chi_i  \bm{P}_{k}(K) \ \text{if}\ K\cap \Gamma \neq \varnothing\}, \\
\bm{M}_h =& \{ \bm{\hat{\mu}}\in \bm{L}^2(\varepsilon_h^*): \forall F\in \varepsilon_h^*, \bm{\hat{\mu}}|_F \in \bm{P}_{k}(F) \ \text{if}\   F\cap \Gamma = \varnothing;  \bm{\hat{\mu}}|_F \in \oplus\chi_i \bm{P}_{k}(F) \ \text{if}\   F\cap \Gamma \neq \varnothing\}, \\
\bm{\tilde{M}}_h = &\{\bm{\tilde{\mu}} \in \bm{L^2}(F): 
\bm{\tilde{\mu}}|_{F}\in \bm{P}_{k}(K)|_F,\forall F\in\mathcal{\varepsilon}_h^{\Gamma}\},\\
\bm{M}_h(\bm{g}_D) =& \{\bm{\hat{\mu}}\in \bm{M}_h :   \bm{\hat{\mu}}|_{F\cap\bar\Omega_i} = \bm{\Pi}_{k}^b(\bm{g}_D|_{F\cap\bar\Omega_i}),\   \forall F\in \varepsilon_h^* \text{ with }F\subset \partial\Omega_D \}.
\end{align*}
% $w_i = w|_{\Omega_i}(i=1,2)$.
Then the X-HDG method is given as follows: seek $(\bm{\sigma}_h,\bm{u}_h,\bm{\hat{u}}_h,\bm{\tilde{u}}_h)\in \bm{W}_h\times \bm{V}_h\times \bm{M}_h(\bm{g}_D)\times \bm{\tilde{M}}_h$ such that
\begin{subequations}\label{xhdgscheme}
	\begin{align}
	(\mathcal{A} \bm{\sigma}_h,\bm{w})_{\mathcal{T}_h} + (\bm{u}_h,\nabla_h\cdot \bm{w})_{\mathcal{T}_h} - \langle \hat{\bm{u}}_h,\bm{w}\bm{n}\rangle_{\partial\mathcal{T}_h\setminus \mathcal{\varepsilon}_h^{\Gamma}} -\langle \tilde{\bm{u}}_h,\bm{w}\bm{n}\rangle_{*,\Gamma} =& 0, \label{xhdg1}\\
	(\nabla_h\cdot\bm{\sigma}_h,\bm{v})_{\mathcal{T}_h}- \langle \tau(\bm{u}_h-\bm{\hat{u}}_h),\bm{v}\rangle_{\partial\mathcal{T}_h\setminus \mathcal{\varepsilon}_h^{\Gamma}} 
	-  \langle \eta(\bm{u}_h-\bm{\tilde{u}}_h),\bm{v}\rangle_{*,\Gamma}  =& (\bm{f},\bm{v}) ,\label{xhdg2}\\
	\langle \bm{\sigma}_h \bm{n},\bm{\hat{\mu}}\rangle_{\partial \mathcal{T}_h\setminus \mathcal{\varepsilon}_h^{\Gamma}} -\langle \tau(\bm{u}_h-\bm{\hat{u}}_h),\bm{\hat{\mu}}\rangle_{\partial \mathcal{T}_h\setminus\mathcal{\varepsilon}_h^{\Gamma}} =&\langle \bm{g}_N,\bm{\hat{\mu}}\rangle_{\partial \mathcal{T}_h\setminus \mathcal{\varepsilon}_h^{\Gamma}} ,\label{xhdg3}\\
	\langle \bm{\sigma}_h \bm{n},\bm{\tilde{\mu}}\rangle_{*,\Gamma} - \langle \eta(\bm{u}_h-\bm{\tilde{u}}_h),\bm{\tilde{\mu}}\rangle_{*,\Gamma}=& \langle \bm{g}_N^{\Gamma},\bm{\tilde{\mu}}\rangle_{*,\Gamma} \label{xhdg4}
	\end{align}
\end{subequations}
for all $(\bm{w},\bm{v},\bm{\hat{\mu}},\bm{\tilde{\mu}})\in \bm{W}_h\times \bm{V}_h\times  \bm{M}_h(\bm{0})\times \bm{\tilde{M}}_h$.
Here   %introduce the following notations:% $e\in \varepsilon_h^{\Gamma}$ by
\begin{align*}
(\cdot,\cdot)_{\mathcal{T}_h}:=\sum\limits_{K\in  \mathcal{T}_h}(\cdot,\cdot)_K,\quad  \langle\cdot,\cdot\rangle_{\partial\mathcal{T}_h\setminus \mathcal{\varepsilon}_h^{\Gamma}}:=\sum\limits_{K\in  \mathcal{T}_h}\langle\cdot,\cdot\rangle_{\partial K\setminus \mathcal{\varepsilon}_h^{\Gamma}},
\end{align*}
and, for  vectors $\bm{\mu},\bm{v}$ and tensor $\bm{w}$ with $\bm{\mu}_i=\bm{\mu}|_{F\cap\bar{\Omega}_i}, \bm{v}_i=\bm{v}|_{F\cap\bar{\Omega}_i}$ and $\bm{w}_i=\bm{w}|_{F\cap\bar{\Omega}_i}$, 
\begin{align*}
\langle \bm{\mu},\bm{v}\rangle_{*,\Gamma} :&=\sum\limits_{F\in \varepsilon_h^{\Gamma}}\int_{F} (\bm{\mu}_1\cdot\bm{v}_1+ \bm{\mu}_2\cdot\bm{v}_2)ds,\\
%\sum\limits_{e\in \varepsilon_h^{\Gamma}}2\int_{e} \{wv \}ds, \quad
\langle \bm{w}\bm{n},\bm{v}\rangle_{*,\Gamma} :&=\sum\limits_{F\in \varepsilon_h^{\Gamma}}\int_{F} ((\bm{w}_1\bm{n}_1)\cdot\bm{v}_1+(\bm{w}_2\bm{n}_2)\cdot\bm{v}_2)ds,%\quad \lVert w\rVert_{*,\Gamma} := (\sum\limits_{e\in \varepsilon_h^{\Gamma}}2\int_{e} \{w^2\}ds)^{\frac{1}{2}}.
\end{align*}
where $\bm{n}_i$ denotes  the unit normal vector along $\Gamma$ pointing from   $\Omega_{i}$ to $\Omega_{j}$ with $i,j=1,2$ and $i\neq j$. 
The stabilization functions $\tau$ and $\eta$ are defined as below:  for  $F  \in \varepsilon_h$, $K\in \mathcal{T}_h $ %, \ F  \in \varepsilon_h^{\Gamma}$ 
and $i=1,2$,
\begin{subequations}
\begin{align}
\tau|_{F\cap\Omega_{i}} &= 2\mu_i h_K^{-1}, \label{stablizationpara1} \quad {\rm if} \ F\subset \partial K\setminus \mathcal{\varepsilon}_h^{\Gamma} \ {\rm and} \ F\cap \Omega_{i}\neq\emptyset ,\\
\eta|_{F\cap \bar{ \Omega}_i}&=2\mu_i h_{K}^{-1},  \quad   \text{ if } F=\Gamma_K \text{ or } F\subset \partial (K \cap \Omega_i). 
\end{align}
\end{subequations}

\begin{rem}\label{rmk2.1}
	In fact, if the homogeneous interface condition $\llbracket \bm{u}\rrbracket= \bm 0$ in the model problem \eqref{pb1} is generalized as $$\llbracket \bm{u}\rrbracket= \bm{g}_D^{\Gamma}\neq \bm{0}$$ and the interface $\Gamma$ is a piecewise straight segment/polygon, then we can introduce   the space 
%	$$\bm{\tilde{M}}_h = \{\bm{\tilde{\mu}} = \{\bm{\tilde{\mu}}_1,\bm{\tilde{\mu}}_2\}: 
%	\bm{\tilde{\mu}}_i|_{F}\in \bm{P}_{k}(F),\forall F\in\mathcal{\varepsilon}_h^{\Gamma}\},$$
%	and introduce the space
	$$\bm{\tilde{M}}_h^*(\bm{g}_D^{\Gamma}) = \{\bm{\tilde{\mu}} = \{\bm{\tilde{\mu}}_1,\bm{\tilde{\mu}}_2\} : \bm{\tilde{\mu}}_i|_{F}\in \bm{P}_{k}(F),\ \llbracket \bm{\tilde{\mu}} \rrbracket |_F=\bm{\tilde{\mu}}_1|_F-\bm{\tilde{\mu}}_2|_F= \bm{\Pi}_k^b(\bm{g}_D^{\Gamma}|_F),
	\forall F\in\mathcal{\varepsilon}_h^{\Gamma}\}$$
	so as to obtain the corresponding 
	  X-HDG scheme:  seek  $(\bm{\sigma}_h,\bm{u}_h,\bm{\hat{u}}_h,\bm{\tilde{u}}_h)\in \bm{W}_h\times \bm{V}_h\times \bm{M}_h(\bm{g}_D)\times \bm{\tilde{M}}_h^*(\bm{g}_D^{\Gamma})$ such that the equations \eqref{xhdgscheme} hold for  $(\bm{w},\bm{v},\bm{\hat{\mu}},\bm{\tilde{\mu}})\in \bm{W}_h\times \bm{V}_h\times  \bm{M}_h(\bm{0})\times \bm{\tilde{M}}_h^*(\bm{0})$.  We note that all the analyses hereafter also apply to this case.

\end{rem}

\begin{thm}\label{wellposed}
	For $k\geq 1$, the X-HDG scheme \eqref{xhdgscheme} admits a unique solution $(\bm{\sigma}_h,\bm{u}_h,\bm{\hat{u}}_h,\bm{\tilde{u}}_h)\in \bm{W}_h\times \bm{V}_h\times \bm{M}_h(\bm{g}_D)\times \bm{\tilde{M}}_h$.
\end{thm}
\begin{proof}
	Since the \eqref{xhdgscheme} is a linear square system, it suffices to show that if all of the given data vanish, i.e. $\bm{f} = \bm{g}_D  = \bm{g}_N = \bm{g}_N^{\Gamma} = \bm{0}$, then we get the zero solution. Taking  $(\bm{w},\bm{v},\bm{\mu},\bm{\tilde{\mu}}) = (\bm{\sigma}_h,\bm{u}_h,\bm{\hat{u}_h},\bm{\tilde{u}_h})$ in \eqref{xhdgscheme} and adding these equations together, we have
	\begin{align}
	(\mathcal{A} \bm{\sigma}_h,\bm{\sigma}_h)_{\mathcal{T}_h} +  \langle \tau(\bm{u}_h-\bm{\hat{u}}_h),\bm{u}_h-\bm{\hat{u}}_h\rangle_{\partial\mathcal{T}_h\setminus \mathcal{\varepsilon}_h^{\Gamma}} 
	+  \langle \eta(\bm{u}_h-\bm{\tilde{u}}_h),\bm{u}_h-\bm{\tilde{u}}_h\rangle_{*,\Gamma} = 0,
	\end{align}
which, together with the relation
$$	(\mathcal{A} \bm{w}, \bm{w})_{\mathcal{T}_h}  =\left(\frac{1}{2 \mu}(\bm{w}-\frac{1}{d} tr(\bm{w}) I), \bm{w}-\frac{1}{d} tr(\bm{w}) I\right)_{\mathcal{T}_h}+\left(\frac{1}{d(d \lambda+2 \mu)}tr(\bm{w}), tr(\bm{w})\right)_{\mathcal{T}_h}, \ \forall 	 \bm{w}\in \bm{W}_h, $$
shows that
	\begin{align*}
	\bm{\sigma}_h &= \bm{0},  \quad  \text{ in } \mathcal{T}_h ,\\
	\bm{u}_h-\bm{\hat{u}}_h &= \bm{0}, \quad \text{ on} \ \partial\mathcal{T}_h\setminus \mathcal{\varepsilon}_h^{\Gamma}, \\
	\{\bm{u}_h-\bm{\tilde{u}}_h\} &= \bm{0},  \quad \text{ on}\  \Gamma.
	\end{align*}
	Here $\{\cdot\} $ is defined by  $\{\bm{v}\} = \frac{1}{2}(\bm{v}_1+\bm{v}_2)$ with $\bm{v}_i=\bm{v}|_{\Gamma\cap\bar{\Omega}_i} $ for $i=1,2$. These  relations, plus
	 \eqref{xhdg1} and integration by parts, yield
	\begin{align*}
	(\mathcal{A}\bm{\sigma}_h,\bm{w})_{\mathcal{T}_h} +(\bm{\epsilon}_h (\bm{u}_h), \bm{w})_{\mathcal{T}_h} - \langle \bm{u}_h-\bm{\hat{u}}_h,\bm{w} \bm{n}\rangle_{\partial\mathcal{T}_h\setminus \mathcal{\varepsilon}_h^{\Gamma}} -\langle \bm{u}_h-\bm{\tilde{u}}_h,\bm{w}\bm{n}\rangle_{*,\Gamma} = 0.
	\end{align*}
	Taking   $\bm{w}_h =\bm{\epsilon}_h (\bm{u}_h)$ in this relation leads to $\bm{\epsilon}_h (\bm{u}_h) = \bm{0}$. In view of  $\bm{\hat{u}}_h = \bm{0}$ on $\partial \Omega_D$,  we get $\bm{u}_h = \{\bm{\tilde{u}}_h\} = \bm{0}$. Thus,  from $\llbracket \bm{\tilde{u}}_h\rrbracket = \bm{0}$ it follows  
	\begin{align*}
	\bm{\sigma}_h =\bm{0} \quad \text{ and } \quad  \bm{u}_h =  \bm{\hat{u}}_h  = \bm{\tilde{u}}_h  =\bm{0}.
	\end{align*}
This completes the proof.
\end{proof}

\subsection{X-HDG scheme on boundary-unfitted meshes}
In this subsection, we shall extend the X-HDG method in Section 2.1 to the case using boundary-unfitted meshes (cf. Figure \ref{domain1} for an example). For simplicity,    we consider the following linear elasticity problem:  %find the stress $\bm{\sigma}$ and displacement $\bm{u}$ such that
\begin{subequations}\label{pb2}
	\begin{align} 
	\mathcal{A} \bm{\sigma}-\bm{\epsilon}_h(\bm{u}) &=\mathbf{0}, \ \  \text { in } \Omega, \\ 
	\nabla \cdot \bm{\sigma} &=\bm{f}, \ \  \text { in } \Omega ,\\
	\bm{u} &=\bm{g}_{D}, \text { on } \partial\Omega_{D}, \\ 
	\bm{\sigma} \bm{n} &=\bm{g}_{N}, \text { on } \partial\Omega_{N}.
	\end{align}
\end{subequations}
%Here the domain $\Omega\subset \mathbb{R}^d (d = 2, 3) $ has piecewise smooth boundary(Fig. \ref{domain1}) $\partial \Omega=\partial\Omega_{D} \cup \partial\Omega_{N}$, where meas($\partial\Omega_{D}$)>0 and $\partial\Omega_{D} \cap \partial\Omega_{N}=\emptyset$ .  Similarly, $\bm{\sigma} : \Omega \rightarrow \mathbb{R}_{sym}^{d \times d}$ denotes the symmetric $d \times d$ stress tensor field, $\bm{u} : \Omega \rightarrow \mathbb{R}^{d}$
%the displacement field, $\bm{\epsilon}_h(\bm{u})=\left(\nabla \bm{u}+(\nabla \bm{u})^{T}\right) / 2$ the strain tensor, and $\mathcal{A} \in
%\mathbb{R}_{s y m}^{d \times d}$ the compliance tensor with
%\begin{align}
%\mathcal{A} \bm{\sigma}=\frac{1}{2 \mu}\left(\bm{\sigma}-\frac{\lambda}{2 \mu+d \lambda} tr(\bm{\sigma}) \bm{I}\right),
%\end{align}
%where $\lambda, \mu$ are the Lam\'e coefficients. $tr(\bm{\sigma})$ denotes the trace of $\bm{\sigma}$
%and $\bm{I}$ is the $d \times d$ identity matrix. $\bm{f}$ is the body force, and $\bm{g}_D$ and $\bm{g}_N$ the surface displacement on $\partial\Omega_D$
%and the surface traction on $\partial\Omega_N$, respectively.

%https://cdn.mathpix.com/snip/images/XgP160zm72S1xw087QQOpYqQKKUWR-jdv2hXFUewcNY.original.fullsize.png
%![](https://cdn.mathpix.com/snip/images/XgP160zm72S1xw087QQOpYqQKKUWR-jdv2hXFUewcNY.original.fullsize.png)
%
%<img> tag: <img src="https://cdn.mathpix.com/snip/images/XgP160zm72S1xw087QQOpYqQKKUWR-jdv2hXFUewcNY.original.fullsize.png" />

\begin{figure}[htp]	
	\centering
	\includegraphics[height = 5.5 cm,width= 6.5 cm]{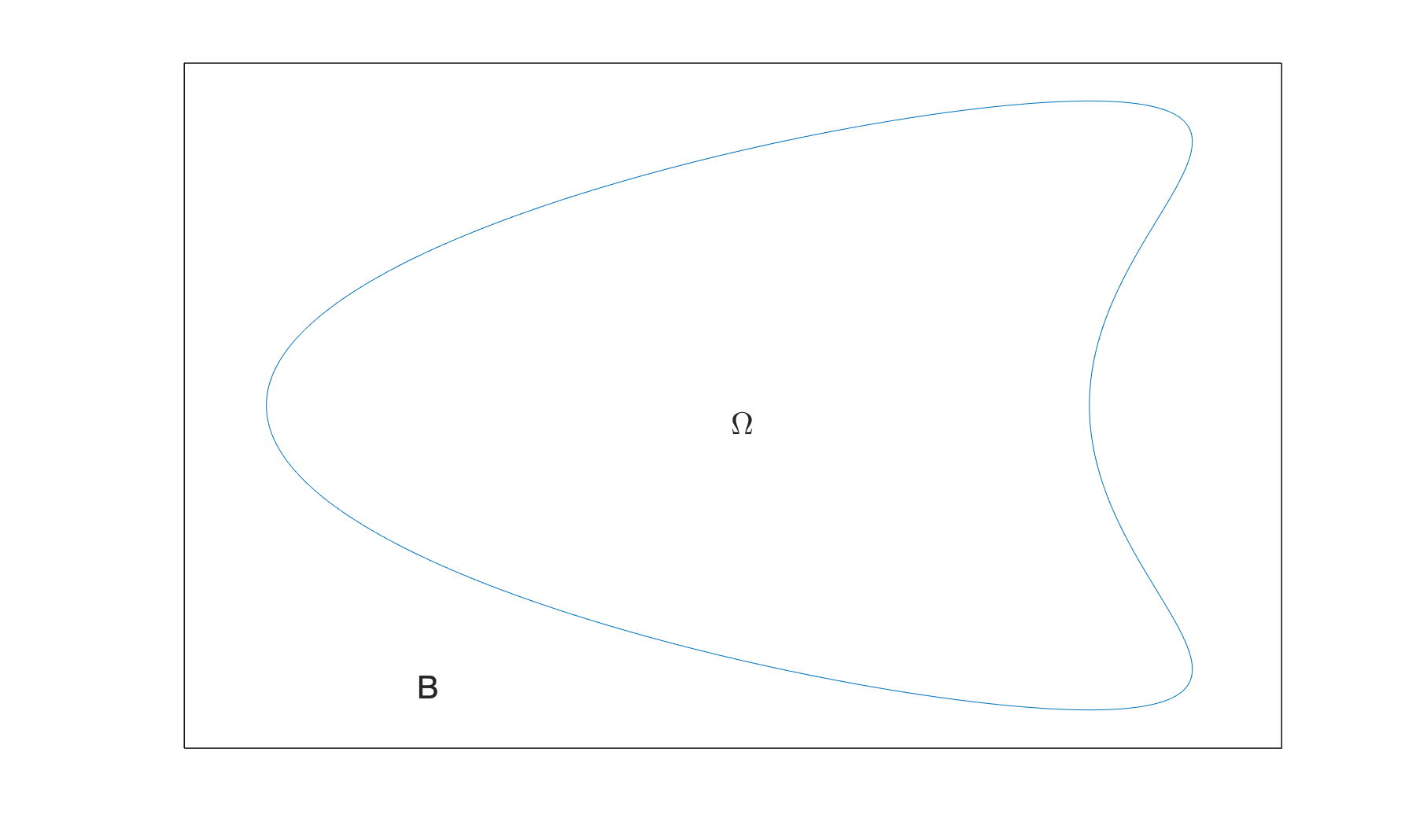} 
	\includegraphics[height = 5.5 cm,width= 6.5 cm]{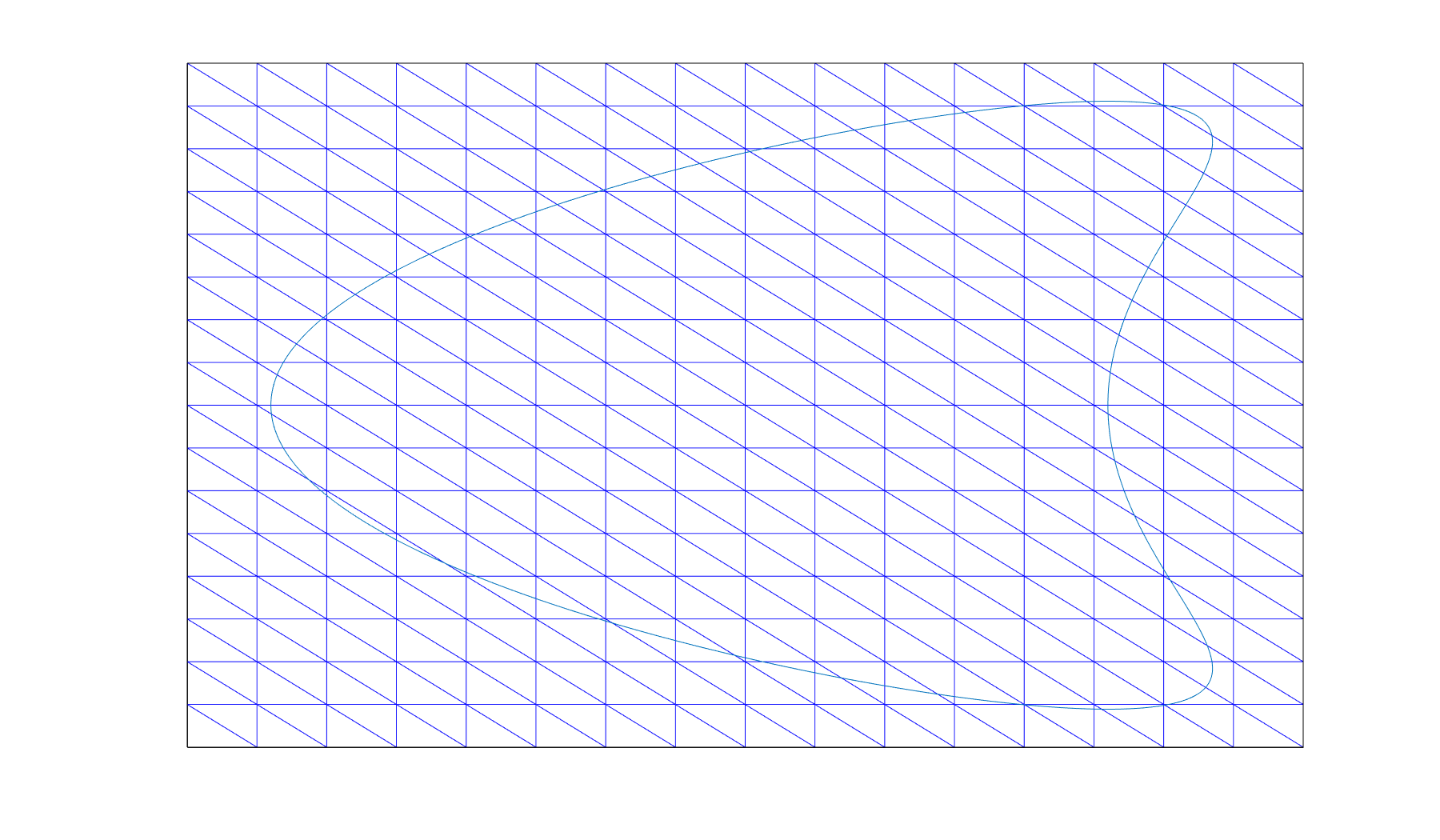} 
	\caption{The geometry of domain with piecewise smooth boundary(left) and boundary unfitted mesh(right)}\label{domain1}
\end{figure}
%{\color{blue}
Let $\mathbb{B}\supset \Omega$ be a simpler  domain than $\Omega$ (Figure \ref{domain1}),  denote $\Omega^{c} := \mathbb{B}\backslash \bar{\Omega}$, and introduce a vector function $\bm{\tilde{f}}$ defined on $\mathbb{B}$ with $\bm{\tilde{f}}|_{\Omega} = \bm{f}$ and $\bm{\tilde{f}}|_{\Omega^c} = \bm{0}$.  Then we can rewrite  problem \eqref{pb2} as an interface problem:  %find the stress $\bm{\sigma}$ and displacement $\bm{u}$ such that
\begin{subequations}\label{pb3}
	\begin{align} 
	\mathcal{A} \bm{\sigma}-\bm{\epsilon}(\bm{u}) &=\mathbf{0}, \text { in } \Omega\cup\Omega^{c} , \\ 
	\nabla \cdot \bm{\sigma} &=\bm{\tilde{f}}, \text { in } \Omega \cup\Omega^{c},\\
	\bm{u} \equiv \bm{0}, \ \bm{\sigma}  &\equiv \bm{0},   \text { in } \Omega^{c},\\
	\llbracket \bm{u} \rrbracket&=\bm{g}_{D}, \text { on } \partial\Omega_{D}, \\ 
	\llbracket \bm{\sigma} \bm{n} \rrbracket &=\bm{g}_{N}, \text { on } \partial\Omega_{N}.
	\end{align}
\end{subequations}
We note that the problem \eqref{pb3} is a special interface problem with $\partial \Omega$ being the interface, for which we 
% the jump of $\bm{u}$ and $\bm{\sigma}$ play a role on $\partial \Omega_D$ and $\partial \Omega_N$ respectively, which is different from the problem \eqref{pb1}.  
only need to approximate the solution in $\Omega$, since the solution in $\Omega^{c}$ is   zero.
%}

Let $\mathcal{T}_h=\cup\{K\}$ be a shape-regular triangulation of the domain $\mathbb{B}$ consisting of open triangles/tetrahedrons. 
Define the following sets of elements and   edges/faces:
\begin{align*}
\mathcal{T}_h^{i} :=&\{K\in \mathcal{T}_h: K\cap\Omega = K\}, \\
%\mathcal{T}_h^{\Gamma} :=&\{K\in \mathcal{T}_h: K\cap\partial\Omega \neq \emptyset\}, \\
\mathcal{T}_h^{\Gamma} :=&\{K_{\Gamma}: K_{\Gamma} = K\cap \Omega, K\in \mathcal{T}_h \ \text{and}\ K\cap\partial\Omega \neq \emptyset\}, \\
\mathcal{T}_h^{*} :=& \mathcal{T}_h^{i}\cup \mathcal{T}_h^{\Gamma},\\
%\end{align*} 
%%and the set of edge/face as
%\begin{align*}
\mathcal{\varepsilon}_h^{i} :=&\{F:\ F \text{ is a edge/face of element in}\  \mathcal{T}_h^i \ \text{and} \ F\cap \partial\Omega = \emptyset\}, \\
\mathcal{\varepsilon}_h^{\Gamma} :=&\{F:\ F \text{ is a edge/face of element in}\  \mathcal{T}_h^{\Gamma} \ \text{and} \ F\cap \partial\Omega = \emptyset\}, \\
\mathcal{\varepsilon}_h^{\partial} :=& \{F:\ F = K\cap\partial\Omega, \forall K\in \mathcal{T}_h^{\Gamma} \ \text{or} \ F \ \text{is a edge/face of K}, \forall K\in \mathcal{T}_h^{i} \ \text{and} \ \bar{K}\cap \partial\Omega\neq \emptyset\},\\
%\mathcal{\varepsilon}_h^{N} :=& \{F:\ F = K\cap\partial\Omega, \forall K\in \mathcal{T}_h^{\Gamma} \ \text{or} \ F \ \text{is a edge/face of K}, \forall K\in \mathcal{T}_h^{i} \ \text{and} \ \bar{K}\cap \partial\Omega\neq \emptyset\},\\
%\mathcal{\varepsilon}_h^{\partial } :=&  \mathcal{\varepsilon}_h^{D} \cup \mathcal{\varepsilon}_h^{N},\\
\mathcal{\varepsilon}_h :=& \mathcal{\varepsilon}_h^{i}\cup \mathcal{\varepsilon}_h^{\Gamma} \cup \mathcal{\varepsilon}_h^{\partial}.
\end{align*} 
Introduce  the following X-HDG finite element spaces:
\begin{align*}
\bm{W}_h :=& \{\bm{w}\in \bm{L^2}( \Omega,S): \bm{w}|_K \in \bm{P}_{k-1}(K) \ {\rm if} \  K\in   \mathcal{T}_h^{*}\}, \\
\bm{V}_h :=& \{\bm{v}\in \bm{L^2}(\Omega): \bm{v}|_K \in  \bm{P}_{k}(K) \ {\rm if} \  K\in   \mathcal{T}_h^{*}\}, \\
\bm{M}_h^{i} :=& \{\bm{\mu}\in \bm{L^2}(\varepsilon_h\backslash\varepsilon_h^{\partial}): \ \bm{\mu}|_F \in  \bm{P}_{k}(F) \  {\rm if} \  F\in \varepsilon_h\backslash\varepsilon_h^{\partial}\},\\
\bm{M}_h^{\partial}: =& \{\tilde{\bm{\mu}}\in \bm{L^2}(\varepsilon_h^{\partial}) : \  \tilde{\bm{\mu}}|_F \in \bm{P}_{k}(K)|_F,\  \forall F\in \varepsilon_h^{\partial}, \ \text{ and for some } K\in   \mathcal{T}_h^{*}\},\\
\bm{M}_h^{\partial}(\bm{g}_D): =& \{\bm{\tilde{\mu}}\in \bm{L^2}(\varepsilon_h^{\partial}) : \  \langle \bm{\tilde{\mu}}, \bm{\mu^*}\rangle_{F}=\langle \bm{g}_D, \bm{\mu^*}\rangle_{F},\  \forall F\in \varepsilon_h^{\partial}\cap\partial\Omega_D, \ {\rm and} \  \bm{\tilde{\mu}}|_F, \bm{\mu^{*}}|_F \in P_{k}(K)|_F   \text{ for some } K\in   \mathcal{T}_h^{*}\}. 
\end{align*}
The X-HDG scheme for problem \eqref{pb3} reads as follows: seek $(\bm{\sigma}_h,\bm{u}_h,\bm{\hat{u}}_h,\bm{\tilde{u}}_h)\in \bm{W}_h\times \bm{V_h}\times \bm{M}_h^i\times \bm{M}_h^{\partial}(\bm{g}_D)$ such that
\begin{subequations}\label{xhdgscheme_1}
	\begin{align}
	(\mathcal{A}\bm{\sigma}_h,\bm{w})_{\mathcal{T}_h^*} + (\bm{u}_h,\nabla_h\cdot \bm{w})_{\mathcal{T}_h^*} - \langle \hat{\bm{u}}_h,\bm{w}\bm{n}\rangle_{\partial\mathcal{T}_h^*\setminus \mathcal{\varepsilon}_h^{\partial}} -\langle \tilde{\bm{u}}_h,\bm{w}\bm{n}\rangle_{\varepsilon_h^{\partial}} =& 0, \label{xhdg1_1}\\
	-(\nabla_h\cdot\bm{\sigma}_h,\bm{v})_{\mathcal{T}_h^*} + \langle \tau(\bm{u}_h-\bm{\hat{u}}_h),\bm{v}\rangle_{\partial\mathcal{T}_h^*\setminus \mathcal{\varepsilon}_h^{\partial}} 
	+  \langle \eta(\bm{u}_h-\bm{\tilde{u}}_h),\bm{v}\rangle_{\varepsilon_h^{\partial}}  =& (\bm{f},\bm{v}) ,\label{xhdg2_1}\\
	\langle \bm{\sigma}_h \bm{n},\bm{\hat{\mu}}\rangle_{\partial \mathcal{T}_h^*\setminus \mathcal{\varepsilon}_h^{\partial}}-\langle \tau(\bm{u}_h-\bm{\hat{u}}_h),\bm{\hat{\mu}}\rangle_{\partial \mathcal{T}_h^*\setminus\mathcal{\varepsilon}_h^{\partial}} =& 0,\label{xhdg3_1}\\
	\langle \bm{\sigma}_h \bm{n},\bm{\tilde{\mu}}\rangle_{ \mathcal{\varepsilon}_h^{\partial}}-\langle \eta(\bm{u}_h-\bm{\tilde{u}}_h),\bm{\tilde{\mu}}\rangle_{\mathcal{\varepsilon}_h^{\partial}} =&\langle \bm{g}_N,\bm{\tilde{\mu}}\rangle_{\mathcal{\varepsilon}_h^{\partial}}\label{xhdg4_1}
	\end{align}
\end{subequations}
for all $(\bm{w},\bm{v},\bm{\hat{\mu}},\bm{\tilde{\mu}})\in \bm{W}_h\times \bm{V_h}\times  \bm{M}_h^i\times \bm{M}_h^{\partial}(\bm{0})$, and the stabilization coefficient is given by
\begin{align}
\tau|_F = \eta|_F = 2\mu h_K^{-1}, \ \forall F\in \partial K \ {\rm or} \ F = K\cap \Omega.
\end{align}
\begin{rem}
	By following the same routine as in the proof of Theorem \ref{wellposed}, we can easily know that the scheme \eqref{xhdgscheme_1} admits a unique solution for $k\geq 1$. %here we omit the proof.
\end{rem}

\section{A priori error estimation}
%This section is devoted  the error analysis of  X-HDG scheme \eqref{xhdgscheme} for the linear elasticity interface problem \eqref{pb1}. 
%We note that all  the analysis below also applies to the scheme  \eqref{xhdgscheme_1} for  boundary-unfitted meshes.

This section is devoted  the error analysis of  X-HDG scheme \eqref{xhdgscheme} for the linear elasticity interface problem \eqref{pb1}.
% in the case of $\bm{g}_D=0$. 
 To this end, we make the following assumption on the interface $\Gamma$ and the triangulation  $\mathcal{T}_h$.\\
{\bf (A4).}
     $\Gamma\cap \partial\Omega_D$
%= \emptyset$ or $\Gamma\cap \partial\Omega_D$ 
contains at most  some vertexes (2D) or edges (3D) of elements in $\mathcal{T}_h$.
%We note that all  the analysis below also applies to the scheme  \eqref{xhdgscheme_1} for  boundary-unfitted meshes.

\subsection{Some basic results}

The following lemma from \cite{wu2010unfitted,xiao2019an}  will be used to carry out   error estimation of projections around   the interface $\Gamma$ (cf. Lemma \ref{ineq}).

\begin{lem}\label{trace}
	There exists a positive constant $h_0$ depending only on the interface $\Gamma$, the shape regularity of the mesh $\mathcal{T}_h$, and $\gamma$ in \eqref{gamma}, such that for any $h \in (0, h_0]$ and   $ K\in \mathcal{T}_h^{\Gamma}$, % interface segment/patch $\Gamma_{K}\subset \Gamma$ with,
	the following estimates hold for either $i=1$ or $i=2$: 
	\begin{align}
	\lVert v\rVert_{0,\Gamma_K} &\apprle h_K^{-1/2}\lVert v\rVert_{0,K\cap\Omega_{i}} +\lVert v\rVert_{0, K\cap\Omega_{i}}^{\frac{1}{2}} \lVert \nabla v\rVert_{0,K\cap\Omega_{i}}^{\frac{1}{2}}, \quad \forall  v\in H^1(K\cap\Omega_{i}),  \\
	%\end{align}
	%	Moreover, for any $v_h\in P_k(K)$, it holds
	%	\begin{align}
	\lVert v_h\rVert_{0,\Gamma_K} &\apprle h_K^{-1/2}\lVert v_h\rVert_{0,K\cap\Omega_{i}}, \quad \forall v_h\in P_r(K) .
	\end{align}
\end{lem}
\begin{rem}
	We note that the condition $h\in(0,h_0]$ for some $h_0$ in this lemma  is not required when $\Gamma_{K}$ is a straight line/plane segment, and this condition  is easy to satisfy when $\Gamma_{K}$ is a curved line/surface segment.
\end{rem}

	For $i=1,2$,  set $\tilde{\Omega}_i: = \{ \cup\bar{K}: K\cap \Omega_i\neq \emptyset, \ K\in \mathcal{T}_h\}$,
	and introduce  extension operators $E_i: H^s(\Omega_i)\rightarrow H^s(\tilde{\Omega}_i)$, with   integer $s\geq1$, such that 
	\begin{align*}
		(E_iw)|_{\Omega_i} = w \quad \text{and} \quad  \| E_iw\|_{H^s(\tilde{\Omega}_i)}\apprle \| w\|_{H^s(\Omega_i)}, \ \forall w\in H^s(\Omega_i).
	\end{align*}
	
		%let $\Pi_{r}^b: L^2(D)\rightarrow P_r(D)$ be the standard $L^2$ orthogonal projection operator. And let $\Pi_{r}: L^2(D)\rightarrow P_r(D)$ be the standard $L^2$ orthogonal projection operator with $D=K$. 
	
	For any $v\in H^1(\Omega_1\cup \Omega_2)$ and $ i = 1,2$,  set  $  \tilde{v}_i :=E_i(v|_{\Omega_i}) $,  and define $Q_r v$ and  $Q_r^b v$, with  integer $r\geq 0$, by
%	 define  $Q_r: L^2(\Omega)\rightarrow \cup_{K\in \mathcal{T}_h} P_r((K\cap\Omega_1)\cup (K\cap\Omega_2))$ and $Q_r^b: L^2(\varepsilon_h^*)\rightarrow \cup_{F\in \varepsilon_h^*} P_r((F\cap\Omega_1)\cup (F\cap\Omega_2))$ by 
		\begin{align} 
		&(Q_rv)|_K: = \chi_1\Pi_{r}(\tilde{v}_1|_{K})+\chi_2\Pi_{r}(\tilde{v}_2|_{K}), \quad  \forall K\in \mathcal{T}_h,\label{QrQrb-1}\\
	& (Q_r^bv)|_F: = \chi_1\Pi_{r}^b(\tilde{v}_1|_{F})+\chi_2\Pi_{r}^b(\tilde{v}_2|_{F}),\quad F\in \varepsilon_h^*. \label{QrQrb-2}
\end{align}	 
Here	 
	  $\Pi_{r}: L^2(K)\rightarrow P_r(K)$ denotes the standard $L^2$ orthogonal projection operator, and we  recall that $\Pi_{r}^b$ is the   $L^2$ orthogonal projection operator  from  $L^2(F) $ onto $P_r(F)$. 
Notice that  that if  $K\cap \Omega_i = K$ and $F\cap \Omega_i = F$, then  $$(Q_rv)|_K: = \Pi_{r}(v|_{K}), \quad (Q_r^bv)|_F: = \Pi_{r}^b(v|_{F}).$$ 
	
	Vector or tensor analogues of $Q_{r}$ and $Q_{r}^b$ are denoted by $\bm{Q}_{r}$ and $\bm{Q}_{r}^b$, respectively.

Based on  Lemma \ref{trace} and standard  properties of the $L^2$ projection operators, % $Q_{r}$  and  $Q_{r}^b$, 
we have the following conclusion.
%\begin{lem}\label{ineq} 
%	Let $s$ be an integer with $1\leq s\leq r+1$. For any $K\in \mathcal{T}_h$, $h\in (0, h_0]$ and $ v\in H^s\left((K\cap\Omega_{1})\cup (K\cap\Omega_{2}) \right)$, we have
%	\begin{align*}
%	\lVert v-Q_{r}v\rVert_{0,K}+h\lVert v-Q_{r}v\rVert_{1,K}&\apprle h_K^s\lVert v\rVert_{s,K}, \\% \qquad \quad \ \forall v\in H^s(K\cap\Omega_{i}),\\
%	\lVert v-Q_{r}v\rVert_{0,\partial K}+\lVert v-Q_{r}v\rVert_{0,\Gamma_{K}}&\apprle h_K^{s-1/2}\lVert v\rVert_{s,K}, \\ %\qquad \forall v\in H^s(K\cap\Omega_{i}),\\
%	%	\lVert v-Q_{r}v\rVert_{0,\Gamma_{K}}&\apprle h_K^{s-1/2}\lVert v\rVert_{s,K}, \\%\qquad \forall v\in H^s(K\cap\Omega_{i}),
%	\lVert v-Q_{r}^bv\rVert_{0,\partial K} %+	\lVert v-Q_{r}^{\Gamma}v\rVert_{0,\Gamma_{K}}
%	&\apprle h_K^{s-1/2}\lVert v\rVert_{s,K},
%	%		\lVert v-Q_{r}^{\Gamma}v\rVert_{0,\Gamma_{K}}&\apprle h_K^{s-1/2}\lVert v\rVert_{s,K}, 
%	\end{align*}
%	where the notations $\lVert \cdot\rVert_{s,K} $ and $\lVert\cdot\rVert_{0,\partial K}$ are understood respectively  as   $\lVert \cdot\rVert_{s,K} = \sum\limits_{i=1}^2\lVert \cdot\rVert_{s,K\cap\Omega_{i}}$ and $\lVert \cdot\rVert_{s,\partial K} = \sum\limits_{i=1}^{2}\lVert \cdot\rVert_{s,\partial K\cap \bar \Omega_{i}}$ when  $K\in  \mathcal{T}_h^\Gamma$. 
%\end{lem}

\begin{lem}\label{ineq} 
	Let $s$ be an integer with $1\leq s\leq r+1$. For any $K\in \mathcal{T}_h$, $h\in (0, h_0]$ and $ v\in H^s\left(\Omega_{1}\cup \Omega_{2} \right)$, we have
	\begin{align}
   \sum_{K\in\mathcal{T}_h}\lVert v-Q_{r}v\rVert_{0,K}^2+h^2\sum_{K\in\mathcal{T}_h}\lVert v-Q_{r}v\rVert_{1,K}^2&\apprle h^{2s}%\sum_{K\in\mathcal{T}_h}
   \lVert v\rVert_{s,\Omega_1\cup \Omega_2}^2,\label{Q1} \\
	\sum_{K\in\mathcal{T}_h}\lVert v-Q_{r}v\rVert_{0,\partial K}^2+\sum_{K\in\mathcal{T}_h^{\Gamma}}\lVert v-Q_{r}v\rVert_{0,\Gamma_{K}}^2&\apprle h^{2s-1}%\sum_{K\in\mathcal{T}_h}
	\lVert v\rVert_{s,\Omega_1\cup \Omega_2}^2, \label{Q2}\\ 
	\sum_{K\in\mathcal{T}_h}\lVert v-Q_{r}^bv\rVert_{0,\partial K} ^2
	&\apprle h^{2s-1}%\sum_{K\in\mathcal{T}_h}
	\lVert v\rVert_{s,\Omega_1\cup \Omega_2}^2,\label{Q3}
	\end{align}
	where the notations $\lVert \cdot\rVert_{s,K} $ and $\lVert\cdot\rVert_{0,\partial K}$ for $K\in  \mathcal{T}_h^\Gamma$ are understood respectively  as   $\lVert \cdot\rVert_{s,K}^2 = \sum\limits_{i=1}^2\lVert \cdot\rVert_{s,K\cap\Omega_{i}}^2$ and $\lVert \cdot\rVert_{s,\partial K}^2 = \sum\limits_{i=1}^{2}\lVert \cdot\rVert_{s,\partial K\cap \bar \Omega_{i}}^2$. 
\end{lem}
\begin{proof}
	From the definition of   $Q_r$ and the properties of the extensive operator $E_i$ and the projection $\Pi_r$ it follows
	\begin{align*}
		\sum_{K\in\mathcal{T}_h}\lVert v-Q_{r}v\rVert_{0,K}^2+ h^2 \sum_{K\in\mathcal{T}_h}\lVert v-Q_{r}v\rVert_{1,K}^2
		\leq & \sum_{i=1}^{2}\sum_{K\in\mathcal{T}_h,K\cap\tilde{\Omega}_i\neq \emptyset}(\lVert E_iv-\Pi_r E_iv\rVert_{0,K}^2+h^2\lVert E_iv-\Pi_r E_iv\rVert_{1,K}^2 )\\
		\apprle & \sum_{i=1}^{2}\sum_{K\in\mathcal{T}_h,K\cap\tilde{\Omega}_i\neq \emptyset}h_K^{2s}\lVert E_iv\rVert_{s,K}^2 \\
		%\apprle  \sum_{i=1}^{2}h^{2s}\lVert E_iv\rVert_{s,\tilde{\Omega}_i}^2		 \\
		\apprle & %\sum_{i=1}^{2}h^{2s}\lVert v\rVert_{s,\Omega_i}^2	
		 h^{2s}%\sum_{K\in\mathcal{T}_h}
		\lVert v\rVert_{s,\Omega_1\cup \Omega_2}^2,
	\end{align*}
which yields \eqref{Q1}. 
Similarly, 
by  the trace inequality and Lemma \ref{trace}   we   get 
	\begin{align*}
	\sum_{K\in\mathcal{T}_h}\lVert v-Q_{r}v\rVert_{0,\partial K}^2+  \sum_{K\in\mathcal{T}_h^{\Gamma}}\lVert v-Q_{r}v\rVert_{0,\Gamma_K}^2
	\apprle & h^{-1}\sum_{i=1}^{2}\sum_{K\in\mathcal{T}_h,K\cap\tilde{\Omega}_i\neq \emptyset}(\lVert E_iv-\Pi_r E_iv\rVert_{0,K}^2+h\lVert E_iv-\Pi_r E_iv\rVert_{1,K}^2 )\\
	\apprle & %\sum_{i=1}^{2}h^{2s}\lVert v\rVert_{s,\Omega_i}^2	
		 h^{2s-1}%\sum_{K\in\mathcal{T}_h}
		\lVert v\rVert_{s,\Omega_1\cup \Omega_2}^2
\end{align*}
and %combine with \eqref{Q1}, so we have\eqref{Q2}.Similarly, 
\begin{align*}
	\sum_{K\in\mathcal{T}_h}\lVert v-Q_{r}^bv\rVert_{0,\partial K}^2
	\leq&  \sum_{i=1}^{2}\sum_{K\in\mathcal{T}_h,K\cap\tilde{\Omega}_i\neq \emptyset}\lVert E_iv-\Pi_r^b E_iv\rVert_{0,\partial K}^2 	\\
	\apprle& h^{2s-1}\sum_{i=1}^{2}\sum_{K\in\mathcal{T}_h,K\cap\tilde{\Omega}_i\neq \emptyset}\lVert E_iv\rVert_{s, K}^2 \\
	 \apprle& h^{2s-1}%\sum_{K\in\mathcal{T}_h}
	 \lVert v\rVert_{s,\Omega_1\cup \Omega_2}^2. 
\end{align*} 
This completes the proof.
\end{proof}
\begin{rem}
	In fact, for $K\in\mathcal{T}_h$ with  $K\cap \Omega_i = K$ ($i = 1$ or $2$) and $ v\in H^s(K)$, it is easy to see that 
		\begin{align*}
		\lVert v-Q_{r}v\rVert_{0,K}+h\lVert v-Q_{r}v\rVert_{1,K}&\apprle h^{s}\lVert v\rVert_{s,K}, \\
		\lVert v-Q_{r}v\rVert_{0,\partial K}+\lVert v-Q_{r}^bv\rVert_{0,\partial K} &\apprle h^{s-1/2}\lVert v\rVert_{s,K}.
	\end{align*}
\end{rem}

\subsection{Error estimation for stress and displacement approximations}
%This section is devoted to the error estimation for stress of the XHDG scheme \eqref{xhdgscheme}. 
Let $(\bm{\sigma},\bm{u})$ be the solution of \eqref{pb1}.  For simplicity of presentation, we define
\begin{align}\label{erroroperator}
\bm{e}_h^{\sigma}: = \bm{Q}_{k-1}\bm{\sigma} - \bm{\sigma}_h , \quad  \bm{e}_h^u: =\bm{Q}_{k}\bm{u} - \bm{u}_h,  \quad  
\bm{e}_h^{\hat{u}}: =\bm{Q}_{k}^b\bm{u} -  \bm{\hat{u}}_h , \quad  \bm{e}_h^{\tilde{u}}: = \bm{Q}_{k}^{\Gamma}\bm{u} - \bm{\tilde{u}}_h.
\end{align}
Here, for  any $ F\in\varepsilon_h^{\Gamma}$,
\begin{align}
(\bm{Q}_{k}^{\Gamma}\bm{u})|_F:=&\left\{ \begin{array}{ll}
\bm{\Pi}_{k}^b(\bm{u}|_{F}), \  \text{if } \ F \text{ is a straight line/plane  segment};\\
 \begin{array}{ll} 
 \frac{1}{2}(\bm{\Pi}_{k}(\tilde{\bm{u}}_1|_{K} )|_{F}+\bm{\Pi}_{k}(\tilde{\bm{u}}_2|_{K} )|_{F}),  &\text{if }  F=F_K
=F\cap K
   \text{ is not a straight} \\
  & \text{ line/plane  segment}  \text{ for   $K\in \mathcal{T}_h^\Gamma$},
  \end{array}
\end{array}\right. \label{Q^Ga}
\end{align} 
where  $ \tilde{\bm{u}}_i:= E_i(\bm{u}|_{\Omega_i}), \ i = 1,2$. 

We have the following lemma on error equations.
\begin{lem}
	For all $(\bm{w},\bm{v},\bm{\mu},\bm{\tilde{\mu}})\in \bm{W}_h\times \bm{V}_h\times \bm{M}_h(\bm{0})\times \bm{\tilde{M}}_h$, it holds
\begin{subequations}\label{err_eqs}
	\begin{align}
	(\mathcal{A} \bm{e}_h^{\sigma},\bm{w})_{\mathcal{T}_h} + (\bm{e}_h^{u},\nabla_h\cdot \bm{w})_{\mathcal{T}_h} - \langle \bm{e}_h^{\hat{u}},\bm{w}\bm{n}\rangle_{\partial\mathcal{T}_h\setminus \mathcal{\varepsilon}_h^{\Gamma}} -\langle \bm{e}_h^{\tilde{u}},\bm{w}\bm{n}\rangle_{*,\Gamma} =& L_1(\bm{w}), \label{err_eq1}\\
	-(\nabla_h\cdot\bm{e}_h^{\sigma},\bm{v})_{\mathcal{T}_h}+ \langle \tau(\bm{e}_h^{u}-\bm{e}_h^{\hat{u}}),\bm{v}\rangle_{\partial\mathcal{T}_h\setminus \mathcal{\varepsilon}_h^{\Gamma}} 
	+  \langle \eta(\bm{e}_h^{u}-\bm{e}_h^{\tilde{u}}),\bm{v}\rangle_{*,\Gamma}  =& L_2(\bm{v}) +L_3(\bm{v})+
L_4(\bm{v}) ,\label{err_eq2}\\
	\langle \bm{e}_h^{\sigma} \bm{n},\bm{\hat{\mu}}\rangle_{\partial \mathcal{T}_h\setminus \mathcal{\varepsilon}_h^{\Gamma}} -\langle \tau(\bm{e}_h^{u}-\bm{e}_h^{\hat{u}}),\bm{\hat{\mu}}\rangle_{\partial \mathcal{T}_h\setminus\mathcal{\varepsilon}_h^{\Gamma}} =& -L_2(\bm{\hat{\mu}}) ,\label{err_eq3}\\
	\langle \bm{e}_h^{\sigma} \bm{n},\bm{\tilde{\mu}}\rangle_{*,\Gamma} - \langle \eta(\bm{e}_h^{u}-\bm{e}_h^{\tilde{u}}),\bm{\tilde{\mu}}\rangle_{*,\Gamma}=& -L_3(\bm{\tilde{\mu}}) ,\label{err_eq4}
	\end{align}
\end{subequations}
where %, for any $\psi \in H^1(\Omega_1\cup\Omega_2)\cup\bm{W}_h\cup  \bm{V}_h\cup  \bm{M}_h\cup  \bm{\tilde{M}}_h$, 
\begin{align*}
L_1(\bm{w}) &:= \langle \bm{u}-\bm{Q}_{k}^{\Gamma}\bm{u},\bm{w}\bm{n} \rangle_{ *,\Gamma}
+(\bm{u}-\bm{Q}_{k}\bm{u},\nabla_h \cdot \bm{w})_{\mathcal{T}_h^{\Gamma}}+(\mathcal{A} (\bm{\sigma}-\bm{Q}_{k-1}\bm{\sigma}),\bm{w})_{\mathcal{T}_h^{\Gamma}}+\langle \bm{u}-\bm{Q}_{k}^b\bm{u},\bm{w} \bm{n} \rangle_{\partial\mathcal{T}_h^{\Gamma}\setminus\varepsilon_h^{\Gamma}} , \\
L_2(\bm{v}) &:= \langle (\bm{\sigma}-\bm{Q}_{k-1}\bm{\sigma})\bm{n},\bm{v} \rangle_{\partial \mathcal{T}_h}+\langle \tau (\bm{Q}_{k}\bm{u} - \bm{Q}_{k}^b\bm{u}),\bm{v} \rangle_{\partial \mathcal{T}_h\setminus \mathcal{\varepsilon}_h^{\Gamma}} , \\
L_3(\bm{v}) &:= \langle (\bm{\sigma} - \bm{Q}_{k-1}\bm{\sigma})\bm{n},\bm{v}\rangle_{*,\Gamma}+ \langle \eta (\bm{Q}_{k}\bm{u}-\bm{Q}_{k}^{\Gamma}\bm{u}),\bm{v} \rangle_{*,\Gamma},\\
L_4(\bm{v}) &:=   -(\bm{\sigma} - \bm{Q}_{k-1}\bm{\sigma},\bm{\epsilon}_h(\bm{v}  ))_{\mathcal{T}_h^{\Gamma}}  ,
\end{align*}
with %\begin{align*}
$(\cdot,\cdot)_{\mathcal{T}_h^{\Gamma}}:=\sum\limits_{K\in  \mathcal{T}_h^{\Gamma}}(\cdot,\cdot)_K
$ and  $ \langle\cdot,\cdot\rangle_{\partial\mathcal{T}_h^{\Gamma}\setminus \mathcal{\varepsilon}_h^{\Gamma}}:=\sum\limits_{K\in  \mathcal{T}_h^{\Gamma}}\langle\cdot,\cdot\rangle_{\partial K\setminus \mathcal{\varepsilon}_h^{\Gamma}}$.

\end{lem}
\begin{proof}
	From the definitions of the operators $\bm{Q}_{k-1}, \bm{Q}_{k}, \bm{Q}_{k}^b$ and $\bm{Q}_{k}^{\Gamma}$,  we   obtain   for any $(\bm{w},\bm{v})\in \bm{W}_h\times \bm{V}_h$ that
		\begin{align*}
		&(\mathcal{A} \bm{Q}_{k-1}\bm{\sigma},\bm{w})_{\mathcal{T}_h} + (\bm{Q}_{k}\bm{u},\nabla_h\cdot \bm{w})_{\mathcal{T}_h} - \langle \bm{Q}_{k}^b\bm{u},\bm{w}\bm{n}\rangle_{\partial\mathcal{T}_h\setminus \mathcal{\varepsilon}_h^{\Gamma}} -\langle \bm{Q}_{k}^{\Gamma}\bm{u},\bm{w}\bm{n}\rangle_{*,\Gamma} \\
		= &\langle \bm{u}- \bm{Q}_{k}^{\Gamma}\bm{u},\bm{w}\bm{n}\rangle_{*,\Gamma}+	 (\bm{u}-\bm{Q}_{k}\bm{u},\nabla_h\cdot \bm{w})_{\mathcal{T}_h^{\Gamma}}+(\mathcal{A} (\bm{\sigma}-\bm{Q}_{k-1}\bm{\sigma}),\bm{w})_{\mathcal{T}_h^{\Gamma}}+\langle \bm{u}-\bm{Q}_{k}^b\bm{u},\bm{w} \bm{n} \rangle_{\partial\mathcal{T}_h^{\Gamma}\setminus\varepsilon_h^{\Gamma}} \\ %+	(\mathcal{A}( \bm{Q}_{k-1}\bm{\sigma}-\bm{\sigma}),\bm{w})_{\mathcal{T}_h}, \\ %\label{err_e1}\\
		&(\nabla_h\cdot\bm{Q}_{k-1}\bm{\sigma},\bm{v})_{\mathcal{T}_h}-  \langle \bm{Q}_{k-1}\bm{\sigma}-\bm{\sigma},\bm{v}\rangle_{\partial\mathcal{T}_h\setminus \mathcal{\varepsilon}_h^{\Gamma}} 
		-  \langle \bm{Q}_{k-1}\bm{\sigma}-\bm{\sigma},\bm{v}\rangle_{*,\Gamma}\\
		=& (\bm{f},\bm{v}) +(\bm{\sigma}-\bm{Q}_{k-1}\bm{\sigma},\bm{\epsilon}_h(  \bm{v}))_{\mathcal{T}_h^{\Gamma}} .
		\end{align*}
		 Subtracting \eqref{xhdg1} and \eqref{xhdg2} from the above two equations respectively yields   \eqref{err_eq1} and \eqref{err_eq2}. And \eqref{err_eq3} and \eqref{err_eq4} follow from \eqref{xhdg3}, \eqref{xhdg4} and the relations
		\begin{align*}
		\langle  \bm{\sigma}\bm{n} ,\bm{\tilde{\mu}} \rangle_{*,\Gamma} = \langle \bm{g}_N^{\Gamma},\bm{\tilde{\mu}}  \rangle_{*,\Gamma},  \quad 
		\langle \bm{\sigma}\bm{n},\bm{\hat{\mu}} \rangle_{\partial \mathcal{T}_h\setminus\mathcal{\varepsilon}_h^{\Gamma}} = \langle \bm{g}_N,\bm{\hat{\mu}}\rangle_{\partial \mathcal{T}_h\setminus\mathcal{\varepsilon}_h^{\Gamma}}.
		\end{align*}
\end{proof}

Let us define 
\begin{align*}
&\lVert \cdot \rVert_{\mathcal{A},\mathcal{T}_h}: = (\mathcal{A}\cdot,\cdot)_{\mathcal{T}_h}^{\frac12}, \quad 
\lVert \cdot\rVert_{0,\mathcal{T}_h} := (\cdot,\cdot)_{\mathcal{T}_h}^{\frac{1}{2}},\quad \lVert \cdot \rVert_{0,\partial\mathcal{T}_h\setminus \mathcal{\varepsilon}_h^{\Gamma}} :=  \langle\cdot,\cdot\rangle_{\partial\mathcal{T}_h\setminus \mathcal{\varepsilon}_h^{\Gamma}}^{\frac{1}{2}},\\
&\lVert \cdot \rVert_{*,\Gamma} := \langle \cdot,\cdot\rangle_{*,\Gamma}^{\frac{1}{2}},  \quad \lVert \cdot\rVert_{0,\mathcal{T}_h^{\Gamma}} := (\cdot,\cdot)_{\mathcal{T}_h^{\Gamma}}^{\frac{1}{2}},
\end{align*}
and introduce a semi-norm $\interleave \cdot\interleave :  (\bm{w},\bm{v},\bm{\mu},\bm{\tilde{\mu}})\in \bm{W}_h\times \bm{V}_h\times \bm{M}_h\times \bm{\tilde{M}}_h\rightarrow  \mathbb{R}$ with 
\begin{align}\label{energynorm}
\interleave (\bm{w},\bm{v},\bm{\hat{\mu}},\bm{\tilde{\mu}})\interleave : = (\lVert \bm{w}\rVert_{\mathcal{A},\mathcal{T}_h}^2 +\lVert \tau^{\frac{1}{2}}(\bm{v}-\bm{\hat{\mu}})\rVert_{0,\partial\mathcal{T}_h\setminus \mathcal{\varepsilon}_h^{\Gamma}}^2+\lVert \eta^{\frac{1}{2}}(\bm{v}-\bm{\tilde{\mu}})\rVert_{*,\Gamma}^2)^{\frac{1}{2}} .
\end{align}
%where % $\mu_{max} = \max\limits_{i=1,2}{\mu_i},$  
	
%where $\lVert \bm{w}\rVert_{0,\mathcal{A},\mathcal{T}_h}^2: = (\mathcal{A}\bm{w},\bm{w})_{0,\mathcal{T}_h}$.

\begin{lem}\label{est_energy_1}
		Let $(\bm{\sigma},\bm{u})\in % [{H^{k}}(\Omega_1\cup \Omega_2)]^{d\times d}\times [{H^{k+1}}(\Omega_1\cup \Omega_2)]^d
		\bm{H^{k}}(\Omega_1\cup \Omega_2,S)\times \bm{H^{k+1}}(\Omega_1\cup \Omega_2) 
		$ and $(\bm{\sigma}_h,\bm{u}_h,\hat{\bm{u}}_h,\tilde{\bm{u}}_h)\in \bm{W}_h\times \bm{V}_h\times \bm{M}_h(\bm{g}_D)\times \bm{\tilde{M}}_h$ be the solutions of the  problem (\ref{pb1}) and the X-HDG scheme (\ref{xhdgscheme}), respectively.  For any $h \in (0, h_0]$, it holds
	\begin{align} %
	&\interleave (\bm{e}_h^{\sigma},\bm{e}_h^u,\bm{e}_h^{\hat{u}} ,\bm{e}_h^{\tilde{u}}) \interleave^2 = \sum_{i=1}^{4}E_i, \label{est_1} \\
	&\lVert  \sqrt{\mu}\bm{\epsilon}_h(\bm{e}_h^{u})\rVert_{0,\mathcal{T}_h} \apprle 
%	\left\{ \begin{array}{ll}
%	\interleave (\bm{e}_h^{\sigma},\bm{e}_h^{u},\bm{e}_h^{\hat{u}},\bm{e}_h^{\tilde{u}})\interleave +\lVert \frac{1}{\sqrt{2\mu}}\bm{e}_h^{\sigma}\rVert_{0,\mathcal{T}_h} , \quad ${\rm if interface is piecewise polygonal,}$ \\
	\interleave  (\bm{e}_h^{\sigma},\bm{e}_h^{u},\bm{e}_h^{\hat{u}},\bm{e}_h^{\tilde{u}})\interleave+\lVert \frac{1}{\sqrt{\mu}}\bm{e}_h^{\sigma}\rVert_{0,\mathcal{T}_h}  +h^k(\lVert \sqrt{\mu} \bm{u}\rVert_{k+1,\Omega_1\cup \Omega_2} +\lVert \frac{1}{\sqrt{\mu}} \bm{\sigma} \rVert_{k,\Omega_1\cup \Omega_2}) , \label{est_2} \quad %${\rm otherwise,}$
%	\end{array}\right.
	\end{align}
	where 
	\begin{align*}
	E_1 &:= \langle (\bm{u}-\bm{Q}_{k}^{\Gamma}\bm{u}),\bm{e}_h^{\sigma}\bm{n} \rangle_{ *,\Gamma}
	 +(\bm{u}-\bm{Q}_{k}\bm{u},\nabla_h\cdot  \bm{e}_h^{\sigma})_{\mathcal{T}_h^{\Gamma}}+(\mathcal{A} (\bm{\sigma}-\bm{Q}_{k-1}\bm{\sigma}),\bm{e}_h^{\sigma})_{\mathcal{T}_h^{\Gamma}}+\langle \bm{u}-\bm{Q}_{k}^b\bm{u},\bm{e}_h^{\sigma}\bm{n} \rangle_{\partial\mathcal{T}_h^{\Gamma}\setminus\varepsilon_h^{\Gamma}}  ,\\
	E_2 &:= \langle (\bm{\sigma} - \bm{Q}_{k-1}\bm{\sigma})\bm{n},\bm{e}_h^u-\bm{e}_h^{\hat{u}} \rangle_{\partial\mathcal{T}_h\setminus \mathcal{\varepsilon}_h^{\Gamma}}+\langle (\bm{\sigma} - \bm{Q}_{k-1}\bm{\sigma})\bm{n},\bm{e}_h^u-\bm{e}_h^{\tilde{u}} \rangle_{*,\Gamma},\\
	E_3 &:= \langle \tau (\bm{Q}_{k}\bm{u} - \bm{Q}_{k}^b\bm{u}),\bm{e}_h^u-\bm{e}_h^{\hat{u}} \rangle_{\partial\mathcal{T}_h\setminus \mathcal{\varepsilon}_h^{\Gamma}} +\langle \eta (\bm{Q}_{k}\bm{u} - \bm{Q}_{k}^{\Gamma}\bm{u}),\bm{e}_h^u-\bm{e}_h^{\tilde{u}} \rangle_{*,\Gamma}, \\
	E_4&=:   -(\bm{\sigma} - \bm{Q}_{k-1}\bm{\sigma},\bm{\epsilon}_h(\bm{e}_h^{u}) )_{\mathcal{T}_h^{\Gamma}} .
	\end{align*}

\end{lem}
\begin{proof}
	Taking $(\bm{w},\bm{v},\bm{\mu} ,\bm{\tilde{\mu}}) = (\bm{e}_h^{\sigma},\bm{e}_h^u,\bm{e}_h^{\hat{u}},\bm{e}_h^{\tilde{u}})$ in the  four equations in Lemma \ref{err_eqs} and adding up them %,  together with the definitions of $L_i(\cdot)(i=1,2,3)$,
yield	  \eqref{est_1}. Then it remains to show \eqref{est_2}. We only consider the case that the interface is not a piecewise line/plane  segment, since the other case is easier. 
%	
%	Then  we just prove \eqref{est_2} in the case that interface is not a piecewise segment/polygon, since the piecewise segment/polygon case is easier. 

Taking $\bm{w} = 2\mu\bm{\epsilon}_h(\bm{e}_h^{u})$ in \eqref{err_eq1} and applying integration by parts, we obtain
	\begin{align*}
	(\mathcal{A} \bm{e}_h^{\sigma},2\mu\bm{\epsilon}_h(\bm{e}_h^{u}))_{\mathcal{T}_h} - (\bm{\epsilon}_h(\bm{e}_h^{u}),2\mu\bm{\epsilon}_h(\bm{e}_h^{u}))_{\mathcal{T}_h} +\langle \bm{e}_h^{u}-\bm{e}_h^{\hat{u}},2\mu\bm{\epsilon}_h(\bm{e}_h^{u})\bm{n}\rangle_{\partial\mathcal{T}_h\setminus \mathcal{\varepsilon}_h^{\Gamma}} +\langle \bm{e}_h^{u}-\bm{e}_h^{\tilde{u}},2\mu\bm{\epsilon}_h(\bm{e}_h^{u})\bm{n}\rangle_{*,\Gamma} = L_1(2\mu\bm{\epsilon}_h(\bm{e}_h^{u})),
	\end{align*}
which, together with 	
%Notice that $\bm{\epsilon}_h(\bm{e}_h^{u})\in \bm{W}_h$  is symmetric,  and we have 
%	$$(\nabla_h\bm{e}_h^{u},2\mu\bm{\epsilon}_h(\bm{e}_h^{u}))_{\mathcal{T}_h}  = 	\lVert \sqrt{2\mu}\bm{\epsilon}_h(\bm{e}_h^{u})\rVert_{0,\mathcal{T}_h}^2.$$
	  the Cauchy-Schwarz inequality and the trace inequality,  implies
	\begin{align*}
	&\lVert \sqrt{2\mu}\bm{\epsilon}_h(\bm{e}_h^{u})\rVert_{0,\mathcal{T}_h}^2\\ %=(\nabla_h\bm{e}_h^{u},2\mu\bm{\epsilon}_h(\bm{e}_h^{u}))_{\mathcal{T}_h} \\
	=& (\mathcal{A} \bm{e}_h^{\sigma},2\mu\bm{\epsilon}_h(\bm{e}_h^{u}))_{\mathcal{T}_h} +\langle \bm{e}_h^{u}-\bm{e}_h^{\hat{u}},2\mu\bm{\epsilon}_h(\bm{e}_h^{u})\bm{n}\rangle_{\partial\mathcal{T}_h\setminus \mathcal{\varepsilon}_h^{\Gamma}} +\langle \bm{e}_h^{u}-\bm{e}_h^{\tilde{u}},2\mu\bm{\epsilon}_h(\bm{e}_h^{u})\bm{n}\rangle_{*,\Gamma} - L_1(2\mu\bm{\epsilon}_h(\bm{e}_h^{u})) \\
	\apprle &  \lVert \sqrt{2\mu}\bm{\epsilon}_h(\bm{e}_h^{u})\rVert_{0,\mathcal{T}_h}( \lVert \sqrt{2\mu}\mathcal{A}\bm{e}_h^{\sigma}\rVert_{0,\mathcal{T}_h} +\lVert \tau^{\frac{1}{2}}(\bm{e}_h^{u}-\bm{e}_h^{\hat{u}}) \rVert_{0,\partial\mathcal{T}_h\setminus \mathcal{\varepsilon}_h^{\Gamma}}+\lVert \eta^{\frac{1}{2}}(\bm{e}_h^{u}-\bm{e}_h^{\tilde{u}}) \rVert_{0,\partial\mathcal{T}_h\setminus \mathcal{\varepsilon}_h^{\Gamma}}
	 + \lVert \sqrt{\frac{2\mu}{h}} (\bm{u}-\bm{Q}_k^{\Gamma}\bm{u})\rVert_{*,\Gamma}\\
	 &+ \lVert \frac{\sqrt{2\mu}}{h}(\bm{u}-\bm{Q}_k\bm{u})\rVert_{0,\mathcal{T}_h} + \lVert \sqrt{2\mu}\mathcal{A}(\bm{\sigma}-\bm{Q}_{k-1}\bm{\sigma}) \rVert_{0,\mathcal{T}_h}+  \lVert \sqrt{\frac{2\mu}{h}} (\bm{u}-\bm{Q}_k^b\bm{u})\rVert_{0,\partial\mathcal{T}_h\setminus\varepsilon_h^{\Gamma}}).
	\end{align*}
	Under the assumption  (A4) it holds  $\bm{Q}_k^b(\bm{u})|_F= \bm{\Pi}_k^b(\bm{u}|_F)$ on $F\subset\partial\Omega_D$. Thus,
	By  
	\eqref{def_A} and  Lemma \ref{ineq} we further get

%		and when (i) $\Gamma\cap \partial\Omega_D = \emptyset$; (ii) $\Gamma\cap \partial\Omega_D \neq \emptyset$ and $\bm{g}_D = \bm{0}$, we have $\bm{Q}_k^b(\bm{u}|_F)= \bm{\Pi}_k^b(\bm{u}|_F)$ on $\partial\Omega_D$,  
		
	\begin{align*}
	&\lVert \sqrt{2\mu}\bm{\epsilon}_h(\bm{e}_h^{u})\rVert_{0,\mathcal{T}_h} \\ 
%	\apprle& \lVert \sqrt{2\mu}\mathcal{A}\bm{e}_h^{\sigma}\rVert_{0,\mathcal{T}_h} +\lVert \tau^{\frac{1}{2}}(\bm{e}_h^{u}-\bm{e}_h^{\hat{u}}) \rVert_{0,\partial\mathcal{T}_h\setminus \mathcal{\varepsilon}_h^{\Gamma}}+\lVert \eta^{\frac{1}{2}}(\bm{e}_h^{u}-\bm{e}_h^{\tilde{u}}) \rVert_{0,\partial\mathcal{T}_h\setminus \mathcal{\varepsilon}_h^{\Gamma}}
%	+ \lVert \sqrt{\frac{2\mu}{h}} (\bm{u}-\bm{Q}_k^{\Gamma}\bm{u})\rVert_{*,\Gamma}  \\
	\apprle& \lVert \frac{1}{\sqrt{2\mu}}\bm{e}_h^{\sigma}\rVert_{0,\mathcal{T}_h} +\lVert \tau^{\frac{1}{2}}(\bm{e}_h^{u}-\bm{e}_h^{\hat{u}}) \rVert_{0,\partial\mathcal{T}_h\setminus \mathcal{\varepsilon}_h^{\Gamma}}+\lVert \eta^{\frac{1}{2}}(\bm{e}_h^{u}-\bm{e}_h^{\tilde{u}}) \rVert_{0,\partial\mathcal{T}_h\setminus \mathcal{\varepsilon}_h^{\Gamma}}
	+ h^k\lVert \sqrt{2\mu} \bm{u}\rVert_{k+1,\Omega_1\cup \Omega_2}+ h^{k}\lVert \frac{1}{\sqrt{2\mu}} \bm{\sigma} \rVert_{k,\Omega_1\cup \Omega_2} \\
	\apprle& \interleave (\bm{e}_h^{\sigma},\bm{e}_h^{u},\bm{e}_h^{\hat{u}},\bm{e}_h^{\tilde{u}})\interleave  +\lVert \frac{1}{\sqrt{2\mu}}\bm{e}_h^{\sigma}\rVert_{0,\mathcal{T}_h} +h^k\lVert \sqrt{2\mu} \bm{u}\rVert_{k+1,\Omega_1\cup \Omega_2}+h^{k}\lVert \frac{1}{\sqrt{2\mu}} \bm{\sigma} \rVert_{k,\Omega_1\cup \Omega_2} ,
	\end{align*}

	which yields the desired estimate \eqref{est_2}.
\end{proof}

%Set 	\begin{align*}
%	\bm{\hat{e}}^{\sigma}: &= \bm{\sigma}-\bm{\hat{\sigma}}_h , \quad
%	\bm{\tilde{e}}^{\sigma}: = \bm{\sigma}-\bm{\tilde{\sigma}}_h , \\
%	(\bm{\hat{\sigma}}_h\bm{n})|_{F\cap \partial K} : &= \bm{\sigma}_h\bm{n}|_{F\cap \partial K} - \tau (\bm{u}_h|_{F\cap \partial K}-\bm{\hat{u}}_h)  \ \forall F\in \varepsilon_h^{*}, \ K \in\mathcal{T}_h, \\
%	(\bm{\tilde{\sigma}}_h\bm{n})|_{F\cap\bar{\Omega}_i} : &= \bm{\sigma}_h\bm{n}|_{F\cap\bar{\Omega}_i} - \eta (\bm{u}_h|_{F\cap \bar{\Omega}_i}-\bm{\tilde{u}}_h) , \ \forall F\in \varepsilon_h^{\Gamma}, \ i = 1,2.
%	\end{align*}
Based on \cref{est_energy_1},   we can obtain   the following result.

\begin{lem}\label{lem36}
	Let $(\bm{\sigma},\bm{u})\in \bm{H^{k}}(\Omega_1\cup \Omega_2,S)\times \bm{H^{k+1}}(\Omega_1\cup \Omega_2) $ and $(\bm{\sigma}_h,\bm{u}_h,\hat{\bm{u}}_h,\tilde{\bm{u}}_h)\in \bm{W}_h\times \bm{V}_h\times \bm{M}_h(\bm{g}_D)\times \bm{\tilde{M}}_h$ be the solutions of the  problem (\ref{pb1}) and the X-HDG scheme (\ref{xhdgscheme}), respectively. Then it holds
	%	where $\bm{\hat{\sigma}}_h, \bm{\tilde{\sigma}}_h$ satisfying 
%	the  estimate
%	\begin{align} \label{add_est}
%	\lVert \frac{1}{\sqrt{2\mu}}\bm{e}_h^{\sigma}\rVert_{0,\mathcal{T}_h}^2 \apprle h^{2k}\lVert \frac{1}{\sqrt{2\mu}} \bm{\sigma} \rVert_{k,\mathcal{T}_h}^2 +h^{2k}\lVert \sqrt{2\mu} \bm{u} \rVert_{k+1,\mathcal{T}_h}^2+ \interleave (\bm{e}_h^{\sigma},\bm{e}_h^u,\bm{e}_h^{\hat{u}} ,\bm{e}_h^{\tilde{u}}) \interleave^2 .
%	\end{align}
	\begin{align} \label{add_est}
	\lVert \bm{e}_h^{\sigma}\rVert_{0,\mathcal{T}_h} %\apprle \lVert \sqrt{(\lambda+\mu)}\bm{e}_h^{\sigma}\rVert_{\mathcal{A},\mathcal{T}_h} 
	\apprle %\sqrt{\lambda_{max} +\mu_{max}} 
	 \interleave (\bm{e}_h^{\sigma},\bm{e}_h^u,\bm{e}_h^{\hat{u}} ,\bm{e}_h^{\tilde{u}}) \interleave  .
	\end{align}
Further more, for any $h \in (0, h_0]$,
	\begin{align}
	\interleave (\bm{e}_h^{\sigma},\bm{e}_h^u,\bm{e}_h^{\hat{u}},\bm{e}_h^{\tilde{\mu}})\interleave 
	\apprle  h^k%\frac{\sqrt{\lambda_{max}+\mu_{max}}}{\sqrt{\mu_{min}}}
	( \lVert \sqrt{ \mu} \bm{u}\rVert_{k+1,\Omega_1\cup \Omega_2}+\lVert \frac{1}{\sqrt{ \mu}}\bm{\sigma}\rVert_{k,\Omega_1\cup \Omega_2}). \label{est_add00}
	\end{align}
%Here $\lambda_{max} = \min\limits_{i=1,2}{\lambda_i}$, $\mu_{max} = \max\limits_{i=1,2}{\mu_i}$ and $\mu_{min} = \min\limits_{i=1,2}{\mu_i}$.
\end{lem}
\begin{proof}
	For any $\bm{w}\in \mathbb{R}^{d\times d}$, let $\bm{w}_*: = \bm{w} - \frac{1}{d}tr(\bm{w})I$ denote its deviatoric tensor. Then we can easily have, for all $\bm{w}, \bm{\tau}\in [L^2(\Omega)]^{d\times d}$, 
	\begin{align*}
		(\mathcal{A}\bm{w},\bm{\tau})_{\mathcal{T}_h} &= (\frac{1}{2\mu}\bm{w}_*,\bm{\tau}_*) _{\mathcal{T}_h}+ (\frac{1}{d(d\lambda+2\mu)}tr(\bm{w}),tr(\bm{\tau}))_{\mathcal{T}_h}, \\
		\| \bm{\tau}\|_{0,\mathcal{T}_h}^2 &= \| \bm{\tau}_*\|_{0,\mathcal{T}_h}^2 + \frac{1}{d} \| tr(\bm{\tau})\|_{0,\mathcal{T}_h}^2,
	\end{align*}
which, together with the definition of $\interleave \cdot\interleave$, indicate 
	\begin{align} \label{add_est-new}
	\lVert \bm{e}_h^{\sigma}\rVert_{0,\mathcal{T}_h} %\apprle \lVert \sqrt{(\lambda+\mu)}\bm{e}_h^{\sigma}\rVert_{\mathcal{A},\mathcal{T}_h} 
	\leq \tilde C \sqrt{\lambda_{max} +\mu_{max}} 
	 \interleave (\bm{e}_h^{\sigma},\bm{e}_h^u,\bm{e}_h^{\hat{u}} ,\bm{e}_h^{\tilde{u}}) \interleave,
	\end{align}
i.e. \eqref{add_est} holds. Here and in what follows   $\tilde C$    denotes a generic positive constant   independent of   $h$,   $\lambda_{max} = \max\limits_{i=1,2}{\lambda_i}$, $\mu_{max} = \max\limits_{i=1,2}{\mu_i}$ and $\mu_{min} = \min\limits_{i=1,2}{\mu_i}$.   % $\mu$ and $ \lambda$. 

From  \eqref{est_1}, the Cauchy-Schwarz inequality,  the inverse inequality and \cref{ineq} it follows
	\begin{align*}
		\interleave (\bm{e}_h^{\sigma},\bm{e}_h^u,\bm{e}_h^{\hat{u}} ,\bm{e}_h^{\tilde{u}}) \interleave^2  
		\leq & \| \bm{u}-\bm{Q}_k^{\Gamma}\bm{u}\|_{*,\Gamma}\|\bm{e}_h^{\sigma}\|_{*,\Gamma} +\| \bm{u}-\bm{Q}_k\bm{u}\|_{0,\mathcal{T}_h^{\Gamma}}\|\nabla_h\cdot\bm{e}_h^{\sigma}\|_{0,\mathcal{T}_h^{\Gamma}}+ \|\frac{1}{\sqrt{2\mu}}( \bm{\sigma}-\bm{Q}_{k-1}\bm{\sigma})\|_{0,\mathcal{T}_h^{\Gamma}}\|\frac{1}{\sqrt{2\mu}}\bm{e}_h^{\sigma}\|_{0,\mathcal{T}_h^{\Gamma}}\\
		&+ \| \bm{u}-\bm{Q}_k^{b}\bm{u}\|_{0,\partial\mathcal{T}_h^{\Gamma}\setminus\varepsilon_h^{\Gamma}}\|\bm{e}_h^{\sigma}\|_{0,\partial\mathcal{T}_h^{\Gamma}\setminus\varepsilon_h^{\Gamma}}  +  \|\frac{1}{\sqrt{2\mu}}( \bm{\sigma}-\bm{Q}_{k-1}\bm{\sigma})\|_{0,\mathcal{T}_h^{\Gamma}}\|\sqrt{2\mu}\bm{\epsilon}_h(\bm{e}_h^{u})\|_{0,\mathcal{T}_h^{\Gamma}}\\
		&+\left(\lVert \tau^{\frac{1}{2}}(\bm{e}_h^{u}-\bm{e}_h^{\hat{u}}) \rVert_{0,\partial\mathcal{T}_h\setminus \mathcal{\varepsilon}_h^{\Gamma}}+\lVert \eta^{\frac{1}{2}}(\bm{e}_h^{u}-\bm{e}_h^{\tilde{u}}) \rVert_{*,\Gamma}\right)\left(\lVert \tau^{-\frac{1}{2}}(\bm{\sigma}-\bm{Q}_{k-1}\bm{\sigma}) \rVert_{0,\partial\mathcal{T}_h\setminus \mathcal{\varepsilon}_h^{\Gamma}}\right.\\
		&\left.+\lVert \tau^{\frac{1}{2}}(\bm{Q}_{k}\bm{u} - \bm{Q}_{k}^b\bm{u}) \rVert_{0,\partial\mathcal{T}_h\setminus \mathcal{\varepsilon}_h^{\Gamma}}+\lVert \eta^{-\frac{1}{2}}(\bm{\sigma}-\bm{Q}_{k-1}\bm{\sigma}) \rVert_{*,\Gamma}+\lVert \eta^{\frac{1}{2}}(\bm{Q}_{k}\bm{u} - \bm{Q}_{k}^{\Gamma}\bm{u})\|_{*,\Gamma}\right)\\
		\leq & \tilde C  h^k\left(\| \sqrt{2\mu}\bm{u}\|_{k+1,\Omega_1\cup\Omega_2}+\| \frac{1}{\sqrt{2\mu}}\bm{\sigma}\|_{k,\Omega_1\cup\Omega_2} \right) \left(\lVert \tau^{\frac{1}{2}}(\bm{e}_h^{u}-\bm{e}_h^{\hat{u}}) \rVert_{0,\partial\mathcal{T}_h\setminus \mathcal{\varepsilon}_h^{\Gamma}}+\lVert \eta^{\frac{1}{2}}(\bm{e}_h^{u}-\bm{e}_h^{\tilde{u}}) \rVert_{*,{\Gamma}}\right. \\
		&\left.+\lVert  \sqrt{2\mu}\bm{\epsilon}_h(\bm{e}_h^{u})\rVert_{0,\mathcal{T}_h}  + \lVert \frac{1}{\sqrt{2\mu}}\bm{e}_h^{\sigma}\rVert_{0,\mathcal{T}_h} \right).
	\end{align*}
This inequality, together with  \eqref{est_2}, \eqref{add_est-new} and the definition of $\interleave \cdot\interleave$,   implies
\begin{align}\label{est_add00-new}
	\interleave (\bm{e}_h^{\sigma},\bm{e}_h^u,\bm{e}_h^{\hat{u}},\bm{e}_h^{\tilde{\mu}})\interleave 
	\leq \tilde C  h^k\frac{\sqrt{\lambda_{max}+\mu_{max}}}{\sqrt{\mu_{min}}}( \lVert \sqrt{ \mu} \bm{u}\rVert_{k+1,\Omega_1\cup \Omega_2}+\lVert \frac{1}{\sqrt{ \mu}}\bm{\sigma}\rVert_{k,\Omega_1\cup \Omega_2}), %\label{est_add00}
	\end{align}
i.e.  the estimate \eqref{est_add00} holds. 
%Here $\lambda_{max} = \min\limits_{i=1,2}{\lambda_i}$, $\mu_{max} = \max\limits_{i=1,2}{\mu_i}$ and $\mu_{min} = \min\limits_{i=1,2}{\mu_i}$.

\end{proof}

\begin{rem}\label{rm33}
Note that the hidden constant factors in  the upper bounds of  \eqref{add_est} and \eqref{est_add00} depend on the Lam\'e coefficient   $\lambda$ (cf.   \eqref{add_est-new} and \eqref{est_add00-new} for the explicit dependence). 	This is due to the use of  extension operators $E_i$ in the interface elements in $\mathcal{T}_h^{\Gamma}$.  In fact, if   the estimates in Lemma \ref{trace} hold simultaneously for  $i = 1,2$, then  in the  analysis we can avoid  to introduce  $E_i$  just by  defining the operators $Q_r, Q_r^b$ as  
		\begin{align}\label{redef}
		(Q_rv)|_K: = \chi_1\Pi_{r}(v|_{K\cap\Omega_{1}})+\chi_2\Pi_{r}(v|_{K\cap\Omega_{2}}),\quad 
		(Q_r^bv)|_F: = \chi_1\Pi_{r}^b(v|_{F\cap\Omega_1})+\chi_2\Pi_{r}^b(v|_{F\cap\Omega_2}). 
		\end{align}
	In this situation,   we can follow  the same line as in the proof of \cite[Theorem 3.4]{chen-xie2016}   to get the 
	coercivity inequality
%	\begin{align} \label{coer_condition}
%	\lVert\frac{1}{\sqrt{2 \mu}}\bm{\tau}_{h}\rVert_{0}^{2} \lesssim  (\mathcal{A}\bm{\tau}_{h}, \bm{\tau}_{h} )+\lVert\sqrt{\frac{h}{2 \mu}}\left(\bm{\tau}_{h} \bm{n}-\bm{
%		\hat{\tau}}_{h} \bm{n}\right)\rVert_{0, \partial\mathcal{T}_h\setminus\varepsilon_h^{\Gamma}}^{2} +\lVert\sqrt{\frac{h}{2 \mu}}\left(\bm{\tau}_{h} \bm{n}-\bm{
%		\tilde{\tau}}_{h} \bm{n}\right)\rVert_{*,\Gamma }^{2}
%	\end{align}
		\begin{align} \label{coer_condition}
	\lVert \bm{w}\rVert_{0,\mathcal{T}_h} \leq \tilde C  \sqrt{\mu_{max}} \left( \lVert \bm{w} \rVert_{\mathcal{A},\mathcal{T}_h} +\lVert\sqrt{\frac{h}{2 \mu}}\left(\bm{w} \bm{n}-\bm{\hat{w}} \bm{n}\right)\rVert_{0, \partial\mathcal{T}_h\setminus\varepsilon_h^{\Gamma}} +\lVert\sqrt{\frac{h}{2 \mu}}\left(\bm{w} \bm{n}-
	\bm{\tilde{w}} \bm{n}\right)\rVert_{*,\Gamma } \right)
	\end{align}
	for all \(\left( \bm{w}, \bm{\hat{w}},  \bm{\tilde{w}}\right) \in \bm{W}_{h} \times \bm{L^2}( \varepsilon_h^*)\times \bm{L^2}( \varepsilon_h^{\Gamma})\) satisfying
	\begin{align*}
	(\bm{w},\bm{\epsilon}_h(\bm{Q}_{k}\bm{v}))-\langle \bm{\hat{w}}\bm{n},\bm{Q}_{k}\bm{v}\rangle_{\partial\mathcal{T}_h\setminus \mathcal{\varepsilon}_h^{\Gamma}} -\langle \bm{\tilde{w}}\bm{n},\bm{Q}_{k}\bm{v}\rangle_{*,\Gamma}&=\mathbf{0}, \ \forall \bm{v}\in \bm{H^1}(\Omega)\  {\rm and}\  \bm{v}|_{\partial\Omega_D} = 0, \\
	\langle \bm{\hat{w}}\bm{n}, \bm{\hat{\mu}}\rangle_{\partial\mathcal{T}_h\setminus \mathcal{\varepsilon}_h^{\Gamma}} &=0, \ \forall \bm{\hat{\mu}} \in \bm{M}_{h}(0), \\ 
	\langle \bm{\tilde{w}}\bm{n}, \bm{\tilde{\mu}}\rangle_{*,\Gamma } &= 0, \ \forall \bm{\hat{\mu}} \in \bm{\tilde{M}}_{h}, \\ 
	tr(\bm{w}) \in& L_{0}^{2}(\Omega),  \ {\rm if }\  \Gamma_{N}=\emptyset.
	\end{align*}
By this uniform coercivity, we can further obtain
%	where % $\mu_{max} = \max\limits_{i=1,2}{\mu_i},$  
%	$\lVert \cdot \rVert_{0,\mathcal{A},\mathcal{T}_h}: = (\mathcal{A}\cdot,\cdot)_{0,\mathcal{T}_h}^{\frac12}$, %\begin{align*}
%$\lVert \cdot\rVert_{0,\mathcal{T}_h} := (\cdot,\cdot)_{\mathcal{T}_h}^{\frac{1}{2}},$  
%$\lVert \cdot \rVert_{0,\partial\mathcal{T}_h\setminus \mathcal{\varepsilon}_h^{\Gamma}} :=  \langle\cdot,\cdot\rangle_{\partial\mathcal{T}_h\setminus \mathcal{\varepsilon}_h^{\Gamma}}^{\frac{1}{2}}, $
%and $\lVert \cdot \rVert_{*,\Gamma} := \langle \cdot,\cdot\rangle_{*,\Gamma}^{\frac{1}{2}}. $
%\end{align*}
%	 uniform estimates
		\begin{align}
			\lVert \bm{e}_h^{\sigma}\rVert_{0,\mathcal{T}_h} &\leq \tilde C% \lVert \sqrt{\mu} \bm{e}_h^{\sigma}\rVert_{0,\mathcal{A},\mathcal{T}_h}\apprle
			%\sqrt{\mu_{max}} 
			 \interleave (\bm{e}_h^{\sigma},\bm{e}_h^u,\bm{e}_h^{\hat{u}} ,\bm{e}_h^{\tilde{u}}) \interleave 
			\leq \tilde C % \sqrt{\mu_{max}}   
			 h^k( \lVert \sqrt{ \mu} \bm{u}\rVert_{k+1,\Omega_1\cup \Omega_2}+\lVert \frac{1}{\sqrt{ \mu}}\bm{\sigma}\rVert_{k,\Omega_1\cup \Omega_2}), \label{est_add001-new}
\end{align}
which finally leads to error estimates which are uniform in $\lambda$ (cf. Remark \ref{rm34}).
\end{rem}

In light of  Lemmas \ref{ineq}, \ref{est_energy_1}, \ref{lem36} and the triangle inequality, %\cref{lem36}, %, inequality \eqref{add_est} and the definition of $\interleave \cdot\interleave$, 
we can easily derive the following   optimal error estimates for  the stress and displacement approximations.

\begin{thm}\label{est_energynorm}
	Let $(\bm{\sigma},\bm{u})\in \bm{H^{k}}(\Omega_1\cup \Omega_2)\times \bm{H^{k+1}}(\Omega_1\cup \Omega_2) $ and $(\bm{\sigma}_h,\bm{u}_h,\hat{\bm{u}}_h,\tilde{\bm{u}}_h)\in \bm{W}_h\times \bm{V}_h\times \bm{M}_h(\bm{g}_D)\times \bm{\tilde{M}}_h$ be the solutions of the  problem (\ref{pb1}) and the X-HDG scheme (\ref{xhdgscheme}), respectively.
	Then    for any $h \in (0, h_0]$ it holds
	\begin{align} 
		\lVert \bm{\sigma}-\bm{\sigma}_h\rVert_{0,\mathcal{T}_h}  \apprle h^k ( \lVert  \bm{u}\rVert_{k+1,\Omega_1\cup \Omega_2}+\lVert \bm{\sigma}\rVert_{k,\Omega_1\cup \Omega_2}), \label{est_add1}\\
	     \lVert  \bm{\epsilon}(\bm{u}) -\bm{\epsilon}(\bm{u}_h ) \rVert_{0,\mathcal{T}_h}  \apprle h^k( \lVert  \bm{u}\rVert_{k+1,\Omega_1\cup \Omega_2}+\lVert \bm{\sigma}\rVert_{k,\Omega_1\cup \Omega_2})). \label{est_add2}
	\end{align}
	%Here   $\mu_{min} = \min\limits_{i=1,2}{\mu_i}$.  
	%The hidden constants is dependent of the shape regular constant, and $\mu, \lambda$. 
\end{thm}

\subsection{ $L^2$  estimation for displacement approximation}
 To derive an $L^2$ error estimate for the displacement approximation by the Aubin-Nitsche's technique of duality argument, we need to introduce an auxiliary problem:
\begin{subequations}\label{pb_dual}
	\begin{align}
\mathcal{A} \bm{\Phi}-\bm{\epsilon}_h(\bm{\phi}) &=\mathbf{0}, \text { in } \Omega_1\cup\Omega_2 \label{pb_dual1}\\ 
\nabla \cdot \bm{\Phi} &=\bm{e}_h^{u}, \text { in } \Omega_1\cup\Omega_2 \label{pb_dual2}\\
\bm{\phi} &=\bm{0}, \text { on } \partial\Omega_{D}  \label{pb_dual3}\\ 
\bm{\Phi} \bm{n} &=\bm{0}, \text { on } \partial\Omega_{N} \label{pb_dual4} \\
\llbracket \bm{\phi} \rrbracket= \bm{0}, \ \llbracket \bm{\Phi}\bm{n}\rrbracket &= \bm{0}, \text { on } \Gamma \label{pb_dual5},
\end{align}
where $ \bm{e}_h^u =\bm{Q}_{k}\bm{u} - \bm{u}_h$. \end{subequations}
In addition, we assume  that the following regularity estimate holds:
\begin{align}\label{regularestimate}
\lVert \bm{\Phi}\rVert_{H^1(\Omega_1\cup\Omega_2)} +\lVert \mu\bm{\phi}\rVert_{H^2(\Omega_1\cup\Omega_2)}  \apprle \lVert \bm{e}_h^{u}\rVert_{0,\mathcal{T}_h}.
\end{align}

We note that this regularity result holds when 􏲳$\Omega$ is convex and $\Gamma$ is smooth with $\Gamma\cap\partial\Omega=\emptyset$   (cf.\cite{Chen1998Finite,%Hansbo2004elasticity,
1987Computationelasticity}).
\begin{thm}\label{0norm}
	Let $(\bm{\sigma},\bm{u})\in \bm{H^{k}}(\Omega_1\cup \Omega_2,S)\times \bm{H^{k+1}}(\Omega_1\cup \Omega_2) $ and $(\bm{\sigma}_h,\bm{u}_h,\hat{\bm{u}}_h,\tilde{\bm{u}}_h)\in \bm{W}_h\times \bm{V}_h\times \bm{M}_h(\bm{g}_D)\times \bm{\tilde{M}}_h$ be the solutions of the  problem (\ref{pb1}) and the X-HDG scheme (\ref{xhdgscheme}), respectively.
	Then  for any $h \in (0, h_0]$ it holds the error estimate 
	\begin{align}
			\lVert \bm{u}-\bm{u}_h\rVert_{0,\mathcal{T}_h}
	\apprle  h^{k+1} ( \lVert  \bm{u}\rVert_{k+1,\Omega_1\cup \Omega_2}+\lVert \bm{\sigma}\rVert_{k,\Omega_1\cup \Omega_2}).    \label{est_L2}
	\end{align}
%The hidden constants is dependent of the shape regular constant, and $\mu, \lambda$. 
\end{thm}

\begin{proof}
	Testing the equations \eqref{pb_dual2} by $\bm{e}_h^u$ and using the projection properties  and integration by parts,   we have 
	\begin{align*}
	&\lVert \bm{e}_h^u\rVert_{0,\mathcal{T}_h}^2 
	= (\nabla_h\cdot \bm{\Phi},\bm{e}_h^u)_{\mathcal{T}_h} \\
	=& (\nabla_h\cdot \bm{Q}_{k-1}\bm{\Phi},\bm{e}_h^u)_{\mathcal{T}_h}
	+\langle (\bm{\Phi} - \bm{Q}_{k-1}\bm{\Phi})\bm{n},\bm{e}_h^u\rangle_{\partial \mathcal{T}_h\setminus \mathcal{\varepsilon}_h^{\Gamma}} +\langle (\bm{\Phi} - \bm{Q}_{k-1}\bm{\Phi})\bm{n},\bm{e}_h^u\rangle_{*,\Gamma} +	 (\bm{\Phi}-\bm{Q}_{k-1}\bm{\Phi},\nabla_h\bm{e}_h^u)_{\mathcal{T}_h^{\Gamma}} ,
	\end{align*}
	which, together with   the fact that $\llbracket \bm{\Phi}\bm{n}	\rrbracket  = \bm{0}$ on $\Gamma$ and the error equation \eqref{err_eq1}, implies 
	\begin{align}\label{eq44}
	\lVert \bm{e}_h^u\rVert_{0,\mathcal{T}_h}^2 
	=& L_1(\bm{Q}_{k-1}\bm{\Phi}) - (\mathcal{A}\bm{e}_h^{\sigma} ,\bm{Q}_{k-1}\bm{\Phi})_{\mathcal{T}_h}
	+\langle (\bm{\Phi} - \bm{Q}_{k-1}\bm{\Phi})\bm{n},\bm{e}_h^u-\bm{e}_h^{\hat{u}}\rangle_{\partial \mathcal{T}_h\setminus \mathcal{\varepsilon}_h^{\Gamma}}  \nonumber \\
	+&\langle (\bm{\Phi} - \bm{Q}_{k-1}\bm{\Phi})\bm{n},\bm{e}_h^u-\bm{e}_h^{\tilde{u}}\rangle_{*,\Gamma}+	 (\bm{\Phi}-\bm{Q}_{k-1}\bm{\Phi},\nabla_h\bm{e}_h^u)_{\mathcal{T}_h^{\Gamma}}  .
	\end{align}
	Taking $(\bm{v},\bm{\hat{\mu}},\bm{\tilde{\mu}})=(\bm{Q}_{k}\bm{\phi},\bm{Q}_{k}^b\bm{\phi},\bm{Q}_{k}^{\Gamma}\bm{\phi})$ in 
	\eqref{err_eq2}-\eqref{err_eq4} yields that %%,\eqref{err_eq3}
	\begin{align*}
	-(\nabla_h\cdot\bm{e}_h^{\sigma},\bm{Q}_{k}\bm{\phi})_{\mathcal{T}_h}+ \langle \tau(\bm{e}_h^u-\bm{e}_h^{\hat{u}}),\bm{Q}_{k}\bm{\phi}\rangle_{\partial\mathcal{T}_h\setminus \mathcal{\varepsilon}_h^{\Gamma}} 
	+  \langle \eta(\bm{e}_h^u-\bm{e}_h^{\tilde{u}}),\bm{Q}_{k}\bm{\phi}\rangle_{*,\Gamma}  &= \sum_{i=2}^{4}L_i(\bm{Q}_{k}\bm{\phi}) ,\\
	\langle \bm{e}_h^{\sigma}\bm{n},\bm{Q}_{k}^b\bm{\phi}\rangle_{\partial \mathcal{T}_h\setminus \mathcal{\varepsilon}_h^{\Gamma}}-\langle \tau(\bm{e}_h^u-\bm{e}_h^{\hat{u}}),\bm{Q}_{k}^b\bm{\phi}\rangle_{\partial \mathcal{T}_h\setminus \mathcal{\varepsilon}_h^{\Gamma}}&= -L_2(\bm{Q}_{k}^b\bm{\phi}),\\
	\langle \bm{e}_h^{\sigma}\bm{n},\bm{Q}_{k}^{\Gamma}\bm{\phi}\rangle_{*,\Gamma}-\langle \eta(\bm{e}_h^u-\bm{e}_h^{\tilde{u}}),\bm{Q}_{k}^{\Gamma}\bm{\phi}\rangle_{*,\Gamma}&= -L_3(\bm{Q}_{k}^{\Gamma}\bm{\phi}).
	\end{align*}
	These relations plus \eqref{pb_dual}  lead to
	\begin{align*}
	&(\mathcal{A}\bm{\Phi},\bm{e}_h^{\sigma})_{\mathcal{T}_h} \\
	=& -(\bm{\phi},\nabla_h\cdot\bm{e}_h^{\sigma})_{\mathcal{T}_h} +\langle \bm{\phi},\bm{e}_h^{\sigma}\bm{n}\rangle_{\partial\mathcal{T}_h\setminus \mathcal{\varepsilon}_h^{\Gamma}} +\langle \bm{\phi},\bm{e}_h^{\sigma}\bm{n}\rangle_{*,\Gamma} \\
	=& -(\bm{Q}_{k}\bm{\phi},\nabla_h\cdot\bm{e}_h^{\sigma})_{\mathcal{T}_h} +\langle \bm{Q}_{k}^b\bm{\phi},\bm{e}_h^{\sigma}\bm{n}\rangle_{\partial\mathcal{T}_h\setminus \mathcal{\varepsilon}_h^{\Gamma}} +\langle \bm{Q}_{k}^{\Gamma}\bm{\phi},\bm{e}_h^{\sigma}\bm{n}\rangle_{*,\Gamma} + \langle \bm{\phi}-\bm{Q}_{k}^{\Gamma}\bm{\phi},\bm{e}_h^{\sigma}\bm{n}\rangle_{*,\Gamma}\\
	=& L_2(\bm{Q}_{k}\bm{\phi}-\bm{Q}_{k}^b\bm{\phi})+L_3(\bm{Q}_{k}\bm{\phi}-\bm{Q}_{k}^{\Gamma}\bm{\phi}) 
	+ \langle \tau(\bm{e}_h^u-\bm{e}_h^{\hat{u}}),\bm{Q}_{k}^b\bm{\phi}-\bm{Q}_{k}\bm{\phi}\rangle_{\partial\mathcal{T}_h\setminus \mathcal{\varepsilon}_h^{\Gamma}} +  \langle \eta(\bm{e}_h^u-\bm{e}_h^{\tilde{u}}),\bm{Q}_{k}^{\Gamma}\bm{\phi}-\bm{Q}_{k}\bm{\phi}\rangle_{*,\Gamma} \\
	& + \langle \bm{\phi}-\bm{Q}_{k}^{\Gamma}\bm{\phi},\bm{e}_h^{\sigma}\bm{n}\rangle_{*,\Gamma}.
	\end{align*}
By \eqref{eq44} we further get
	\begin{align*}
	\lVert \bm{e}_h^u\rVert_{0,\mathcal{T}_h}^2  = \sum_{j=1}^{4}I_j+(\bm{Q}_{k-1}\bm{\Phi}-\bm{\Phi},\mathcal{A}\bm{e}_h^{\sigma})_{\mathcal{T}_h}+	(\bm{\Phi}-\bm{Q}_{k-1}\bm{\Phi},\nabla_h\bm{e}_h^u)_{\mathcal{T}_h^{\Gamma}} ,
	\end{align*}
	where 
	\begin{align*}
	I_1 :=& \langle (\bm{Q}_{k-1}\bm{\Phi}-\bm{\Phi})\bm{n},\bm{e}_h^u-\bm{e}_h^{\hat{u}}\rangle_{\partial \mathcal{T}_h\setminus \mathcal{\varepsilon}_h^{\Gamma}} +\langle (\bm{Q}_{k-1}\bm{\Phi}-\bm{\Phi})\bm{n},\bm{e}_h^u-\bm{e}_h^{\tilde{u}}\rangle_{*,\Gamma}, \\
	I_2 :=&L_2(\bm{Q}_{k}\bm{\phi}-\bm{Q}_{k}^b\bm{\phi})+L_3(\bm{Q}_{k}\bm{\phi}-\bm{Q}_{k}^{\Gamma}\bm{\phi}), \\
	I_3 := &\langle \tau(\bm{e}_h^u-\bm{e}_h^{\hat{u}}),\bm{Q}_{k}^b\bm{\phi}-\bm{Q}_{k}\bm{\phi}\rangle_{\partial\mathcal{T}_h\setminus \mathcal{\varepsilon}_h^{\Gamma}} +  \langle \eta(\bm{e}_h^u-\bm{e}_h^{\tilde{u}}),\bm{Q}_{k}^{\Gamma}\bm{\phi}-\bm{Q}_{k}\bm{\phi}\rangle_{*,\Gamma} , \\
	I_4 :=&  \langle \bm{\phi}-\bm{Q}_{k}^{\Gamma}\bm{\phi},\bm{e}_h^{\sigma}\bm{n}\rangle_{*,\Gamma} + L_1(\bm{Q}_{k-1}\bm{\Phi}) + L_4(\bm{Q}_k\bm{\phi}) .
	\end{align*}
	In light of the Cauchy-Schwarz inequality and  Lemma \ref{ineq}, % and the definition of operator $\mathcal{A}$ in \eqref{def_A}, 
	we   obtain
	\begin{align*}
	(\bm{Q}_{k-1}\bm{\Phi}-\bm{\Phi},\mathcal{A}\bm{e}_h^{\sigma})_{\mathcal{T}_h} +	 (\bm{\Phi}-\bm{Q}_{k-1}\bm{\Phi},\bm{\epsilon}_h(\bm{e}_h^u))_{\mathcal{T}_h^{\Gamma}}&\leq (\lVert \mathcal{A}\bm{e}_h^{\sigma}\rVert_{0,\mathcal{T}_h} + \lVert \bm{\epsilon}_h(\bm{e}_h^u)\rVert_{0,\mathcal{T}_h})\lVert \bm{Q}_{k-1}\bm{\Phi}-\bm{\Phi} \rVert_{0,\mathcal{T}_h} \\
	&\apprle h(\lVert \mu^{-1}\bm{e}_h^{\sigma}\rVert_{0,\mathcal{T}_h} +\lVert \bm{\epsilon}_h(\bm{e}_h^u)\rVert_{0,\mathcal{T}_h})\lVert \bm{\Phi} \rVert_{1,\Omega_{1}\cup\Omega_{2}}.
	\end{align*}
	From the definition of $\interleave \cdot \interleave$ it follows
	\begin{align*}
	I_1 &\leq \lVert (\bm{Q}_{k-1}\bm{\Phi}-\bm{\Phi})\rVert_{\partial\mathcal{T}_h\setminus \mathcal{\varepsilon}_h^{\Gamma}} \lVert \tau^{-\frac{1}{2}}\tau^{\frac{1}{2}}(\bm{e}_h^u-\bm{e}_h^{\hat{u}}) \rVert_{\partial\mathcal{T}_h\setminus \mathcal{\varepsilon}_h^{\Gamma}} +  \lVert (\bm{Q}_{k-1}\bm{\Phi}-\bm{\Phi})\rVert_{*,{\Gamma}} \lVert \eta^{-\frac{1}{2}}\eta^{\frac{1}{2}}(\bm{e}_h^u-\bm{e}_h^{\hat{u}}) \rVert_{*,{\Gamma}}\\
	&\apprle h\lVert \bm{\Phi}\rVert_{1,\Omega_{1}\cup\Omega_{2}} \interleave \mu^{-\frac{1}{2}}(\bm{e}_h^{\sigma},\bm{e}_h^u,\bm{e}_h^{\hat{u}},\bm{e}_h^{\tilde{u}})\interleave.
	\end{align*}
	Similarly, we can obtain %the   estimates of $I_2,I_3$:
	\begin{align*}
	I_2 &\leq h^{k+1}\lVert \mu\bm{\phi}\rVert_{2,\Omega_{1}\cup\Omega_{2}}(\lVert \mu^{-1}\bm{\sigma}\rVert_{k,\Omega_{1}\cup\Omega_{2}} +\lVert \bm{u}\rVert_{k+1,\Omega_{1}\cup\Omega_{2}} ), \\
	I_3 &\leq h\lVert \mu\bm{\phi}\rVert_{2,\Omega_{1}\cup\Omega_{2}} \interleave \mu^{-\frac{1}{2}}(\bm{e}_h^{\sigma},\bm{e}_h^u,\bm{e}_h^{\hat{u}},\bm{e}_h^{\tilde{u}})\interleave .
	\end{align*}
	It remains to estimate  $I_4$.
	Due to the fact that $\bm{u}, \bm{Q}_{k}^{\Gamma}\bm{u}$ and $\bm{\Phi}$ are all single-valued on $F\in \varepsilon_h^*$,   we have

	\begin{align*}
	I_4 =& \langle \bm{\phi}-\bm{Q}_{k}^{\Gamma}\bm{\phi},\bm{e}_h^{\sigma}\bm{n}\rangle_{*,\Gamma} +\langle \bm{u}-\bm{Q}_{k}^{\Gamma}\bm{u},\bm{Q}_{k-1}\bm{\Phi}\bm{n}\rangle_{*,\Gamma} + (\bm{u}-\bm{Q}_{k}\bm{u},\nabla_h\cdot  \bm{Q}_{k-1}\bm{\Phi})_{\mathcal{T}_h^{\Gamma}}\\
	&+(\mathcal{A} (\bm{\sigma}-\bm{Q}_{k-1}\bm{\sigma}),\bm{Q}_{k-1}\bm{\Phi})_{\mathcal{T}_h^{\Gamma}}+\langle \bm{u}-\bm{Q}_{k}^b\bm{u},\bm{Q}_{k-1}\bm{\Phi}\bm{n} \rangle_{\partial\mathcal{T}_h^{\Gamma}\setminus\varepsilon_h^{\Gamma}}-(\bm{\sigma} - \bm{Q}_{k-1}\bm{\sigma},\nabla_h\bm{Q}_k\bm{\phi} )_{\mathcal{T}_h^{\Gamma}}\\
	= &\langle  \bm{\phi}-\bm{Q}_{k}^{\Gamma}\bm{\phi},\bm{e}_h^{\sigma}\bm{n}\rangle_{*,\Gamma} +\langle \bm{u}-\bm{Q}_{k}^{\Gamma}\bm{u},(\bm{Q}_{k-1}\bm{\Phi}-\bm{\Phi})\bm{n}\rangle_{*,\Gamma} + (\bm{u}-\bm{Q}_{k}\bm{u},\nabla_h\cdot  \bm{Q}_{k-1}\bm{\Phi})_{\mathcal{T}_h^{\Gamma}}\\
	&+(\bm{\sigma}-\bm{Q}_{k-1}\bm{\sigma},\bm{Q}_{k-1}\bm{\epsilon}_h(\bm{\phi}))_{\mathcal{T}_h^{\Gamma}}+\langle \bm{u}-\bm{Q}_{k}^b\bm{u},(\bm{Q}_{k-1}\bm{\Phi}-\bm{\Phi})\bm{n} \rangle_{\partial\mathcal{T}_h^{\Gamma}\setminus\varepsilon_h^{\Gamma}}-(\bm{\sigma} - \bm{Q}_{k-1}\bm{\sigma},\bm{\epsilon}_h(\bm{Q}_k\bm{\phi} )_{\mathcal{T}_h^{\Gamma}}\\
	= &\langle  \bm{\phi}-\bm{Q}_{k}^{\Gamma}\bm{\phi},\bm{e}_h^{\sigma}\bm{n}\rangle_{*,\Gamma} +\langle \bm{u}-\bm{Q}_{k}^{\Gamma}\bm{u},(\bm{Q}_{k-1}\bm{\Phi}-\bm{\Phi})\bm{n}\rangle_{*,\Gamma} + (\bm{u}-\bm{Q}_{k}\bm{u},\nabla_h\cdot  \bm{Q}_{k-1}\bm{\Phi})_{\mathcal{T}_h^{\Gamma}}\\
	&+\langle \bm{u}-\bm{Q}_{k}^b\bm{u},(\bm{Q}_{k-1}\bm{\Phi}-\bm{\Phi})\bm{n} \rangle_{\partial\mathcal{T}_h^{\Gamma}\setminus\varepsilon_h^{\Gamma}}+(\bm{\sigma} - \bm{Q}_{k-1}\bm{\sigma},\bm{Q}_{k-1}\bm{\epsilon}_h(\bm{\phi}) -\bm{\epsilon}_h(\bm{Q}_k\bm{\phi} )_{\mathcal{T}_h^{\Gamma}}\\
	\apprle &h\lVert \mu^{-1}\bm{e}_h^{\sigma}\rVert_{0,\mathcal{T}_h}\lVert \mu\bm{\phi}\rVert_{2,\Omega_{1}\cup\Omega_{2}} + h^{k+1}( \lVert \bm{u}\rVert_{k+1,\Omega_{1}\cup\Omega_{2}}\lVert \bm{\Phi} \rVert_{1,\Omega_{1}\cup\Omega_{2}} + \lVert \mu^{-1}\bm{\sigma}\rVert_{k,\Omega_{1}\cup\Omega_{2}}\lVert \mu\bm{\phi}\rVert_{2,\Omega_{1}\cup\Omega_{2}} ). 
	\end{align*}

	All the above estimates, together with  the regularity assumption
	\eqref{regularestimate} and Theorem \ref{est_energynorm}, imply the desired estimate  \eqref{est_L2}.
\end{proof}

	\begin{rem}\label{rm34} 
	We note that the hidden constant $C$ in the estimates of \cref{est_energynorm} and  \cref{0norm}  depends on the Lam\'e coefficient $\lambda$, although   the numerical results in Section 4 demonstrate the uniform convergence of the X-HDG method.
	
	In fact, as shown in Remark \ref{rm33},	if   the estimates in Lemma \ref{trace} hold simultaneously for  $i = 1,2$ (e.g.  when $\Gamma$ is not close to an edge or a vertex of element), then  we  can use  \eqref{redef}, instead of \eqref{QrQrb-1}-\eqref{QrQrb-2}, to define  the operators $Q_r, Q_r^b$   in the whole error analysis. As a result, 
we can derive  the following     uniform optimal error estimates:
		\begin{align}
			\lVert \bm{\sigma}-\bm{\sigma}_h\rVert_{0,\mathcal{T}_h}  &\leq \tilde C h^k\sqrt{\mu_{max}}( \lVert \sqrt{ \mu} \bm{u}\rVert_{k+1,\Omega_1\cup \Omega_2}+\lVert \frac{1}{\sqrt{ \mu}}\bm{\sigma}\rVert_{k,\Omega_1\cup \Omega_2}), \label{est_add11}\\
			\lVert  \bm{\epsilon}(\bm{u}) -\bm{\epsilon}(\bm{u}_h ) \rVert_{0,\mathcal{T}_h}  &\leq \tilde C h^k \frac{\sqrt{\mu_{max}}}{\mu_{min}}( \lVert \sqrt{ \mu} \bm{u}\rVert_{k+1,\Omega_1\cup \Omega_2}+\lVert \frac{1}{\sqrt{ \mu}}\bm{\sigma}\rVert_{k,\Omega_1\cup \Omega_2}),  \label{est_add21}\\
			\lVert \bm{u}-\bm{u}_h\rVert_{0,\mathcal{T}_h}
			&\leq \tilde C  h^{k+1}  \frac{\sqrt{\mu_{max}}}{\mu_{min}}( \lVert \sqrt{ \mu} \bm{u}\rVert_{k+1,\Omega_1\cup \Omega_2}+\lVert \frac{1}{\sqrt{ \mu}}\bm{\sigma}\rVert_{k,\Omega_1\cup \Omega_2}).    \label{est_L21}
		\end{align}
Here  we recall that $\tilde C$ is   a generic positive constant   independent of   $h$,    $\mu$ and $ \lambda$.  
	\end{rem}

\begin{rem}
	By following the same  routines as in the analysis of the interface-unfitted scheme \eqref{xhdgscheme}, it is easy to see that  Theorems \ref{est_energynorm} and   \ref{0norm}   still hold for   the   boundary-unfitted scheme \eqref{xhdgscheme_1} of the problem \eqref{pb2} in  either of the following two cases (also cf. Remark \ref{rmk2.1}): 

	\begin{enumerate}   
	 \item[(i)] $   \partial\Omega_D$   is   a piecewise straight segment/polygon;
	
	\item[(ii)]  $   \partial\Omega_D$   is not a piecewise straight segment/polygon and $\bm{g}_D=0$.
	\end{enumerate}
	
%	The Dirichlet boundary condition in \eqref{pb2} is homegeneous, i.e. $ \llbracket \bm{u} \rrbracket=0$  on $ \partial\Omega_{D}$, if     $   \partial\Omega_D$   is not a piecewise straight segment/polygon.
	% the boundary of $\Omega$ is piecewise polygonal; \\ 
\end{rem}

\section{Numerical experiments}
In this section, we   shall provide several numerical examples to verify the performance of the proposed  interface/boundary-unfitted X-HDG method.

\begin{exmp}  \label{ex2} A plane strain test with a circular interface.
\end{exmp}
%used in \cite{Qiu2016An}
	This example is a plane strain test.  In   \eqref{pb1} we set   (cf. Figure \ref{circledomain}) 
	$$\Omega=[0,1]^2,\ \Omega_2 = \{(x,y): (x-\frac{1}{2})^2+(y-\frac{1}{2})^2<\frac{3}{64}\}, \ % with   $r_0 = \frac{\sqrt{3}}{8}$, 
	\text{ and }\Omega_1 = \Omega\backslash \overline{\Omega}_2.$$
	The exact solution $(\bm{u},\bm{\sigma})$ in $ \Omega_1\cup\Omega_2$ is given by
	\begin{align*}
	\bm{u}(x,y) = \begin{pmatrix}
	-x^2(x-1)^2y(y-1)(2y-1) \\
	y^2(y-1)^2x(x-1)(2x-1)
	\end{pmatrix}, \quad \bm{\sigma }(x,y)=2\mu \bm{\epsilon}_h(\bm{u})+\lambda \text{div}\ \bm{u}\  \bm{ I},  %\quad  {\rm in}\  \Omega_1\cup\Omega_2,
	\end{align*}
where   the Lam\'e coefficients % $\mu, \lambda$ satisfy 
	$\mu = \frac{E}{2(1+\nu)},\quad  \lambda = \frac{E\nu}{(1+\nu)(1-2\nu)},$ 
	with  the Young's modulus $E|_{\Omega_1\cup\Omega_2} = 3$, the Poisson ratio $\nu|_{\Omega_1} = 0.4$ and $\nu|_{\Omega_2} = 0.4, 0.49, 0.4999,0.499999$. We note that the material tends to incompressible  as $\nu\rightarrow 0.5$  (or $\lambda\rightarrow \infty$).  The  force term, boundary conditions  and interface conditions can be derived explicitly.

%	where $\nu$ is the Possion Ratio and  $E$  . , , i.e. $\mu_1 = 1.0714, \lambda_1 = 4.2857$, $\mu_2 = 1.0714, \lambda_2 = 4.2857$, $\mu_2 = 1.0067, \lambda_2 = 4.9329E+01$ , $\mu_2 = 1.0001, \lambda_2 = 4.9993E+03$ and $\mu_2 = 1.0, \lambda_2 = 5.0E+05$, respectively.
\begin{figure}[htbp]
	\centering	
	\includegraphics[height = 4.3 cm,width=5.5 cm]{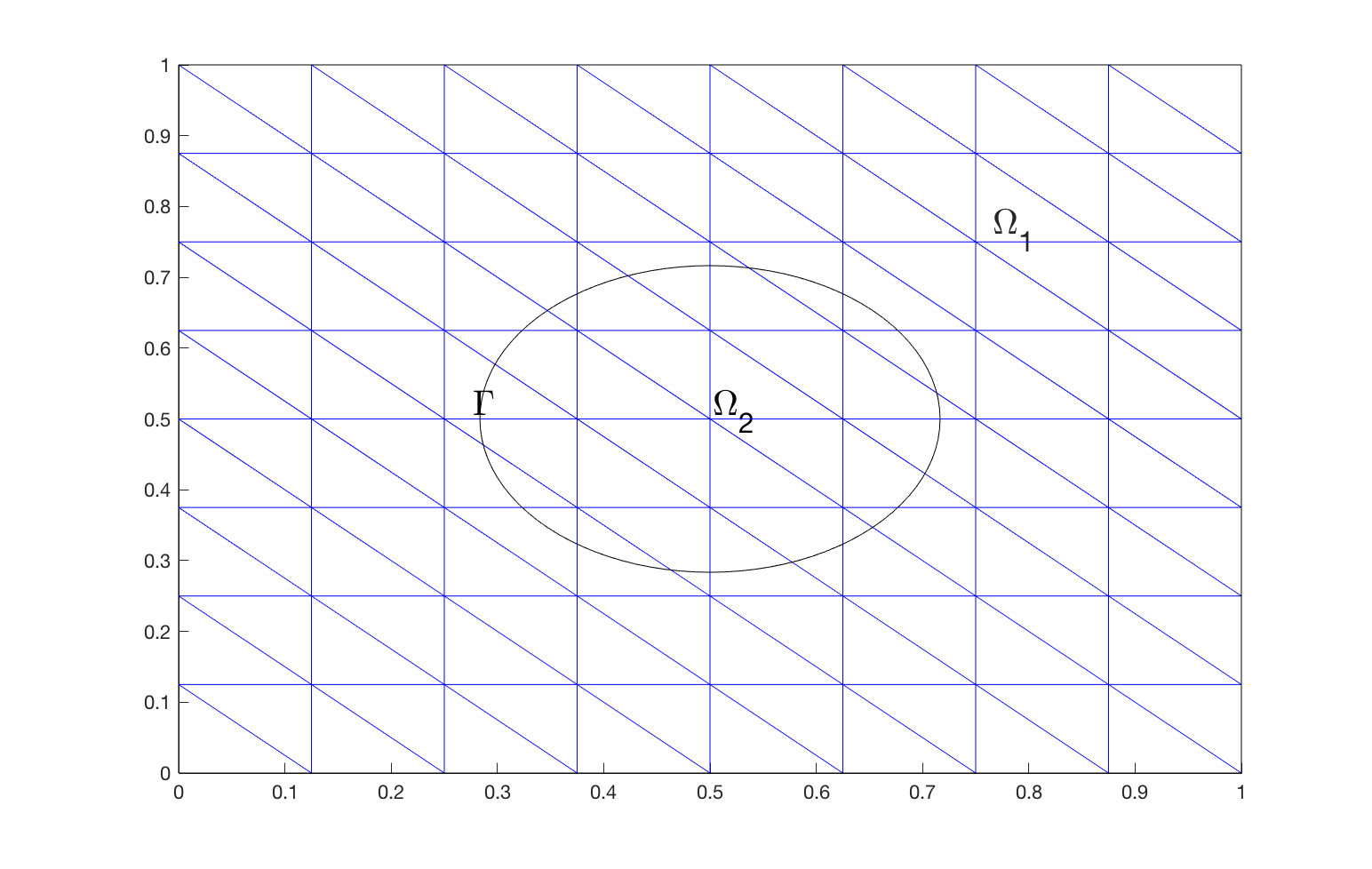} 				
	\caption{The domain with a circular interface: $8\times 8$ mesh}\label{circledomain}	
\end{figure}

We use $N\times N$ uniform triangular meshes for the computation.  Errors of displacement and stress approximations with $k = 1,2$ are shown in Table \ref{ex_table2}. We can see that our X-HDG method \eqref{xhdgscheme} yields  $(k+1)$-th   and  $k$-th orders   of convergence  for $\lVert \bm{u}-\bm{u}_{h}\rVert_0$ and $\lVert \bm{\sigma}-\bm{\sigma}_{h}\rVert_0$, respectively,    which are uniform as $\nu$ tends to $0.5$. These results are  conformable to Theorems \ref{0norm} and  \ref{est_energynorm}.

\begin{table}[H]
	\normalsize
	\caption{History of convergence:  Example \ref{ex2} }
	\label{ex_table2}
	\centering
	\footnotesize
	%		\subtable[ $k =1$ and $$ and $\alpha_1:\alpha_2=1:1000$]
	{
		\begin{tabular}{p{1.5cm}<{\centering}|p{1.15cm}<{\centering}|p{1.45cm}<{\centering}|p{0.45cm}<{\centering}|p{1.45cm}<{\centering}|p{0.45cm}<{\centering}|p{1.45cm}<{\centering}|p{0.45cm}<{\centering}|p{1.45cm}<{\centering}|p{0.45cm}<{\centering}}
			\hline   
			\multirow{3}{*}{$\nu|_{\Omega_2}$}&\multirow{3}{*}{mesh}&  \multicolumn{4}{c|}{$k  =1$}&\multicolumn{4}{c}{$ k = 2$} \\
			\cline{3-10}
			&&\multicolumn{2}{c|}{$\frac{\lVert \bm{u}-\bm{u}_{h}\rVert_0}{\lVert \bm{u}\rVert_0}$ }&\multicolumn{2}{c|}{$\frac{\lVert \bm{\sigma}-\bm{\sigma}_{h}\rVert_0}{\lVert \bm{\sigma}\rVert_0}$}&\multicolumn{2}{c|}{$\frac{\lVert \bm{u}-\bm{u}_{h}\rVert_0}{\lVert \bm{u}\rVert_0}$ }&\multicolumn{2}{c}{$\frac{\lVert \bm{\sigma}-\bm{\sigma}_{h}\rVert_0}{\lVert \bm{\sigma}\rVert_0}$ }\cr\cline{3-10}  
			&&error&order&error&order&error&order&error&order\cr 
			\cline{1-10}
			\multirow{4}{*}{$ 0.4$}
			&$8\times 8$	    &1.3656E-01   & -- 	&3.0139E-01	  &	  --  &8.6182E-03	  
			&   --   &4.5657E-02    &--		   \\
			%			\hline
			&$16\times 16$	 &3.7733E-02    &1.86	&1.4693E-01	  &1.04		 &1.3317E-03  &2.69        &1.1478E-02    &1.99	 \\
			%			\hline
			&$32\times 32$	&9.8295E-03    &1.94	&7.2247E-02	  &1.02   &1.8939E-04	  &2.81  &2.7901E-03    &2.04	 \\
			%			\hline
			&$64\times 64$	&2.4915E-03   &1.98	  &3.5934E-02  &1.01	 &2.5407E-05	  &2.90   &6.7861E-04    &2.04	   \\
			\hline
			
			\multirow{4}{*}{$ 0.49$}
			&$8\times 8$	    &1.3343E-01   & -- 	&3.0096E-01	  &	  --  &8.6171E-03	  
			&   --   &4.5613E-02    &--		   \\
			%			\hline
			&$16\times 16$	 &3.6398E-02    &1.87	&1.4674E-01	  &1.04		 &1.3315E-03  &2.69        &1.1463E-02    &1.99	 \\
			%			\hline
			&$32\times 32$	&9.4020E-03    &1.95	&7.2150E-02	  &1.02   &1.8936E-04	  &2.81  &2.7857E-03    &2.04	 \\
			%			\hline
			&$64\times 64$	&2.2876E-03   &2.04	  &3.5885E-02  &1.01	 &2.5404E-05	  &2.90   &6.7738E-04    &2.04	   \\
			\hline
			
			\multirow{4}{*}{$0.4999$}
			&$8\times 8$	    &1.3312E-01   & -- 	&3.0093E-01	  &	  --  &8.6171E-03	  
			&   --   &4.5609E-02    &--		   \\
			%			\hline
			&$16\times 16$	 &3.6268E-02    &1.88	&1.4673E-01	  &1.04		 &1.3315E-03  &2.69        &1.1461E-02    &1.99	 \\
			%			\hline
			&$32\times 32$	&9.3616E-03    &1.95	&7.2142E-02	  &1.02   &1.8936E-04	  &2.81  &2.7852E-03    &2.04	 \\
			%			\hline
			&$64\times 64$	&2.2709E-03   &2.04	  &3.5881E-02  &1.01	 &2.5404E-05	  &2.90   &6.7726E-04    &2.04	   \\
			\hline
			
			\multirow{4}{*}{$0.499999$}
			&$8\times 8$	    &1.3311E-01   & -- 	&3.0093E-01	  &	  --  &8.6171E-03	  
			&   --   &4.5609E-02    &--		   \\
			%			\hline
			&$16\times 16$	 &3.6266E-02    &1.88	&1.4673E-01	  &1.04		 &1.3315E-03  &2.69        &1.1461E-02    &1.99	 \\
			%			\hline
			&$32\times 32$	&9.3612E-03    &1.95	&7.2142E-02	  &1.02   &1.8936E-04	  &2.81  &2.7852E-03    &2.04	 \\
			%			\hline
			&$64\times 64$	&2.2709E-03   &2.04	  &3.5881E-02  &1.01	 &2.5404E-05	  &2.90   &6.7726E-04    &2.04	   \\
			\hline
		\end{tabular}
	}
\end{table}

\begin{rem} We note that in  implementation of the scheme \eqref{xhdgscheme}  on very refined meshes one may need some special handling of the approximation space $\bm{\tilde{M}}_h = \{\bm{\tilde{\mu}} \in \bm{L^2}(F): 
	\bm{\tilde{\mu}}|_{F}\in \bm{P}_{k}(K)|_F,\forall F\in\mathcal{\varepsilon}_h^{\Gamma}\}$. Taking  the circular interface in $2-dimension$ as an example, when  the mesh size $h$ becomes small enough,  $F\in \varepsilon_h^{\Gamma}$ will be close to a line segment. In this situation, the coordinates $x$ and $y$ on $F$ are approximately linearly-dependent.  Thus, the direct use of $\bm{\tilde{M}_h}$ may lead to a very  large condition number of the resultant stiffness matrix. In this situation, one can replace $\bm{\tilde{M}}_h$ with
	$$\bm{\tilde{M}}_h^{*1} = \{\bm{\tilde{\mu}} \in \bm{L^2}(F): 
	\bm{\tilde{\mu}}|_{F}\in span\{1,x,\cdots, x^k\},\forall F\in\mathcal{\varepsilon}_h^{\Gamma}\}$$
	or $$\bm{\tilde{M}}_h^{*2} = \{\bm{\tilde{\mu}} \in \bm{L^2}(F): 
	\bm{\tilde{\mu}}|_{F}\in span\{1,y,\cdots, y^k\},\forall F\in\mathcal{\varepsilon}_h^{\Gamma}\}$$
	according to the average slope of   $F$.
	%We note that in
	% the scheme \eqref{xhdgscheme},  $\tilde{\bm{\mu}}|_{F}\in \bm{P}_{k}(K)|_F$ for any  $\bm{\tilde{\mu}}\in \bm{\tilde{M}_h}$,. If $\Gamma$ is a piecewise smooth curve, then $F$ will become more and more closer to a  as the mesh size $h$ becomes smaller and smaller, $F$ will Taking circular interface in $2-dimension$ as an example, if $k = 1$, then $P_{k}(K)|_F = span\{1,x,y\}$.  With the mesh size $h$ decreasing, the properties of linearly dependent becomes more and more strong, i.e. $ax+by\approx 1$, which will influence the condition number of matrix. 
	Numerical tests indicate that such a modification  does not affect the  accuracy of the scheme. 
	%we can take $\tilde{\bm{\mu}}=  span\{1,x\}^2$ or $\tilde{\bm{\mu}}=  span\{1,y\}^2$ in implementation, which won't influence the optimal convergence rate in examples in Section 4. Similarly, if $k = 2$, then $P_{k}(K)|_F = span\{1,x,y,xy,x^2\}$. With the mesh size $h$ decreasing, we have $ax+by\approx 1$ and $xy \approx x(1-ax)$, so we can take $\tilde{\bm{\mu}}=  span\{1,x,x^2\}^2$ or $\tilde{\bm{\mu}}=  span\{1,y,y^2\}^2$ in implementation.
\end{rem}

%\begin{figure}[htbp]
%	\centering
%	\begin{minipage}[t]{0.4\textwidth}		
%		\includegraphics[height = 5cm,width=6cm]{ex2_0k_1} 	
%	\end{minipage}
%	\begin{minipage}[t]{0.4\textwidth}		
%		\includegraphics[height = 5cm,width=6cm]{ex2_0k_2} 	
%	\end{minipage}
%	\begin{minipage}[t]{0.4\textwidth}		
%		\includegraphics[height = 5cm,width=6cm]{ex2_1k_1} 	
%	\end{minipage}
%	\begin{minipage}[t]{0.4\textwidth}		
%		\includegraphics[height = 5cm,width=6cm]{ex2_1k_2} 	
%	\end{minipage}
%	\begin{minipage}[t]{0.4\textwidth}		
%		\includegraphics[height = 5cm,width=6cm]{ex2_2k_1} 	
%	\end{minipage}
%	\begin{minipage}[t]{0.4\textwidth}		
%		\includegraphics[height = 5cm,width=6cm]{ex2_2k_2} 	
%	\end{minipage}
%	\begin{minipage}[t]{0.4\textwidth}		
%		\includegraphics[height = 5cm,width=6cm]{ex2_3k_1} 	
%	\end{minipage}
%	\begin{minipage}[t]{0.4\textwidth}		
%		\includegraphics[height = 5cm,width=6cm]{ex2_3k_2} 	
%	\end{minipage}
%	\tiny\caption{Convergence rate of Example \ref{ex2} with $k = 1$(left) and $k=2$(right): $\nu|_{\Omega_{2}} = 0.4, 0.49, 0.4999, 0.499999$(from top to bottom).}	\label{error_examp_2}
%\end{figure}

\begin{exmp}\label{ex3}  A plane strain test on a circular domain: boundary-unfitted meshes.
\end{exmp}

 This example is  to test the performance of  the X-HDG scheme  \eqref{xhdgscheme_1} with  boundary-unfitted meshes (cf. Figure \ref{domain_circle}).  In \eqref{pb2} we set  
 $$\Omega =\{(x,y): (x-\frac{1}{2})^2+(y-\frac{1}{2})^2<\frac{3}{16}\}.$$ 
 And the exact solution $(\bm{u},\bm{\sigma})$ has the same form as in Example \ref{ex2}, i.e.
	\begin{align*}
	\bm{u}(x,y) &= \begin{pmatrix}
	-x^2(x-1)^2y(y-1)(2y-1) \\
	y^2(y-1)^2x(x-1)(2x-1)
	\end{pmatrix}, \quad \bm{\sigma }(x,y)= \frac{E}{1+\nu} \bm{\epsilon}_h(\bm{u})+\frac{E\nu}{(1+\nu)(1-2\nu)} \text{div}\ \bm{u}\  \bm{ I},  %\quad  {\rm in}\  \Omega_1\cup\Omega_2,
	\end{align*}
	where   
%	the Lam\'e coefficients % $\mu, \lambda$ satisfy 
%	$\mu = \frac{E}{2(1+\nu)},\quad  \lambda = \frac{E\nu}{(1+\nu)(1-2\nu)},$ 
	   the Young's modulus $E = 3$, and the Poisson ratio $\nu = 0.49, 0.4999,0.499999$. %\end{exmp}

%We take $E = 3$, and $\nu = 0.49, 0.4999,0.499999$, i.e. $\mu = 1.0067, \lambda = 4.9329E+01$ , $\mu = 1.0001, \lambda = 4.9993E+03$ and $\mu = 1.0, \lambda = 5.0E+05$, respectively. The stresses and force term can be derived explicity. Errors of displacement and stress approximations with $k = 1,2$ are shown in Table \ref{ex_table3}. We can see that our X-HDG scheme \eqref{xhdgscheme_1} yields optimal convergence rates which are uniform with respect to $\lambda$ for all cases. 

\begin{figure}[H]
	\centering
	\includegraphics[height = 5 cm,width= 5.5 cm]{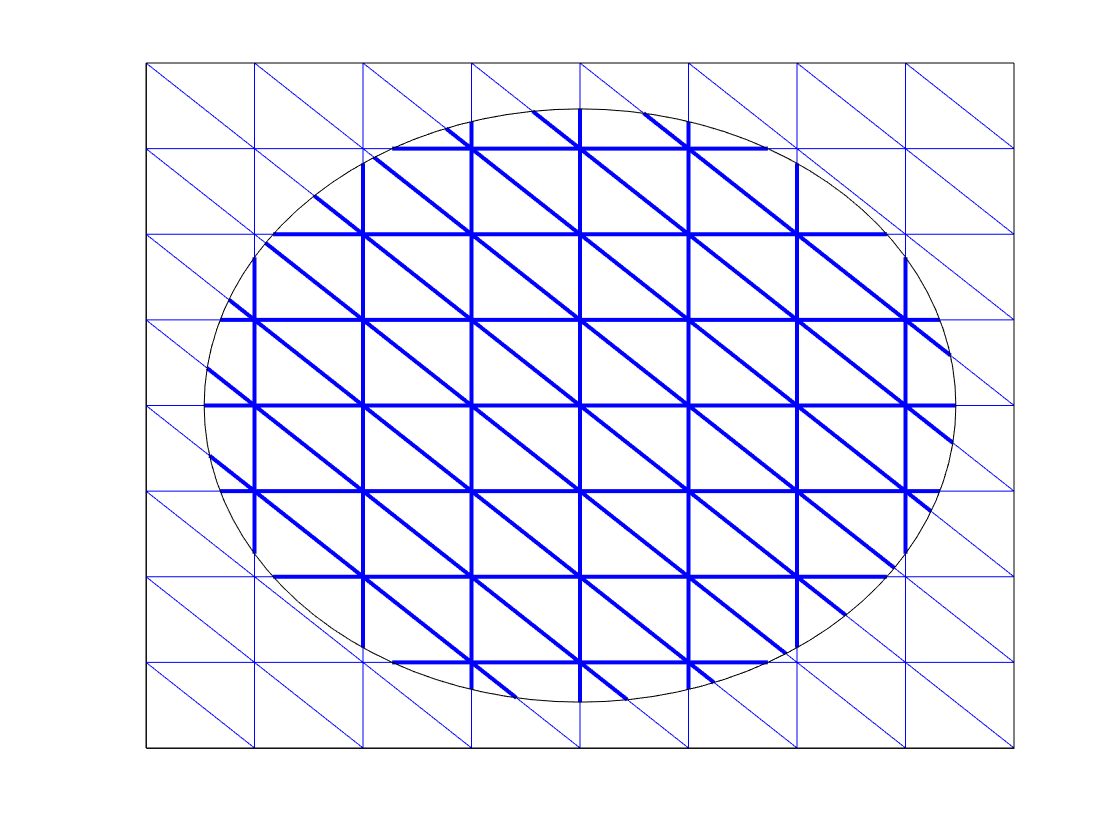} 
	\includegraphics[height = 5 cm,width= 5.5 cm]{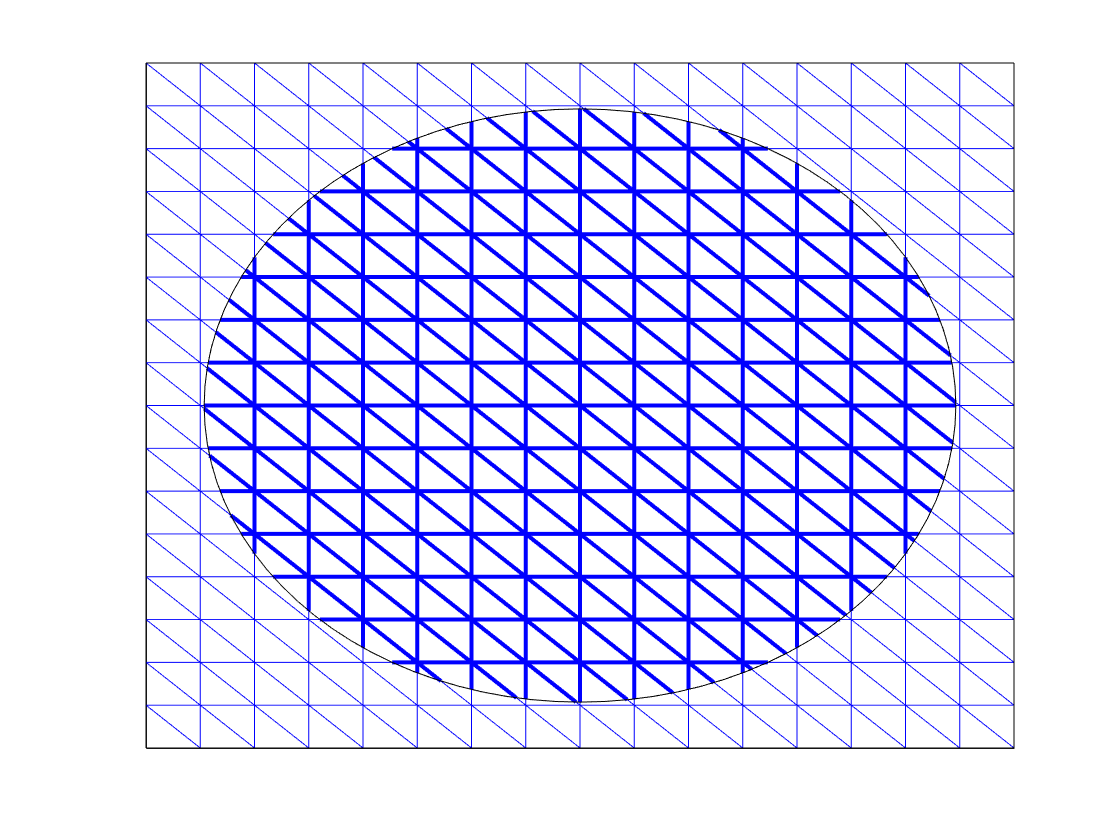} 
	\caption{The geometry of domain in  Example \ref{ex3}: $8\times 8$(left) and $16\times 16$(right) meshes }\label{domain_circle}
\end{figure}

\begin{table}[H]
	\normalsize
	\caption{History of convergence: Example \ref{ex3} }
	\label{ex_table3}
	\centering
	\footnotesize
	%		\subtable[ $k =1$ and $$ and $\alpha_1:\alpha_2=1:1000$]
	{
		\begin{tabular}{p{1.5cm}<{\centering}|p{1.15cm}<{\centering}|p{1.45cm}<{\centering}|p{0.45cm}<{\centering}|p{1.45cm}<{\centering}|p{0.45cm}<{\centering}|p{1.45cm}<{\centering}|p{0.45cm}<{\centering}|p{1.45cm}<{\centering}|p{0.45cm}<{\centering}}
			\hline   
			\multirow{3}{*}{$\nu$}&\multirow{3}{*}{mesh}&  \multicolumn{4}{c|}{$k  =1$}&\multicolumn{4}{c}{$ k = 2$} \\
			\cline{3-10}
			&&\multicolumn{2}{c|}{$\frac{\lVert \bm{u}-\bm{u}_{h}\rVert_0}{\lVert \bm{u}\rVert_0}$ }&\multicolumn{2}{c|}{$\frac{\lVert \bm{\sigma}-\bm{\sigma}_{h}\rVert_0}{\lVert \bm{\sigma}\rVert_0}$}&\multicolumn{2}{c|}{$\frac{\lVert \bm{u}-\bm{u}_{h}\rVert_0}{\lVert \bm{u}\rVert_0}$ }&\multicolumn{2}{c}{$\frac{\lVert \bm{\sigma}-\bm{\sigma}_{h}\rVert_0}{\lVert \bm{\sigma}\rVert_0}$ }\cr\cline{3-10}  
			&&error&order&error&order&error&order&error&order\cr 
			\cline{1-10}
			\multirow{4}{*}{$ 0.49$}
			&$8\times 8$	    &5.3684E-02   & -- 	&2.4740E-01	  &	  --  &4.8695E-03	  
			&   --   &3.1598E-02    &--		   \\
			%			\hline
			&$16\times 16$	 &1.4048E-02    &1.93	&1.2945E-01	  &0.93		 &6.7638E-04  &2.85        &8.4985E-03    &1.89	 \\
			%			\hline
			&$32\times 32$	&3.5215E-03    &2.00	&6.4187E-02	  &1.01   &9.2824E-05	  &2.87  &2.1532E-03    &1.98	 \\
			%			\hline
			&$64\times 64$	&8.8990E-04   &1.98	  &3.1914E-02  &1.01	 &1.2318E-05	  &2.91   &5.3899E-04    &2.00	   \\
			\hline
			
			\multirow{4}{*}{$0.4999$}
			&$8\times 8$	    &5.3551E-02   & -- 	&2.4783E-01	  &	  --  &4.8595E-03	  
			&   --    &3.1711E-02    &--		   \\
			%			\hline
			&$16\times 16$	 &1.3991E-02    &1.94	&1.2974E-01	  &0.93		 &6.7533E-04  &2.85  &8.5228E-03    &1.90 \\
			%			\hline
			&$32\times 32$	&3.5087E-03    &2.00	&6.4289E-02	  &1.01   &9.2709E-05	  &2.86   &2.1575E-03    &1.98	 \\
			%			\hline
			&$64\times 64$	&8.8752E-04   &1.98	  &3.1944E-02  &1.01	 &1.2305E-05	  &2.91    &5.3966E-04    &2.00	   \\
			\hline
			
			\multirow{4}{*}{$0.499999$}
			&$8\times 8$	    &5.3549E-02   & -- 	&2.4784E-01	  &	  --  &4.8594E-03	  
			&   --    &3.1713E-02    &--		   \\
			%			\hline
			&$16\times 16$	 &1.3991E-02    &1.94	&1.2975E-01	  &0.93		 &6.7532E-04  &2.85  &8.5231E-03    &1.90 \\
			%			\hline
			&$32\times 32$	&3.5085E-03    &2.00	&6.4290E-02	  &1.01   &9.3270E-05	  &2.86   &2.1670E-03    &1.98	 \\
			%			\hline
			&$64\times 64$	&8.8749E-04   &1.98	  &3.1944E-02  &1.01	 &1.2305E-05	  &2.92    &5.3967E-04    &2.01	   \\
			\hline
		\end{tabular}
	}
\end{table}

In    \eqref{xhdgscheme_1} we take $\mathbb{B}=[0,1]^2 $, and  use $N\times N$ uniform triangular meshes. Numerical results listed in % Errors of displacement and stress approximations with $k = 1,2$ are shown in 
Table \ref{ex_table3} for $k=1$ and $k=2$ show that our X-HDG method \eqref{xhdgscheme_1} yields  $(k+1)$-th   and  $k$-th orders   of uniform convergence  for $\lVert \bm{u}-\bm{u}_{h}\rVert_0$ and $\lVert \bm{\sigma}-\bm{\sigma}_{h}\rVert_0$, respectively.

\begin{exmp}  \label{ex4} A  test on a  non-convex  domain: inner-boundary-unfitted meshes. % (Figure \ref{domain_cutcircle}).
\end{exmp}
	
	This  example  is also used  to test the performance of  the X-HDG scheme  \eqref{xhdgscheme_1} with  boundary-unfitted meshes (cf. Figure \ref{domain_cutcircle}). 
	%This example is a plane stress test with a non-convex domain. 
	Set
	 $$\Omega = [0,1]^2\backslash \{(x,y): (x-\frac{1}{2})^2+(y-\frac{1}{2})^2<\frac{3}{64}\}.$$ 
	 The exact solution $(\bm{u},\bm{\sigma})$ of problem \eqref{pb2} is given by 
	\begin{align*}
	\bm{u}(x,y) &= \begin{pmatrix}
	y^4 \\
	x^4
	\end{pmatrix}, \quad \bm{\sigma }(x,y)=2\mu \bm{\epsilon}_h(\bm{u})+\lambda \text{div}\ \bm{u}\  \bm{ I},  %\quad   {\rm in}\  \Omega,\\
%	\sigma(x,y) &= \begin{pmatrix}
%	0,& 4\mu(x^3+y^3) \\
%	4\mu(x^3+y^3),& 0
%	\end{pmatrix},\quad  {\rm in}\  \Omega. 
	\end{align*}
	where the Lam\'e coefficients   $\mu = 1$, $\lambda = 1, 10^9$.  %The force term and boundary condition can be derived explicitly.
%\end{exmp}

 \begin{figure}[H]
	\centering
	\includegraphics[height = 5 cm,width= 5.5 cm]{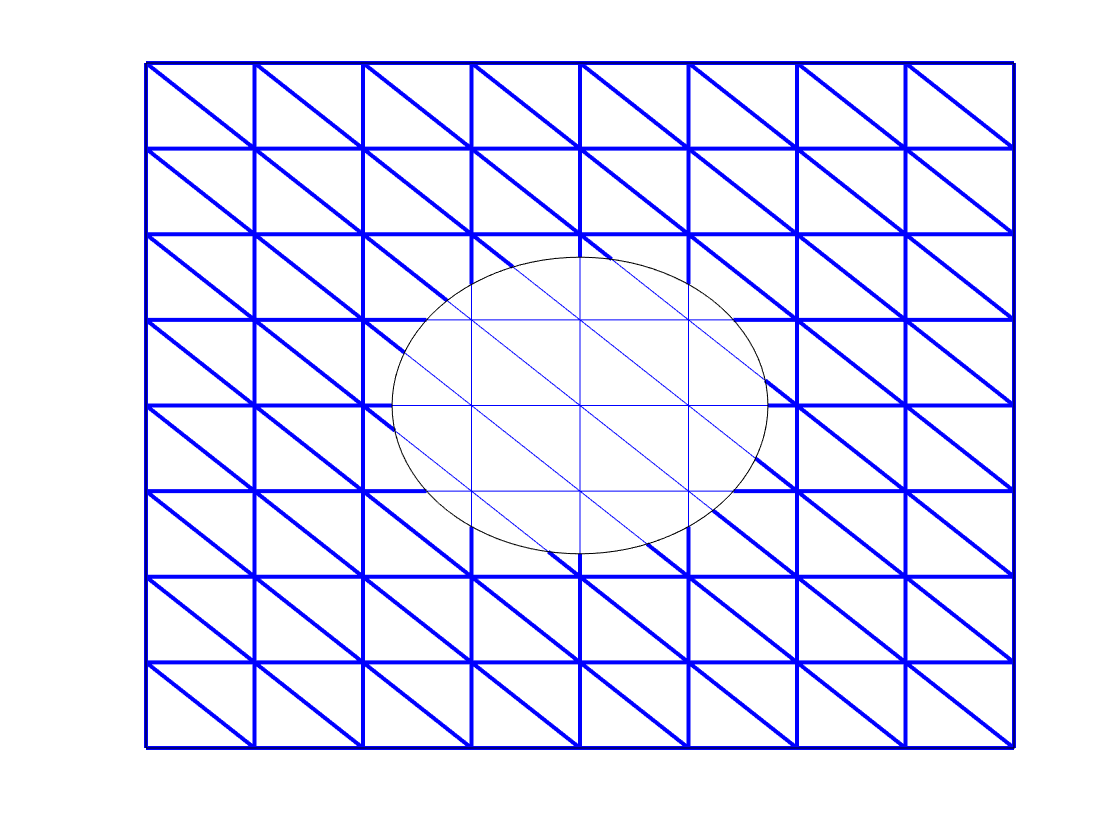} 
	\includegraphics[height = 5 cm,width= 5.5 cm]{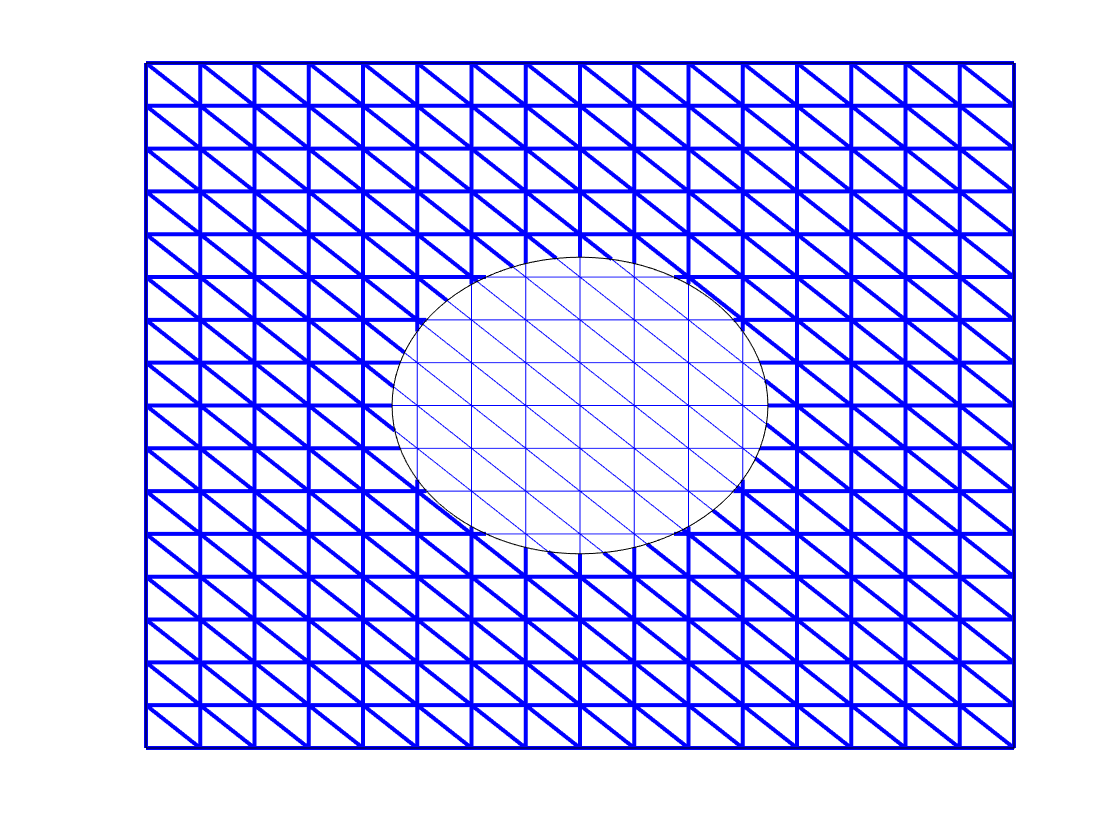} 
	\caption{The geometry of domain in Example \ref{ex4}:   $8\times 8$(left) and $16\times 16$(right) meshes}\label{domain_cutcircle}
\end{figure}

In    \eqref{xhdgscheme_1} we take $\mathbb{B}=[0,1]^2 $, and  use $N\times N$ uniform triangular meshes. Numerical results  in % Errors of displacement and stress approximations with $k = 1,2$ are shown in 
Table \ref{ex_table4} for $k=1$ and $k=2$ demonstrate  that the proposed X-HDG method is of  $(k+1)$-th   and  $k$-th orders   of  convergence  for $\lVert \bm{u}-\bm{u}_{h}\rVert_0$ and $\lVert \bm{\sigma}-\bm{\sigma}_{h}\rVert_0$, respectively.

%
%With the fact that the outer boundary is piecewise line segment, so the mesh is just fitted with the outer boundary and unfitted with inner boundary.  We use $N\times N$ uniform triangular meshes(cf. Figure \ref{domain_cutcircle}) for the computation.  Errors of displacement and stress approximations with $k = 1,2$ are shown in Table \ref{ex_table4}. We can see that our X-HDG scheme \eqref{xhdgscheme_1} yields optimal convergence rates for both displacement and stress, which are uniform with respect to $\lambda$ for all cases. 

%{\color{blue} With the fact that the outer boundary is piecewise line segment, so the mesh is just fitted with the outer boundary and unfitted with inner boundary.} We take $\mu = 1$, $\lambda$ is taken as $\lambda = 1, 10^9$. The force term can be derived explicity. Errors of displacement and stress approximations with $k = 1,2$ are shown in Table \ref{ex_table4}. We can see that our X-HDG scheme \eqref{xhdgscheme_1} yields optimal convergence rates for both displacement and stress, which are uniform with respect to $\lambda$ for all cases. 

\begin{table}[H]
	\normalsize
	\caption{History of convergence: Example \ref{ex4} }
	\label{ex_table4}
	\centering
	\footnotesize
	%		\subtable[ $k =1$ and $$ and $\alpha_1:\alpha_2=1:1000$]
	{
		\begin{tabular}{p{0.7cm}<{\centering}|p{1.15cm}<{\centering}|p{1.45cm}<{\centering}|p{0.45cm}<{\centering}|p{1.45cm}<{\centering}|p{0.45cm}<{\centering}|p{1.45cm}<{\centering}|p{0.45cm}<{\centering}|p{1.45cm}<{\centering}|p{0.45cm}<{\centering}}
			\hline   
			\multirow{3}{*}{k}&\multirow{3}{*}{mesh}&  \multicolumn{4}{c|}{$\lambda =1$}&\multicolumn{4}{c}{$\lambda =10^9$} \\
			\cline{3-10}
			&&\multicolumn{2}{c|}{$\frac{\lVert \bm{u}-\bm{u}_{h}\rVert_0}{\lVert \bm{u}\rVert_0}$ }&\multicolumn{2}{c|}{$\frac{\lVert \bm{\sigma}-\bm{\sigma}_{h}\rVert_0}{\lVert \bm{\sigma}\rVert_0}$}&\multicolumn{2}{c|}{$\frac{\lVert \bm{u}-\bm{u}_{h}\rVert_0}{\lVert \bm{u}\rVert_0}$ }&\multicolumn{2}{c}{$\frac{\lVert \bm{\sigma}-\bm{\sigma}_{h}\rVert_0}{\lVert \bm{\sigma}\rVert_0}$ }\cr\cline{3-10}  
			&&error&order&error&order&error&order&error&order\cr 
			\cline{1-10}
			\multirow{4}{*}{1}
			&$8\times 8$	    &2.3323E-02   & -- 	&8.7293E-02	  &	  --  &1.7339E-02	  
			&   --   &1.2011E-01    &--		   \\
%			\hline
			&$16\times 16$	 &6.2669E-03    &1.90	&4.1513E-02	  &1.07		 &5.1284E-03  &1.76        &5.3700E-02    &1.16	 \\
%			\hline
			&$32\times 32$	&1.6009E-03    &1.97	&2.0172E-02	  &1.04   &1.4018E-03	  &1.87  &2.2952E-02    &1.23	 \\
%			\hline
			&$64\times 64$	&4.0156E-04   &2.00	  &9.9689E-03  &1.02	 &3.6201E-04	  &1.95   &1.0511E-02    &1.13	   \\
			\hline
			
			\multirow{4}{*}{2}
			&$8\times 8$	    &1.5152E-03   & -- 	&4.9384E-03	  &	  --  &1.4237E-03	  
			&   --    &5.8621E-03    &--		   \\
			%			\hline
			&$16\times 16$	 &2.2757E-04    &2.74	&1.1363E-03	  &2.12		 &2.2009E-04  &2.69  &1.2826E-03    &2.19	 \\
			%			\hline
			&$32\times 32$	&3.0812E-05    &2.88	&2.6608E-04	  &2.09   &3.0196E-05	  &2.87   &2.8741E-04    &2.16	 \\
			%			\hline
			&$64\times 64$	&3.9847E-06   &2.95	  &6.3934E-05  &2.06	 &3.9278E-06	  &2.94    &6.7190E-05    &2.10	   \\
			\hline    
		\end{tabular}
	}
\end{table}

\begin{exmp} A plane stress test in crack-tip domain. \label{ex5}
\end{exmp}
	This example is a near crack-tip plane stress problem \cite{BarbieriA2012}.  In    \eqref{pb2} we take  $$\Omega = [0,1]^2\backslash \{(x,y_c): 0\leq x\leq x_c\}$$  
with the crack $\partial\Omega_{N}= \{(x,y_c): 0\leq x\leq x_c\}$  (cf. Figure \ref{domain_crack}), and $\partial\Omega_{D}=\partial\Omega\backslash \partial\Omega_{N}$, 
where $(x_c,y_c) = (\frac{1}{2}, \frac{1}{2})$ is the crack tip. The Lam\'e coefficients $\mu = \frac{E}{2(1+\nu)}$ and $ \lambda = \frac{E\nu}{(1+\nu)(1-\nu)}$ with  $\nu = 1/3, E = 8/3$. 
	The exact solution $(\bm{u},\bm{\sigma})$ is given by
	\begin{align*}
	\bm{u}(x,y) &= \begin{pmatrix}
	\frac{K_I}{2\mu}\sqrt{\frac{r}{2\pi}}{\rm cos}(\frac{\theta}{2})[\kappa -1 + 2{\rm sin}^2(\frac{\theta}{2})] \\
	\frac{K_I}{2\mu}\sqrt{\frac{r}{2\pi}}{\rm sin}(\frac{\theta}{2})[\kappa +1 - 2{\rm cos}^2(\frac{\theta}{2})]
	\end{pmatrix}, \\%\quad   {\rm in}\  \Omega,
	\bm{\sigma}(x,y) &= \begin{pmatrix}
	\frac{K_I}{\sqrt{2\pi r} }{\rm cos}(\frac{\theta}{2})[1 -{\rm sin}(\frac{\theta}{2}){\rm sin}(\frac{3\theta}{2})],& \frac{K_I}{\sqrt{2\pi r} }{\rm cos}(\frac{\theta}{2}) {\rm sin}(\frac{\theta}{2}){\rm cos}(\frac{3\theta}{2}) \\
	\frac{K_I}{\sqrt{2\pi r} }{\rm cos}(\frac{\theta}{2}) {\rm sin}(\frac{\theta}{2}){\rm cos}(\frac{3\theta}{2}),& \frac{K_I}{\sqrt{2\pi r} }{\rm cos}(\frac{\theta}{2})[1 + {\rm sin}(\frac{\theta}{2}){\rm sin}(\frac{3\theta}{2})] 
	\end{pmatrix},%\quad  {\rm in}\  \Omega. 
	\end{align*}
	where  $r$ is the distance from the crack tip,   $\theta = \text{arctan2}(y-y_c,x-x_c)$,   $\kappa = \frac{3-\nu}{1+\nu}$,  and the stress intensity factor (SIF) $K_I=\sqrt{\frac{\pi}{2} }$.  %According to the computation in \cite{BarbieriA2012}, the stress intensity factors (SIFs) $K_I$ is taken as $\sqrt{\frac{\pi}{2} }$. 
	We note that the boundary condition along the  line crack is a homogeneous Neumann condition, i.e.   $\bm{\sigma} \bm{n}|_{\partial\Omega_{N}} = 0$, and that $\bm{u}\notin \bm{H}^{3/2}(\Omega)$ but $\bm{u} \in  \bm{H}^{3/2-\epsilon}(\Omega)$ for any $\epsilon>0$.
	%, and others are Dirichlet boundary condition.
%\end{exmp}

 \begin{figure}[htp]
	\centering
	\includegraphics[height = 5 cm,width= 5.5 cm]{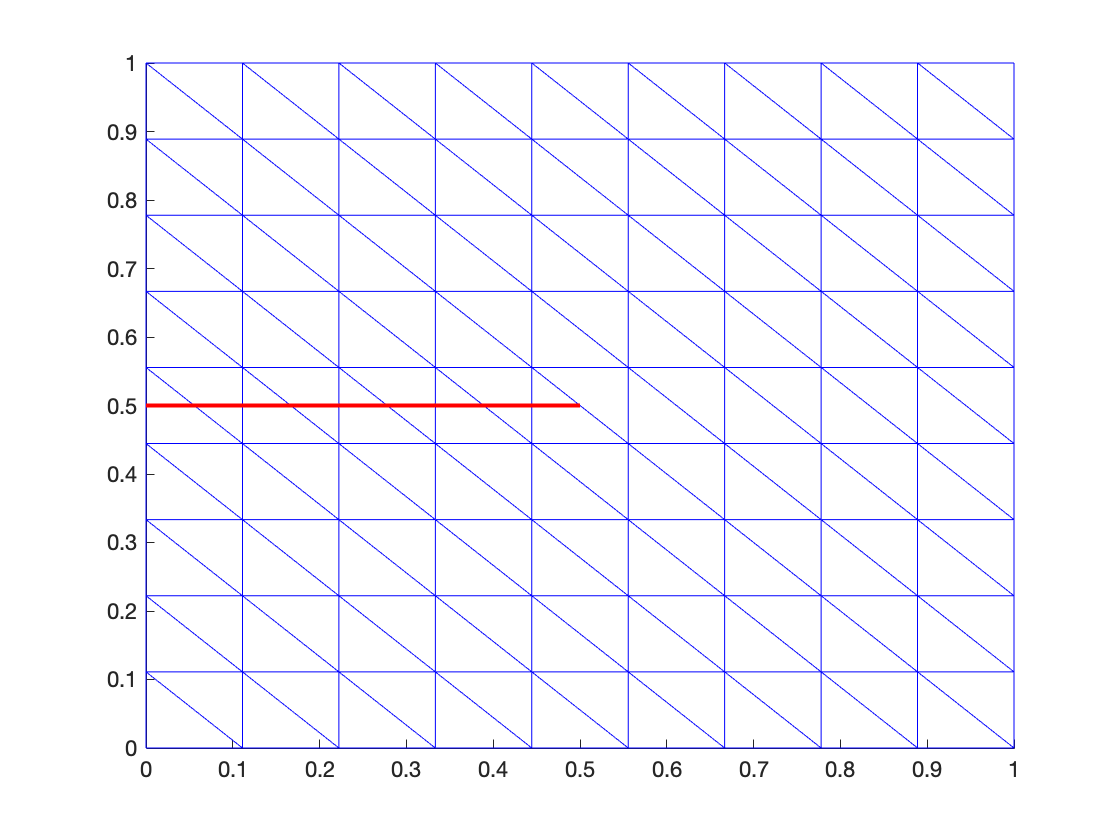} 
	\caption{The crack domain in Example \ref{ex5}:    $9\times 9$ mesh} \label{domain_crack}
\end{figure}

In  the X-HDG scheme  \eqref{xhdgscheme_1} we take $\mathbb{B}=[0,1]^2 $, and  use $N\times N$ uniform triangular meshes. Due to the low regularity of the exact solution, we only consider the lowest order  case of the scheme, i.e. $k=1$. 
%n this example, the second component of displacement is discontinuous at the line crack. 
From the numerical results in  Table \ref{ex_table5}, we can see that the convergence rate  is $0.5$ for the stress error $\lVert \bm{\sigma}-\bm{\sigma}_{h}\rVert_0$, which is as same as  that in \cite{BarbieriA2012}, and that  the convergence rate  is $1$ for the  displacement error $\lVert \bm{u}-\bm{u}_{h}\rVert_0$. The second component  of the displacement approximation  $\bm{u}_{h}$  at $129\times 129$ mesh  is also plotted   in Figure \ref{uh2crack}. 

\begin{table}[H]
	\normalsize
	\caption{History of convergence for Example \ref{ex5} }
	\label{ex_table5}
	\centering
	\footnotesize
	%		\subtable[ $k =1$ and $$ and $\alpha_1:\alpha_2=1:1000$]
	{
		\begin{tabular}{p{0.7cm}<{\centering}|p{1.3cm}<{\centering}|p{1.5cm}<{\centering}|p{0.6cm}<{\centering}|p{1.5cm}<{\centering}|p{0.6cm}<{\centering}}%|p{1.45cm}<{\centering}|p{0.45cm}<{\centering}|p{1.45cm}<{\centering}|p{0.45cm}<{\centering}}
			\hline   
%			\multirow{3}{*}{k}&\multirow{3}{*}{mesh}
%			\multirow{3}{*}{k}&\multirow{3}{*}{mesh}&  \multicolumn{4}{c|}{$\lambda =1$}&\multicolumn{4}{c}{$\lambda =10^9$} \\
%			\cline{3-10}&&
			\multirow{3}{*}{k}&\multirow{3}{*}{mesh}&\multicolumn{2}{c|}{$\frac{\lVert \bm{u}-\bm{u}_{h}\rVert_0}{\lVert \bm{u}\rVert_0}$ }&\multicolumn{2}{c}{$\frac{\lVert \bm{\sigma}-\bm{\sigma}_{h}\rVert_0}{\lVert \bm{\sigma}\rVert_0}$}\cr\cline{3-6}  
			&&error&order&error&order\cr 
%			\multirow{3}{*}{k}&\multirow{3}{*}{mesh}&\multicolumn{2}{c|}{$\frac{\lVert \bm{u}-\bm{u}_{h}\rVert_0}{\lVert \bm{u}\rVert_0}$ }&\multicolumn{2}{c|}{$\frac{\lVert \bm{\sigma}-\bm{\sigma}_{h}\rVert_0}{\lVert \bm{\sigma}\rVert_0}$}&\multicolumn{2}{c|}{$\frac{\lVert \bm{u}-\bm{u}_{h}\rVert_0}{\lVert \bm{u}\rVert_0}$ }&\multicolumn{2}{c}{$\frac{\lVert \bm{\sigma}-\bm{\sigma}_{h}\rVert_0}{\lVert \bm{\sigma}\rVert_0}$ }\cr\cline{3-10}  
%			&&error&order&error&order&error&order&error&order\cr 
			\cline{1-6}
			\multirow{5}{*}{1}
			&$9\times 9$	    &3.5241E-02   & -- 	&2.5316E-01	  &	  --  \\
%			&1.9352E-02	   &   --   &1.0436E-01    &--		   \\
			%			\hline
			&$17\times 17$	 &1.9991E-02    &0.89	&1.8320E-01	  &0.51		 \\
%			&5.7076E-03  &1.76        &4.7910E-02    &1.12	 \\
			%			\hline
			&$33\times 33$	&1.0712E-02    &0.94	&1.3107E-01	  &0.50   \\
%			&1.5326E-03	  &1.90       &2.1874E-02    &1.13	 \\
			%			\hline
			&$65\times 65$	&5.5568E-03   &0.97	  &9.3236E-02  &0.50   	 \\
%			&3.9341E-04	  &1.96      &1.0407E-02    &1.07	   \\
			%			\hline
			&$129\times 129$	&2.8322E-03   &0.98	  &6.6125E-02  &0.50   	\\ 
%			&3.9341E-04	  &1.96      &1.0407E-02    &1.07	   \\
			\hline
		\end{tabular}
	}
\end{table}

\begin{figure}[H]
	\centering
	\includegraphics[height = 5 cm,width= 5.5 cm]{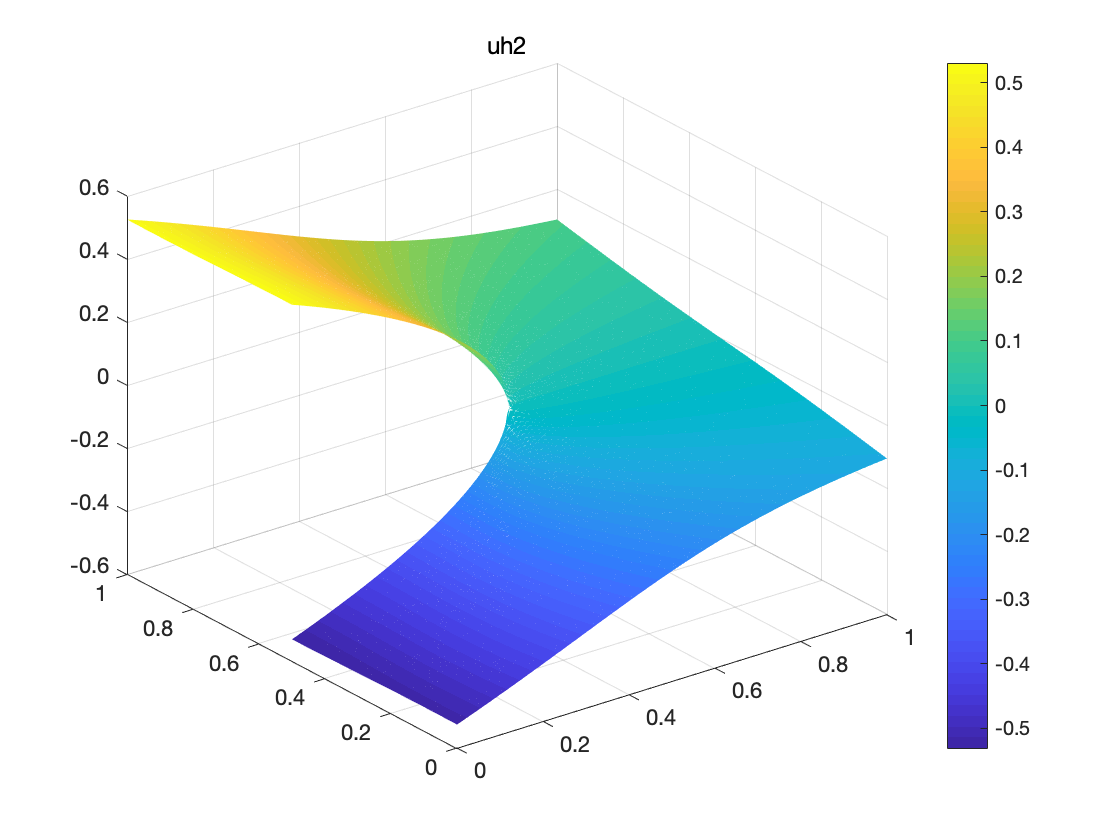} 
	\includegraphics[height = 5 cm,width= 5.5 cm]{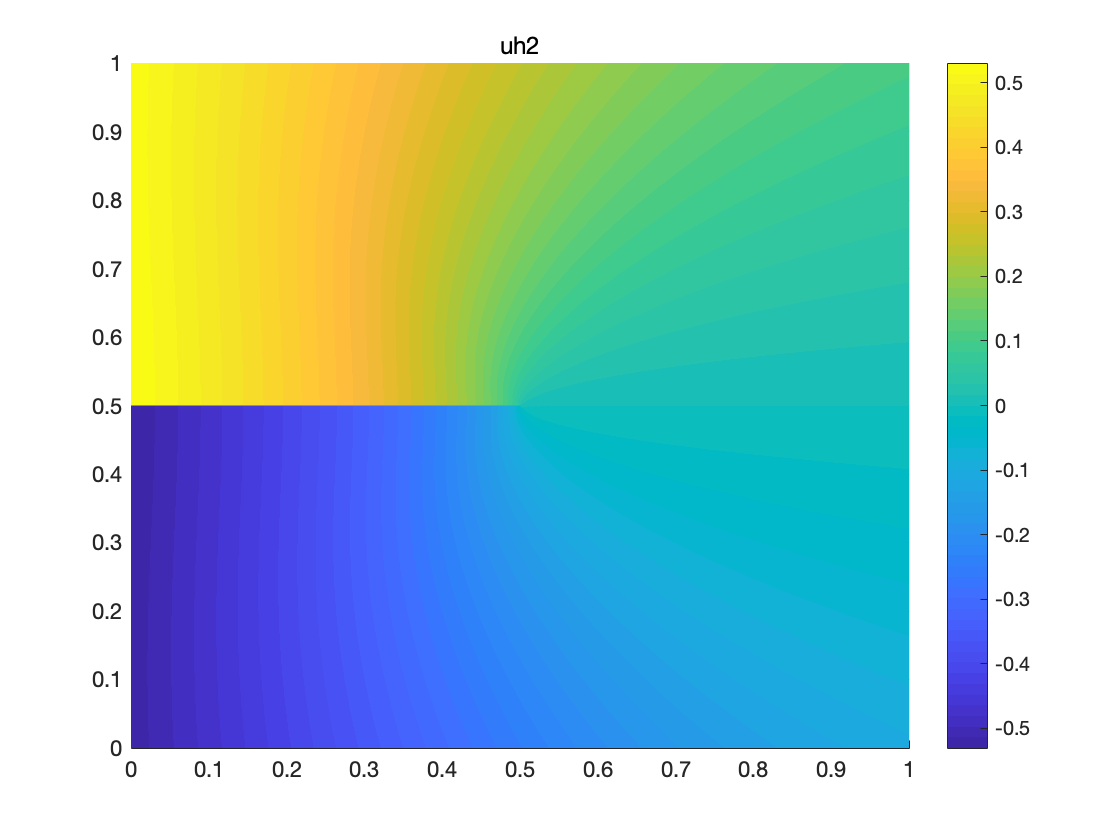} 
	\caption{The second component of displacement approximation in Example \ref{ex5}:  $129\times 129$ mesh}\label{uh2crack}
\end{figure}

\section{Concluding remarks}
In this paper, we have proposed and analyzed an arbitrary order  interface/boundary-unfitted eXtended hybridizable discontinuous Galerkin method  of optimal convergence for  linear elasticity interface problems.  %This X-HDG method is . % uniformly in  the Lam\'e constant $\lambda$. 
Numerical experiments have demonstrated the  performance and  robustness of the method.

%	\tiny
\footnotesize
%\bibliographystyle{plain}
%\bibliography{yihui_reference_papers,yihui_reference_books}

\begin{thebibliography}{10}
	
	\bibitem{Adjerid2015An}
	S.~Adjerid, N.~Chaabane, and T.~Lin.
	\newblock An immersed discontinuous finite element method for {S}tokes
	interface problems.
	\newblock {\em Computer Methods in Applied Mechanics \& Engineering},
	293:170--190, 2015.
	
	\bibitem{babuvska1970finite}
	I.~Babu{\v{s}}ka.
	\newblock The finite element method for elliptic equations with discontinuous
	coefficients.
	\newblock {\em Computing}, 5(3):207--213, 1970.
	
	\bibitem{babuska2011Stable}
	I.~Babu{\v{s}}ka and U.~Banerjee.
	\newblock Stable generalized finite element method ({SGFEM}).
	\newblock {\em Computer Methods in Applied Mechanics \& Engineering},
	201(1):91--111, 2011.
	
	\bibitem{babuvska1994special}
	I.~Babu{\v{s}}ka, G.~Caloz, and J.~E. Osborn.
	\newblock Special finite element methods for a class of second order elliptic
	problems with rough coefficients.
	\newblock {\em SIAM Journal on Numerical Analysis}, 31(4):945--981, 1994.
	
	\bibitem{BarbieriA2012}
	E.~Barbieri, N.~Petrinic, M.~Meo, and V.~L. Tagarielli.
	\newblock A new weight-function enrichment in meshless methods for multiple
	cracks in linear elasticity.
	\newblock {\em International Journal for Numerical Methods in Engineering},
	90(2):177--195, 2012.
	
	\bibitem{barrett1987fittedandunfitted}
	J.~W. Barrett and C.~M. Elliott.
	\newblock Fitted and unfitted finite element methods for elliptic equations
	with smooth interfaces.
	\newblock {\em IMA Journal of Numerical Analysis}, 7(3):283--300, 1987.
	
	\bibitem{Burman2009A}
	R.~Becker, E.~Burman, and P.~Hansbo.
	\newblock A {N}itsche extended finite element method for incompressible
	elasticity with discontinuous modulus of elasticity.
	\newblock {\em Computer Methods in Applied Mechanics \& Engineering},
	198(41):3352--3360, 2009.
	
	\bibitem{belytschko2009review}
	T.~Belytschko, R.~Gracie, and G.~Ventura.
	\newblock A review of extended/generalized finite element methods for material
	modeling.
	\newblock {\em Modelling and Simulation in Materials Science and Engineering},
	17(4):043001, 2009.
	
	\bibitem{BodartXFEM2018}
	O.~Bodart, V.~Cayol, S.~Court, and J.~Koko.
	\newblock X{FEM}-based fictitious domain method for linear elasticity model
	with crack.
	\newblock {\em SIAM Journal on Scientific Computing}, 38(2):B219--B246, 2018.
	
	\bibitem{Brambles96fin}
	J.~H. Bramble and J.~T. King.
	\newblock A finite element method for interface problems in domains with smooth
	boundaries and interfaces.
	\newblock {\em Advances in Computational Mathematics}, 6(1):109--138, 1996.
	
	\bibitem{burman2017cut}
	E.~Burman, P.~Hansbo, and M.G. Larson.
	\newblock A cut finite element method with boundary value correction for the
	incompressible stokes equations.
	\newblock pages 183--192, 2017.
	
	\bibitem{burman2018cut}
	E.~Burman, P.~Hansbo, and M.G. Larson.
	\newblock A cut finite element method with boundary value correction.
	\newblock {\em Mathematics of Computation}, 87(310):633--657, 2018.
	
	\bibitem{Cai2017Discontinuous}
	Z.~Cai, C.~He, and S.~Zhang.
	\newblock Discontinuous finite element methods for interface problems: {R}obust
	a priori and a posteriori error estimates.
	\newblock {\em SIAM Journal on Numerical Analysis}, 55(1):400--418, 2017.
	
	\bibitem{Cai2011Discontinuous}
	Z.~Cai, X.~Ye, and S.~Zhang.
	\newblock Discontinuous galerkin finite element methods for interface problems:
	A priori and a posteriori error estimations.
	\newblock {\em SIAM Journal on Numerical Analysis}, 49(5):1761--1787, 2011.
	
	\bibitem{chen-xie2016}
	G.~Chen and X.~Xie.
	\newblock A robust weak {G}alerkin finite element method for linear elasticity
	with strong symmetric stresses.
	\newblock {\em Computational Methods in Applied Mathematics}, 16(3):389--408,
	2016.
	
	\bibitem{chen2015robust}
	H.~Chen, J.~Li, and W.~Qiu.
	\newblock Robust a posteriori error estimates for {HDG} method for
	convection--diffusion equations.
	\newblock {\em IMA Journal of Numerical Analysis}, 36(1):437--462, 2015.
	
	\bibitem{chen2014robust}
	H.~Chen, P.~Lu, and X.~Xu.
	\newblock A robust multilevel method for hybridizable discontinuous {G}alerkin
	method for the {H}elmholtz equation.
	\newblock {\em Journal of Computational Physics}, 264:133--151, 2014.
	
	\bibitem{Chen1998Finite}
	Z.~Chen and J.~Zou.
	\newblock Finite element methods and their convergence for elliptic and
	parabolic interface problems.
	\newblock {\em Numerische Mathematik}, 79(2):175--202, 1998.
	
	\bibitem{Cockburn2009Unified}
	B.~Cockburn, J.~Gopalakrishnan, and R.~Lazarov.
	\newblock Unified hybridization of discontinuous {G}alerkin, mixed, and
	continuous {G}alerkin methods for second order elliptic problems.
	\newblock {\em SIAM Journal on Numerical Analysis}, 47(2):1319--1365, 2009.
	
	\bibitem{Cockburn2011HDGstokes}
	B.~Cockburn, J.~Gopalakrishnan, and N.~C. Nguyen.
	\newblock Analysis of {HDG} methods for {S}tokes flow.
	\newblock {\em Mathematics of Computation}, 80(274):723--760, 2011.
	
	\bibitem{cockburn2010comparison}
	B.~Cockburn, N.~C. Nguyen, and J.~Peraire.
	\newblock A comparison of {HDG} methods for {S}tokes flow.
	\newblock {\em Journal of Scientific Computing}, 45(1):215--237, 2010.
	
	\bibitem{cockburn2014priori}
	B.~Cockburn, W.~Qiu, and M.~Solano.
	\newblock A priori error analysis for {HDG} methods using extensions from
	subdomains to achieve boundary conformity.
	\newblock {\em Mathematics of Computation}, 83(286):665--699, 2014.
	
	\bibitem{cockburn2014divergence}
	B.~Cockburn and F-J. Sayas.
	\newblock Divergence-conforming {HDG} methods for {S}tokes flows.
	\newblock {\em Mathematics of Computation}, 83(288):1571--1598, 2014.
	
	\bibitem{Cockburn2013HDGelasticityweaklystresses}
	B.~Cockburn and K.~Shi.
	\newblock Superconvergent {HDG} methods for linear elasticity with weakly
	symmetric stresses.
	\newblock {\em IMA Journal of Numerical Analysis}, 33(3):747--770, 2013.
	
	\bibitem{cockburn2012solving}
	B.~Cockburn and M.~Solano.
	\newblock Solving dirichlet boundary-value problems on curved domains by
	extensions from subdomains.
	\newblock {\em SIAM Journal on Scientific Computing}, 34(1):A497--A519, 2012.
	
	\bibitem{cockburn2014solving}
	B.~Cockburn and M.~Solano.
	\newblock Solving convection-diffusion problems on curved domains by extensions
	from subdomains.
	\newblock {\em Journal of Scientific Computing}, 59(2):512--543, 2014.
	
	\bibitem{Dong2016An}
	H.~Dong, B.~Wang, Z.~Xie, and L.~L. Wang.
	\newblock An unfitted hybridizable discontinuous {G}alerkin method for the
	poisson interface problem and its error analysis.
	\newblock {\em IMA Journal of Numerical Analysis}, 37(1):444--476, 2016.
	
	\bibitem{DuflotThe2008}
	M.~Duflot.
	\newblock The extended finite element method in thermoelastic fracture
	mechanics.
	\newblock {\em International Journal for Numerical Methods in Engineering},
	74(5):827--847, 2008.
	
	\bibitem{gao2001continuum}
	H.~Gao, Y.~Huang, and F.~F. Abraham.
	\newblock Continuum and atomistic studies of intersonic crack propagation.
	\newblock {\em Journal of the Mechanics and Physics of Solids},
	49(9):2113--2132, 2001.
	
	\bibitem{gibiansky2000multiphase}
	L.~V. Gibiansky and O.~Sigmund.
	\newblock Multiphase composites with extremal bulk modulus.
	\newblock {\em Journal of the Mechanics and Physics of Solids}, 48(3):461--498,
	2000.
	
	\bibitem{GongYan2010Immersed}
	Y.~Gong, Z.~Li, and D.~Gaffney.
	\newblock Immersed interface finite element methods for elasticity interface
	problems with non-homogeneous jump conditions.
	\newblock {\em Numerical Mathematics: Theory,Methods and Applications},
	46(1):472--495, 2010.
	
	\bibitem{guzman2018inf}
	J.~Guzm{\'a}n and M.~Olshanskii.
	\newblock Inf-sup stability of geometrically unfitted {S}tokes finite elements.
	\newblock {\em Mathematics of Computation}, 87(313):2091--2112, 2018.
	
	\bibitem{G2016eXtended1}
	C.~Gürkan, M.~Kronbichler, and S.~Fernández-Méndez.
	\newblock e{X}tended hybridizable discontinuous {G}alerkin with heaviside
	enrichment for heat bimaterial problems.
	\newblock {\em Journal of Scientific Computing}, 72(2):542--567, 2017.
	
	\bibitem{G2016eXtendedvoid}
	C.~Gürkan, E.~Sala-Lardies, M.~Kronbichler, and S.~Fernández-Méndez.
	\newblock e{X}tended hybridizable discontinous {G}alerkin ({X-HDG}) for void
	problems.
	\newblock {\em Journal of Scientific Computing}, 66(3):1313--1333, 2016.
	
	\bibitem{HanChenWangXie2019Extended}
	Y.~Han, H.~Chen, X.~Wang, and X.~Xie.
	\newblock E{X}tended {HDG} methods for second order elliptic interface
	problems.
	\newblock {\em arXiv preprint arXiv:1910.09769}, 2019.
	
	\bibitem{hansbo2002unfitted}
	A.~Hansbo and P.~Hansbo.
	\newblock An unfitted finite element method, based on {N}itsche’s method, for
	elliptic interface problems.
	\newblock {\em Computer Methods in Applied Mechanics \& Engineering},
	191(47-48):5537--5552, 2002.
	
	\bibitem{Hansbo2004elasticity}
	A.~Hansbo and P.~Hansbo.
	\newblock A finite element method for the simulation of strong and weak
	discontinuities in solid mechanics.
	\newblock {\em Computer Methods in Applied Mechanics \& Engineering},
	193(33):3523--3540, 2004.
	
	\bibitem{Huang2002Some}
	J.~Huang and J.~Zou.
	\newblock Some new a priori estimates for second-order elliptic and parabolic
	interface problems.
	\newblock {\em Journal of Differential Equations}, 184(2):570--586, 2002.
	
	\bibitem{Jou1997Microstructural}
	H.~J. Jou, P.~H. Leo, and J.~S. Lowengrub.
	\newblock Microstructural evolution in inhomogeneous elastic media.
	\newblock {\em Journal of Computational Physics}, 131(1):109--148, 1997.
	
	\bibitem{Lehrenfeld2016Optimal}
	C.~Lehrenfeld and A.~Reusken.
	\newblock Optimal preconditioners for {N}itsche-{XFEM} discretizations of
	interface problems.
	\newblock {\em Numerische Mathematik}, 135(2):1--20, 2016.
	
	\bibitem{leo2000microstructural}
	P.~H. Leo, J.~S. Lowengrub, and Q.~Nie.
	\newblock Microstructural evolution in orthotropic elastic media.
	\newblock {\em Journal of Computational Physics}, 157(1):44--88, 2000.
	
	\bibitem{leveque1994immersed}
	R.~J Leveque and Z.~Li.
	\newblock The immersed interface method for elliptic equations with
	discontinuous coefficients and singular sources.
	\newblock {\em SIAM Journal on Numerical Analysis}, 31(4):1019--1044, 1994.
	
	\bibitem{Li-X2016analysis}
	B.~Li and X.~Xie.
	\newblock Analysis of a family of {HDG} methods for second order elliptic
	problems.
	\newblock {\em Journal of Computational and Applied Mathematics}, 307:37--51,
	2016.
	
	\bibitem{Li-X2016SIAM}
	B.~Li and X.~Xie.
	\newblock {BPX} preconditioner for nonstandard finite element methods for
	diffusion problems.
	\newblock {\em SIAM Journal on Numerical Analysis}, 54(2):1147--1168, 2016.
	
	\bibitem{Li-X-Z2016analysis}
	B.~Li, X.~Xie, and S.~Zhang.
	\newblock Analysis of a two-level algorithm for {HDG} methods for diffusion
	problems.
	\newblock {\em Communications in Computational Physics}, 19(5):1435--1460,
	2016.
	
	\bibitem{Li2018A}
	H.~Li, J.~Li, and H.~Yuan.
	\newblock A review of the extended finite element method on macrocrack and
	microcrack growth simulations.
	\newblock {\em Theoretical and Applied Fracture Mechanics}, 2018.
	
	\bibitem{Li2010Optimal}
	J.~Li, M.~J. Markus, B.~I. Wohlmuth, and J.~Zou.
	\newblock Optimal a priori estimates for higher order finite elements for
	elliptic interface problems.
	\newblock {\em Applied Numerical Mathematics}, 60(1):19--37, 2010.
	
	\bibitem{lizhilin1998immersed}
	Z.~Li.
	\newblock The immersed interface method using a finite element formulation.
	\newblock {\em Applied Numerical Mathematics}, 27(3):253--267, 1998.
	
	\bibitem{lizhilin2006immersed}
	Z.~Li and K.~Ito.
	\newblock {\em The immersed interface method: numerical solutions of PDEs
		involving interfaces and irregular domains}, volume~33.
	\newblock SIAM, 2006.
	
	\bibitem{lin2015partially}
	T.~Lin, Y.~Lin, and X.~Zhang.
	\newblock Partially penalized immersed finite element methods for elliptic
	interface problems.
	\newblock {\em SIAM Journal on Numerical Analysis}, 53(2):1121--1144, 2015.
	
	\bibitem{LinTao2013A}
	T.~Lin, D.~Sheen, and X.~Zhang.
	\newblock A locking-free immersed finite element method for planar elasticity
	interface problems.
	\newblock {\em Journal of Computational Physics}, 247(16):228--247, 2013.
	
	\bibitem{lintao2012linearifem}
	T.~Lin and X.~Zhang.
	\newblock Linear and bilinear immersed finite elements for planar elasticity
	interface problems.
	\newblock {\em Journal of Computational and Applied Mathematics},
	236(18):4681--4699, 2012.
	
	\bibitem{moes1999finite}
	N.~Mo{\"e}s, J.~Dolbow, and T.~Belytschko.
	\newblock A finite element method for crack growth without remeshing.
	\newblock {\em International Journal for Numerical Methods in Engineering},
	46(1):131--150, 1999.
	
	\bibitem{persson2003effect}
	B.~N.~J. Persson and S.~Gorb.
	\newblock The effect of surface roughness on the adhesion of elastic plates
	with application to biological systems.
	\newblock {\em Journal of Chemical Physics}, 119(21):11437--11444, 2003.
	
	\bibitem{Plum2003Optimal}
	M.~Plum and C.~Wieners.
	\newblock Optimal a priori estimates for interface problems.
	\newblock {\em Numerische Mathematik}, 95(4):735--759, 2003.
	
	\bibitem{Qin2016A}
	F.~Qin, J.~Chen, Z.~Li, and M.~Cai.
	\newblock A cartesian grid nonconforming immersed finite element method for
	planar elasticity interface problems.
	\newblock {\em Computers \& Mathematics with Applications}, 73(3):404--418,
	2016.
	
	\bibitem{Qiu2016An}
	W.~Qiu, J.~Shen, and K.~Shi.
	\newblock An {HDG} method for linear elasticity with strong symmetric stresses.
	\newblock {\em Mathematics of Computation}, 87(309):69--93, 2016.
	
	\bibitem{ShenAn2010}
	Y.~Shen and A.~Lew.
	\newblock An optimally convergent discontinuous {G}alerkin-based extended
	finite element method for fracture mechanics.
	\newblock {\em International Journal for Numerical Methods in Engineering},
	82(6):716--755, 2010.
	
	\bibitem{ShenStability2010}
	Y.~Shen and A.~Lew.
	\newblock Stability and convergence proofs for a discontinuous-{G}alerkin-based
	extended finite element method for fracture mechanics.
	\newblock {\em Computer Methods in Applied Mechanics \& Engineering},
	199(37-40):2360--2382, 2010.
	
	\bibitem{sigmund2001design}
	O.~Sigmund.
	\newblock Design of multiphysics actuators using topology optimization--{P}art
	{II}: {T}wo-material structures.
	\newblock {\em Computer Methods in Applied Mechanics \& Engineering},
	190(49-50):6605--6627, 2001.
	
	\bibitem{solano2019high}
	M.~Solano and F.~Vargas.
	\newblock A high order {HDG} method for {S}tokes flow in curved domains.
	\newblock {\em Journal of Scientific Computing}, 79(3):1505--1533, 2019.
	
	\bibitem{strouboulis2000design}
	T.~Strouboulis, I.~Babu{\v{s}}ka, and K.~Copps.
	\newblock The design and analysis of the generalized finite element method.
	\newblock {\em Computer Methods in Applied Mechanics \& Engineering},
	181(1-3):43--69, 2000.
	
	\bibitem{sutton1995interfaces}
	A.~P. Sutton and R.~W. Balluffi.
	\newblock Interfaces in crystalline materials.
	\newblock 1995.
	
	\bibitem{Thomas2010The}
	P.~F. Thomas and B.~Ted.
	\newblock The e{X}tended/{G}eneralized finite element method: {A}n overview of
	the method and its applications.
	\newblock {\em International Journal for Numerical Methods in Engineering},
	84(3):253--304, 2010.
	
	\bibitem{Wang2019}
	T.~Wang, C.~Yang, and X.~Xie.
	\newblock Extended finite element methods for optimal control problems governed
	by poisson equation in non-convex domains.
	\newblock {\em Science China Mathematics}, 61, 2019.
	
	\bibitem{wu2010unfitted}
	H.~Wu and Y.~Xiao.
	\newblock An unfitted $ hp $-interface penalty finite element method for
	elliptic interface problems.
	\newblock {\em arXiv preprint arXiv:1007.2893}, 2010.
	
	\bibitem{xiao2019an}
	H.~Wu and Y.~Xiao.
	\newblock An unfitted hp-interface penalty finite element method for elliptic
	interface problems.
	\newblock 37(3):316--339, 2019.
	
	\bibitem{xu2013estimate}
	J.~Xu.
	\newblock Estimate of the convergence rate of finite element solutions to
	elliptic equations of second order with discontinuous coefficients.
	\newblock {\em arXiv preprint arXiv:1311.4178}, 2013.
	
	\bibitem{zhang2004immersed}
	L.~Zhang, A.~Gerstenberger, X.~Wang, and W.~K. Liu.
	\newblock Immersed finite element method.
	\newblock {\em Computer Methods in Applied Mechanics \& Engineering},
	193(21-22):2051--2067, 2004.
	
\end{thebibliography}

\end{document}